\documentclass[11pt,twoside]{article}
\usepackage[hmargin=0.8in,vmargin=0.9in]{geometry}
\usepackage{fancyhdr}
\usepackage{graphicx}
\usepackage{subfigure}
\usepackage{amssymb}
\usepackage{stmaryrd}
\usepackage{amsmath}
\usepackage{amsfonts}
\usepackage{theorem}
\usepackage{mathrsfs}
\usepackage{mathtools}
\usepackage{bm}
\usepackage{color}
\usepackage{setspace}
\usepackage{exscale}
\usepackage{relsize}
\usepackage{algorithm}
\usepackage{algpseudocode}

\usepackage{epstopdf}
\usepackage{float}
\usepackage{cite}

\usepackage{color}

\usepackage{booktabs,multirow} 
\usepackage{array} 
\usepackage{paralist} 
\usepackage{verbatim} 
\usepackage{subfigure} 

\theoremstyle{plain}			
\newtheorem{thm}{Theorem}[section]

\newtheorem{prop}[thm]{Proposition}

\newtheorem{defn}[thm]{Definition}

{\theorembodyfont{\rmfamily}}

\newenvironment{DA}{{\flushleft \bf Declarations:}}{}
\newtheorem{rmk}{Remark}[section]

\allowdisplaybreaks[1]

\numberwithin{equation}{section}
\numberwithin{figure}{section}
\numberwithin{table}{section}

\newcommand\eref[1]{(\ref{#1})}

\newcommand*\xbar[1]{%
  \hbox{%
    \vbox{%
      \hrule height 0.5pt 
      \kern0.4ex
      \hbox{%
        \kern-0.05em
        \ensuremath{#1}%
        \kern-0.00em
      }%
    }%
  }%
}

\newcommand{\mF}{\bm{F}}

\newcommand{\mS}{\bm{S}}

\newcommand{\mU}{\bm{U}}

\newcommand{\dt}{\Delta t}
\newcommand{\dx}{\Delta x}
\newcommand{\dy}{\Delta y}

\newcommand{\dxi}{\Delta \xi}

\newcommand{\jph}{{j+\frac{1}{2}}}
\newcommand{\jmh}{{j-\frac{1}{2}}}

\newcommand{\kph}{{k+\frac{1}{2}}}
\newcommand{\kmh}{{k-\frac{1}{2}}}

\def\softd{{\leavevmode\setbox1=\hbox{d}%
		\hbox to 1.05\wd1{d\kern-0.4ex{\char039}\hss}}}

\title{A Bayesian Approach to Feedback Control for Hyperbolic Balance Laws}
\author{Markus Bambach\thanks{Department of Mechanical and Process Engineering, ETH Z\"urich, Zürich, 8005, Switzerland; {\tt mbambach@ethz.ch}}, ~Shaoshuai Chu\thanks{Department of Mathematics, RWTH Aachen University, Aachen, 52056, Germany; {\tt chu@igpm.rwth-aachen.de}},
~Michael Herty\thanks{Department of Mathematics, RWTH Aachen University, Aachen, 52056, Germany; Department of Mathematics and Applied
Mathematics, University of Pretoria, Private Bag X20, Hatfield 0028, South Africa; {\tt herty@igpm.rwth-aachen.de}}, ~Yunong Lin\thanks{Department of Mechanical and Process Engineering, ETH Z\"urich, Zürich, 8092, Switzerland; {\tt yunlin@mavt.ethz.ch}} }

\begin{document}
\date{}
\maketitle

\begin{abstract}

We propose a Bayesian framework for feedback boundary control of hyperbolic balance laws. The method propagates a probability distribution over feedback parameters using Lyapunov decay estimates as a likelihood. For linear models, it recovers available analytical stability results and extends to nonlinear regimes where theory is limited. Using first-order local Lax–Friedrichs (LLF) discretizations, we validate the approach on the decoupled wave system and the linearized Saint-Venant equations, reproducing known stability intervals and mixed boundary couplings. We then treat nonlinear and stochastic problems, including the nonlinear Saint-Venant system, one- and two-dimensional Burgers equations, Burgers equation with random initial data, and nonconservative perturbations with source terms, and show that the inferred stability domains are robust with respect to the indicator and the prior. Finally, we demonstrate transfer to a second-order semi-discrete LLF scheme and to a two-parameter feedback model for laser powder bed fusion with power regulation.
\end{abstract}

\smallskip
\noindent
{\bf Key words:} Hyperbolic systems of balance laws; exponential stability; boundary feedback control.

\medskip
\noindent
{\bf AMS subject classification:} 93C20, 93D15, 93D05, 35L65, 65M12, 65M08.

\section{Introduction}
This paper focuses on the stabilization of hyperbolic systems, which in the one-dimensional (1-D) case read as
\begin{equation}\label{1.1}
  \mU_t + \mF(\mU)_x= \mS(\mU),
\end{equation}
where $x\in \Omega$ is a spatial variable, $t \ge 0$ is the time, $\mU\in\mathbb R^d$ is a vector of unknown functions, $\mF$: $\mathbb R^d\to\mathbb R^d$ denotes the (generally nonlinear) flux and $\bm S:\mathbb R^d \to \mathbb R^d$  the source term.

We consider the stabilization of the dynamics near a stationary state of the system~\eqref{1.1} by means of boundary feedback control. This topic has attracted considerable attention due to applications in shallow-water systems, gas flow, traffic flow, and many other areas; see, e.g.,~\cite{Bastin2016,coron2007control}. On the analytical side, a substantial body of work has been devoted to the design of stabilizing boundary conditions and feedback laws for hyperbolic systems, typically based on Lyapunov or energy methods that yield exponential decay estimates and admissible ranges of feedback parameters; see, for instance, \cite{bastin2007lyapunov,coron2007control,coron2008dissipative,coron2007strict}. Further developments, including the treatment of source terms, nonlinear couplings, and networked configurations, can be found in \cite{HY2016,li2005exact,li2009exact,li2008controllabilite,li2003exact,xu2002exponential}. These continuous results are complemented by analytical and numerical analyses, as well as
discretizations that are required to inherit the stabilization properties at the semi-discrete or fully discrete level; see, e.g.,\cite{Bardos1979,banda2013numerical}. Many of the above contributions develop Lyapunov-type methodologies both for the continuous system~\eqref{1.1} and, in some cases, for its numerical approximations, thereby providing explicit decay estimates and conditions on the feedback parameters.

Using a suitable Lyapunov function, the exponential stabilization inequality \eref{2.2a} has been established for linear deterministic hyperbolic systems. The analysis of the Lyapunov function yields explicit estimates for the feedback control parameter $\kappa$; see equation~\eref{2.1}. In the linear case, these estimates take the form of dissipativity conditions on $\kappa$ that guarantee exponential stability; see, e.g.,~\cite{DDTK2017,coron2008dissipative}. In the discrete setting, carefully constructed Lyapunov functionals lead to stability proofs and explicit decay rates consistent with the continuous theory; see \cite{banda2013numerical,Bastin2016,Meurer2013,CHK_Stability,DBC2012}. The Saint--Venant system for shallow-water flows is a canonical example, in which source terms and variable topography are intrinsic. At the continuous level, dissipative boundary conditions for one-dimensional nonlinear systems were established in \cite{coron2008dissipative}, and coupled models such as the Saint--Venant--Exner system were treated via backstepping in \cite{DDTK2017}. On the numerical side, boundary feedback for non-uniform balance laws was proposed in \cite{BW2020}.

Although appropriate boundary controls can be derived for specific systems, most of them concern linearized hyperbolic balance laws. Determining stabilizing controls in more complex settings, such as fully nonlinear, nonconservative, or multidimensional problems, remains highly challenging. To address this difficulty, we propose a Bayesian framework. Instead of working with a single feedback parameter, we numerically evolve a probability density on the feedback parameters; see equation~\eref{2.1}. Starting from a prior distribution, we update this density using Bayes' theorem. As a likelihood function, we employ the observed decay rate of an $L^2$-based Lyapunov functional; see~\eref{2.2a}. The resulting posterior distribution identifies the range of stabilizing feedback parameters. Since this approach requires repeated simulations of the underlying hyperbolic balance law, we present the Bayesian framework directly in the discretized setting below. The approach is non-intrusive: it can be applied to a given scheme and does not require the design of alternative Lyapunov functionals. We first consider linear deterministic models (linear advection/wave and the linearized Saint--Venant system) and verify that the recovered stability domains agree with the theoretical predictions. We then treat nonlinear problems (nonlinear Saint--Venant and Burgers), followed by stochastic variants, a nonconservative case, and a second-order local Lax-Friedrichs (LLF) discretization; see, e.g.,\cite{KTcl,Rus61}. Finally, we examine the evolution of the thermal field during a single-track scanning process in laser powder bed fusion and recover its stability region via the same Bayesian update. These experiments demonstrate both the accuracy and robustness of the proposed method and indicate its potential for more complex systems.

The paper is organized as follows. In \S\ref{sec2}, we introduce the Bayesian boundary feedback framework, including the first-order LLF discretization and the probability-update algorithm with damping and normalization. In \S\ref{sec3}, we present a comprehensive suite of experiments: linear deterministic tests (wave and linearized Saint-Venant), nonlinear benchmarks (nonlinear Saint-Venant and Burgers), stochastic variants, a nonconservative case, a second-order semi-discrete LLF discretization, and the evolution of the thermal field during a single-track scanning process in laser powder bed fusion, all assessed within the proposed Bayesian boundary feedback framework. Finally, \S\ref{sec4} contains concluding remarks and perspectives.

\section{A Bayesian Boundary Feedback Method}\label{sec2}
In this section, we introduce the Bayesian feedback control method for the studied hyperbolic system \eref{1.1} with the specific boundary conditions. 

\subsection{First-Order Local Lax-Friedrichs Scheme}\label{sec2.1}
Assume the computational domain is covered with uniform cells $C_j:=[x_\jmh,x_\jph]$ with $x_\jph-x_\jmh\equiv\dx$ centered at $x_j=(x_\jmh+x_\jph)/2$, $\,j=1,\ldots,N_x$, and the cell average values

\begin{equation*}
  \xbar\mU_j(t):\approx\frac{1}{\dx}\int\limits_{C_j}\mU(x,t)\,{\rm d}x
\end{equation*}
are available at a certain time level $t\ge0$.

The computed cell averages $\xbar \mU_j$ of the 1-D system \eref{1.1} are evolved in time by the first-order LLF scheme:
\begin{equation*}
  \xbar \mU^{n+1}_j=  \xbar \mU^{n}_j - \dfrac{\dt}{\dx}\Big( \bm{{\cal F}}_\jph - \bm{{\cal F}}_\jmh \Big)+\dt \,\mS(\xbar \mU^n_j),
\end{equation*}
where $\bm{{\cal F}}_\jph$ stands for finite-volume numerical flux and in this paper, we use the central (local Lax-Friedrichs, Rusanov; see, e.g., \cite{KTcl,Rus61}) fluxes 
\begin{equation*}
  \bm{{\cal F}}_\jph=\frac{1}{2} \big( \mF (\xbar \mU_j) +\mF(\xbar \mU_{j+1}) \big)-\frac{\alpha_\jph}{2}\big(\xbar \mU_{j+1}-\xbar \mU_j \big), 
\end{equation*}
with the local speeds of propagation estimated using the eigenvalues $\lambda_1({\cal A}) \le \ldots \le \lambda_d({\cal A})$ of the matrix ${\cal A}=\frac{\partial \mF}{\partial \mU}$: 
\begin{equation*}
\alpha_\jph = \max \big\{|\lambda_1 (\xbar \mU_j)|, |\lambda_1 (\xbar \mU_{j+1})|, \ldots, |\lambda_d (\xbar \mU_j)|, |\lambda_d (\xbar \mU_{j+1})|      \big\}. 
\end{equation*}
Finally, the time step $\dt$ is computed by the CFL condition, $\dt = {\rm CFL} \frac{\dx}{ \max\limits_j \{\alpha_\jph\} }$.

\begin{rmk}
It is straightforward to verify that, for linear systems, the first-order LLF scheme coincides with the corresponding upwind scheme.
\end{rmk}

\subsection{Evolution of the Probability Distribution}
For simplicity, we work on the one-dimensional spatial domain $[0,1]$. We assume that, in a neighborhood of the stationary state of interest, the incoming and outgoing characteristic fields are well defined. More precisely, let $\Lambda_i(A)$, $i=1,\ldots,d$, denote the eigenvalues of the Jacobian matrix $A=\frac{\partial \mF}{\partial \mU}$ at the stationary state, and assume that
$$
\Lambda_i(A)>0,\quad i=1,\ldots,m,\qquad \Lambda_i(A)<0,\quad i=m+1,\ldots,d.
$$
Then the first $m$ characteristic fields enter the domain at $x=0$ and leave at $x=1$, whereas the remaining $d-m$ characteristic fields enter the domain at $x=1$ and leave at $x=0$. This assumption is made only to simplify the presentation of the boundary feedback law; the Bayesian update procedure itself only requires an admissible feedback parametrization and a nonnegative stability indicator. We therefore impose boundary conditions at discrete time levels $t=t^n$ in the form
\begin{equation}\label{2.1}
\xbar U^{\,(p)}_0(t^n) = \kappa^{\,(p)}\,\xbar U^{\,(p)}_{N_x}(t^n), \quad p=1,\ldots,m, 
\end{equation}
and 
\begin{equation}\label{2.2}
\xbar U^{\,(p)}_{N_x+1}(t^n) = \kappa^{\,(p)}\,\xbar U^{\,(p)}_{1}(t^n), \quad p=m+1,\ldots,d.  
\end{equation}
Here, $U^{(p)}$ denotes the $p$-th characteristic component of the solution variable, obtained after local diagonalization near the stationary state. For diagonal systems, this notation coincides with the physical components of $\mU$. Here, $\kappa^{\,(p)}$ is the $p$-th component of the feedback parameters $\bm \kappa$. For clarity of presentation, we first describe the algorithm for a single scalar parameter $\kappa$ and omit the component index; the case of several parameters is treated in complete analogy.

We are interested in those values of $\kappa$ for which the solution of the closed–loop system \eqref{1.1}, \eqref{2.1}, \eqref{2.2} satisfies an exponential decay estimate in
$L^2(\Omega)$, in the following sense.

\begin{defn}[Exponential Stabilization]
Fix a feedback parameter $\kappa$ and consider the corresponding closed–loop system \eqref{1.1}, \eqref{2.1}, \eqref{2.2}. We say that $\kappa$ is (continuously) exponentially
stabilizing in $L^2(\Omega)$ if there exist constants $C_1 \ge 1$ and $\nu>0$, independent of $t$ and of the initial datum $\mU(\cdot,0)\in L^2(\Omega)$, such that
\begin{equation*}
||\mU(\cdot,t)||^2_{L^2(\Omega)} \le C_1 e^{-\nu t} ||\mU(\cdot,0)||^2_{L^2(\Omega)},\quad t\ge 0.
\end{equation*}
In the semi-discrete setting, this corresponds to
\begin{equation}\label{2.2a}
\Delta x\sum_{j=1}^{N_x}\|\mU_j(t)\|_2^2\le C_1 e^{-\nu t}
\Delta x\sum_{j=1}^{N_x}\|\mU_j(0)\|_2^2,\qquad t\ge 0.
\end{equation}
\end{defn}
In the numerical experiments, this decay is monitored by Lyapunov–type indicators.

\begin{defn}[Lyapunov–Type Indicators]
Let $\xbar \mU_j^{\,n}(\kappa) = \big(\xbar u^{(1),n}_j(\kappa),\dots, \xbar u^{(d),n}_j(\kappa)\big)^\top$ denote the numerical solution of the LLF scheme of \S\ref{sec2.1} at time $t^n$ in cell $C_j$, obtained with boundary parameter $\kappa$. For fixed constants $\mu_i>0$, $i=1,\dots,d$, we define the discrete Lyapunov functional
\begin{equation}\label{2.3}
\mathcal{L}^n_{\,\mathrm{Lyapunov}}(\kappa)= \dx \sum_{i=1}^{m} \sum_{j=1}^{N_x}e^{-\mu_i x_j}\bigl(\xbar u^{(i),n}_j(\kappa)\bigr)^2+\dx \sum_{i=m+1}^{d} \sum_{j=1}^{N_x}e^{\mu_i x_j}\bigl(\xbar u^{(i),n}_j(\kappa)\bigr)^2,
\end{equation}
and the discrete energy
\begin{equation}\label{2.4}
\mathcal{L}^n_{\,\mathrm{Energy}}(\kappa)= \dx \sum_{i=1}^{d} \sum_{j=1}^{N_x} \bigl(\xbar u^{(i),n}_j(\kappa)\bigr)^2.
\end{equation}
Any quantity of the form $\mathcal{L}^n(\kappa)$ obtained from \eqref{2.3} or \eqref{2.4} will be referred to as a Lyapunov–type indicator. In \S \ref{sec3} we simply write $\mathcal{L}^n(\kappa)$ for the chosen indicator. At the same time, one can easily verify 
\begin{equation}\label{2.4a}
\widetilde{C}_2\mathcal{L}^n_{\,\mathrm{Lyapunov}}(\kappa) \le   \mathcal{L}^n_{\,\mathrm{Energy}}(\kappa)  \le \widetilde{C}_1 \mathcal{L}^n_{\,\mathrm{Lyapunov}}(\kappa),
\end{equation}
where $\widetilde{C}_1$ and $\widetilde{C}_2$ are two positive constants, which are independent of $n$ and $\kappa$. Therefore, both indicators can be used since they are equivalent by \eref{2.4a}. 
\end{defn}

\begin{prop}[Discrete Exponential Stabilization for the LLF Scheme]
Assume that the deterministic system \eqref{1.1} with boundary conditions \eqref{2.1}–\eqref{2.2} satisfies the hypotheses of \cite{coron2007control,banda2013numerical} and that the numerical solution is computed by the first–order LLF scheme of \S\ref{sec2.1}. Then there exist constants $\nu>0$ and $\widetilde{C}\ge 0$, depending only on the system and the feedback law, such that for every feedback parameter $\kappa$ in the corresponding stability domain the numerical solution satisfies
\begin{equation}\label{2.5}
\mathcal{L}^n_{\mathrm{Lyapunov}}(\kappa)\le e^{-\nu t^n}\,\mathcal{L}^0_{\mathrm{Lyapunov}}(\kappa),\qquad n\ge 0,
\end{equation}
and, in addition, the exponential stabilization inequality in the discrete $L^2$–norm
\begin{equation}\label{2.6}
\mathcal{L}^n_{\mathrm{Energy}}(\kappa)\le \widetilde{C}\,e^{-\nu t^n}\,\mathcal{L}^0_{\mathrm{Energy}}(\kappa),\qquad n\ge 0.
\end{equation}
An analogous statement for stochastic systems is obtained in \cite{CHK_Stability}. In particular, the indicators \eqref{2.3}–\eqref{2.4} are consistent discrete realizations of the continuous Lyapunov estimate \eqref{2.2a} and can be used to detect stabilizing feedback parameters.
\end{prop}

The estimates \eqref{2.5}–\eqref{2.6} are proved in \cite[Theorem~3.2]{banda2013numerical} for deterministic systems and extended to the random setting in \cite[Theorem~4.2]{CHK_Stability}. The proofs are based on discrete Lyapunov functionals with weights $e^{-\mu_i x_j}$ and $e^{\mu_i x_j}$, which are precisely the weights used in \eqref{2.3}. We therefore refer to these references for the details.

Formally, the Bayesian feedback procedure can be viewed as an iteration on a family of posterior distributions $\{{\cal P}^n\}_{n\ge 0}$ supported on a compact parameter set. At each time step, the decay behaviour of the indicator $\mathcal{L}^n(\kappa)$ provides information on whether the chosen value of $\kappa$ is stabilizing.

\begin{defn}[Probability Distributions of Feedback Parameters]
Let $\kappa \in [a,b]^d$ denote the feedback parameters and let ${\cal P}^n : [a,b]^d \to [0,\infty)$ be a nonnegative function. If
\begin{equation}\label{2.7a}
\int_{[a,b]^d} {\cal P}^n(\kappa)\, \mathrm{d}\kappa = 1,
\end{equation}
then ${\cal P}^n$ is called a probability density function of $\kappa$ on $[a,b]^d$. To obtain a discrete representation, let \(\{\kappa_\ell\}_{\ell=0}^{N_\kappa}\subset [a,b]^d\) be a finite grid of feedback parameters and let \(\omega_\ell>0\) denote the corresponding quadrature weights. A vector ${\cal P}^n = \bigl({\cal P}^n(\kappa_0),\dots,{\cal P}^n(\kappa_{N_\kappa})\bigr)$ with ${\cal P}^n(\kappa_\ell)\ge 0$ for all $\ell$ and
\begin{equation}\label{2.7}
\sum_{\ell=0}^{N_\kappa}\omega_\ell P^n(\kappa_\ell)=1
\end{equation}
is called a discrete probability mass function on the grid $\{\kappa_\ell\}_{\ell=0}^{N_\kappa}$. For a uniform one-dimensional grid, one may take $\omega_\ell=\Delta \kappa$, up to the usual endpoint convention. We interpret ${\cal P}^n$ and its discrete counterpart as the posterior distribution of $\kappa$ after the $n$-th update.
\end{defn}

\begin{defn}[Bayes’ Formula]
Bayes’ formula provides the rule for updating probabilities when new information becomes available. Given an event space $A \in {\cal X}$ and observed data $B$ with $\mathbb{P}(B)>0$, Bayes’ formula reads
\begin{equation}\label{2.10a}
\mathbb{P}(A|B)= \frac{\mathbb{P}(B|A)\,\mathbb{P}(A)}{\mathbb{P}(B)}.
\end{equation}
Here, $\mathbb{P}(A)$ is the prior probability of $A$, $\mathbb{P}(B|A)$ is the likelihood of observing $B$ under the hypothesis $A$, and $\mathbb{P}(A|B)$ is the posterior probability after taking the information $B$ into account. If the event space is discretized with the events $A_n$, then \eref{2.10a} leads to 
\begin{equation*}
\mathbb{P}(A_k\mid B)= \frac{\mathbb{P}(B|A_k)\,\mathbb{P}(A_k)}{\displaystyle\sum_{i=0}^{N_\kappa} \mathbb{P}(B|A_i)\,\mathbb{P}(A_i)},\qquad k=0,\dots,N_\kappa.
\end{equation*}
\end{defn}

In our setting, the hypotheses are the grid values $\{\kappa_\ell\}_{\ell=0}^{N_\kappa}$ and the data $B$ are the observed values of the Lyapunov–type indicator at time $t^{n+1}$.

\begin{defn}[Bayesian Update Operator]
Given a prior ${\cal P}^n$ as in \eqref{2.7a} and \eqref{2.7} and non–negative likelihood $\Lambda^{n+1}(\kappa)>0$, the Bayesian update of the feedback distribution is defined by
\begin{equation}\label{2.8}
{\cal P}^{n+1}(\kappa)= \frac{\Lambda^{n+1}(\kappa)\,{\cal P}^n(\kappa)} {\displaystyle \int_{[a,b]^d} \Lambda^{n+1}(\kappa)\,{\cal P}^n(\kappa) \mathrm{d}\kappa },
\end{equation}
and 
\begin{equation}\label{2.8a}
{\cal P}^{n+1}(\kappa_\ell)= \frac{\Lambda^{n+1}(\kappa_\ell)P^n(\kappa_\ell)} {\sum_{m=0}^{N_\kappa}\omega_m \Lambda^{n+1}(\kappa_m)P^n(\kappa_m)}.
\end{equation}
\end{defn}

\begin{prop}[Positivity and Normalization]
In the continuous case, assume that ${\cal P}^n$ satisfies \eqref{2.7a} and that the denominator in \eqref{2.8} is positive. Then ${\cal P}^{n+1}$ defined by \eqref{2.8} is nonnegative and satisfies \eqref{2.7a}. In the discrete case, assume that ${\cal P}^n$ satisfies \eqref{2.7} and that the denominator in \eqref{2.8a} is positive. Then ${\cal P}^{n+1}$  defined by \eqref{2.8a} is nonnegative and satisfies \eqref{2.7}. Moreover, if $\Lambda^{n+1}=1$ on a set of stabilizing parameters and $\Lambda^{n+1}=\alpha\in(0,1)$ on its complement, then the total probability mass assigned to the stabilizing set is non-decreasing with respect to the time level $n$.
\end{prop}

\noindent{\bf Proof.}{\textit{
Positivity and normalization follow immediately from \eqref{2.8} in the continuous case and from \eqref{2.8a} together with \eqref{2.7} in the discrete case. In the two–level case $\Lambda^{n+1}\in\{1,\alpha\}$ with $\alpha\in(0,1)$, the probability mass on the stabilizing set is multiplied by $1$ at each step, whereas mass on the complement is multiplied by $\alpha$ and subsequently renormalized. A direct computation shows that the fraction of total mass carried by the stabilizing set is therefore non–decreasing. \hfill $\square$}}

In the computations below, one can choose the likelihood $\Lambda^{n+1}(\kappa)$ as simple functions of the indicator $\mathcal{L}_{\rm \, Lyapunov}^{n+1}(\kappa)$ and its initial value $\mathcal{L}_{\rm \, Lyapunov}^0(\kappa)$. One can adopt 
\begin{equation}\label{2.9}
\Lambda^{n+1}(\kappa):=\begin{cases}
        1, & \mathcal{L}^{n+1}_{\rm \, Lyapunov}(\kappa)\le C_1 e^{-\nu t^{n+1}}\mathcal{L}^{0}_{\rm \, Lyapunov}(\kappa),\\[0.3em]
        \alpha, & otherwise,
    \end{cases}
\end{equation}
where $\alpha\in(0,1)$ and $C_1>0$ is a constant. Thus, parameters with non–increasing indicators retain their weight, whereas parameters with increasing indicators are penalized by the damping factor $\alpha$. However, this likelihood is directly applicable only to systems for which Lyapunov functionals and the corresponding decay rates are available. For systems whose Lyapunov stability analysis is very complicated,  it is too challenging to use the method \eref{2.9}.

To address this issue, we therefore adopt
\begin{equation}\label{2.9a}
\Lambda^{n+1}(\kappa):=\begin{cases}
        1, & \mathcal{L}^{n+1}(\kappa)\le \mathcal{L}^0(\kappa),\\[0.3em]
        \alpha, & \mathcal{L}^{n+1}(\kappa)>\mathcal{L}^0(\kappa),
    \end{cases}
\end{equation}
where $\mathcal{L}$ can be the Lyapunov function $\mathcal{L}_{\rm \, Lyapunov}$ or the energy function $\mathcal{L}_{\rm \, Energy}$. In this manner, one can numerically identify candidate stabilizing domains for a given system without requiring an explicit analytical decay rate. The numerical results reported in \S \ref{sec3} 
show that both \eref{2.9} and \eref{2.9a} yield the same stability domains for $\kappa$.

In the numerical implementation, the distribution ${\cal P}^n$ is approximated by a discrete probability mass function on the uniform grid $\{\kappa_\ell\}_{\ell=0}^{N_\kappa}\subset[a,b]$. Given a prior ${\cal P}^0(\kappa_\ell)>0$ satisfying \eref{2.7}, reference indicator values $\mathcal{L}^0(\kappa_\ell)$, a damping factor $\alpha\in(0,1)$, and a tolerance $\varepsilon_0>0$, we employ the following algorithm, which is the concrete realization of the Bayesian update \eqref{2.8} with the choice \eqref{2.9a}.

\begin{algorithm}[H]
  \caption{Bayesian update of the feedback parameter distribution}
  \begin{algorithmic}[1]
    \State Choose a parameter grid $\{\kappa_\ell\}_{\ell=0}^{N_\kappa} \subset [a,b]^d$.
    \State Prescribe a prior ${\cal P}^0(\kappa_\ell) > 0$ with  $\sum_{\ell=0}^{N_\kappa} \omega_\ell {\cal P}^0(\kappa_\ell) = 1$.
    \State Compute reference indicator values ${\cal L}_\ell^{0} = {\cal L}^{0}(\kappa_\ell)$.
    \State Fix a damping factor $\alpha \in (0,1)$ and a tolerance $\varepsilon_0 > 0$.
    \For{$n = 0,1,2,\dots$}
      \State {Random choice of $\ell\in \{0,\dots,N_\kappa\}$}.
        \State Advance $\mU_j^{n}(\kappa_\ell) \to \mU_j^{n+1}(\kappa_\ell)$ with the LLF scheme in \S \ref{sec2.1}.
        \If{${\cal L}^{n+1}(\kappa_\ell) \le {\cal L}^{0}(\kappa_\ell)$}
          \State $\widetilde{\cal P}^{\,n+1}(\kappa_\ell) \leftarrow {\cal P}^{n}(\kappa_\ell)$
        \Else
          \State $\widetilde{\cal P}^{\,n+1}(\kappa_\ell) \leftarrow \alpha\,{\cal P}^{n}(\kappa_\ell)$
        \EndIf
      \State Normalize the probabilities as in \eref{2.8a}.
      \State Compute the variation of the distribution ${\cal V}^{n+1} =\bigl|\bigl|{\cal P}^{n+1} - {\cal P}^{n}\bigr|\bigr|_{\ell^1}$.
      \If{${\cal V}^{n+1} \le \varepsilon_0$}
        \State \textbf{break}
      \EndIf
    \EndFor
  \end{algorithmic}
\end{algorithm}

In all numerical experiments reported in \S\ref{sec3}, we use $\varepsilon_0=10^{-12}$ and $\alpha=\tfrac{1}{2}$. Parameters $\kappa_\ell$ in the unstable region satisfy ${\cal P}^n(\kappa_\ell)\approx 0$ in the limiting distribution, whereas stabilizing parameters satisfy ${\cal P}^n(\kappa_\ell)\gg 0$.

\begin{rmk}
In place of the reference value $\mathcal{L}^0(\kappa_\ell)$ used in the algorithm, one may also compare $\mathcal{L}^{n+1}(\kappa_\ell)$ with the previous indicator value $\mathcal{L}^n(\kappa_\ell)$ and modify the damping step as follows:
\begin{equation*}
    \text{if } \mathcal{L}^{n+1}(\kappa_\ell) \le \mathcal{L}^n(\kappa_\ell)
    \ \Rightarrow\ 
    \widetilde{P}^{\,n+1}(\kappa_\ell) \gets P^n(\kappa_\ell),
    \quad
    \text{if } \mathcal{L}^{n+1}(\kappa_\ell) > \mathcal{L}^n(\kappa_\ell)
    \ \Rightarrow\ 
    \widetilde{P}^{\,n+1}(\kappa_\ell) \gets \alpha\,P^n(\kappa_\ell).
\end{equation*}
The subsequent normalization and the stopping criterion based on ${\cal V}^{n+1}$ remain unchanged. In all experiments presented in \S\ref{sec3}, this alternative update leads to the same stability domains as the reference–based criterion, with only minor differences in the number of iterations needed for convergence.
\end{rmk}

\begin{rmk}
The Bayesian update step does not rely on strict hyperbolicity itself. It only requires an admissible finite-dimensional feedback parametrization and a nonnegative stability indicator. If eigenvalues coincide or vanish, the construction of characteristic variables, boundary feedback laws, and Lyapunov-type indicators becomes model-dependent. Once these ingredients are specified, the probability update and the positivity-normalization property remain unchanged.
\end{rmk}

\section{Numerical Results}\label{sec3}

In this section, we assess the proposed boundary-feedback control on a suite of numerical tests. We begin with linear deterministic models (linear wave; linearized Saint-Venant) and verify that the stability domains agree with theory. We then treat nonlinear problems (nonlinear Saint-Venant; Burgers), followed by stochastic variants, a nonconservative case, and a second-order LLF discretization. Finally, we examine the evolution of the thermal field during a single-track scanning process in laser powder bed fusion and recover its stability domain via the same Bayesian update.
In Examples 1--6, we use the CFL number 1, while in Examples 7--9, we use the CFL number 0.5.

\subsection{Linear Deterministic Examples}
In this section, we consider two linear deterministic examples---the linear wave equations and the linearized Saint-Venant system---to show that the stability domains computed by the proposed method are consistent with the theoretical results in \cite{banda2013numerical}.

\paragraph*{Example 1---Linear Wave Equations.}
In the first example taken from \cite{banda2013numerical}, we consider the linear wave equations 
\begin{equation}\label{3.1a}
\begin{pmatrix}
  u^{(1)}  \\
  u^{(2)}  \\
\end{pmatrix}_t + \begin{pmatrix}
  1 & 0  \\
  0 & -1  \\
\end{pmatrix}
\begin{pmatrix}
  u^{(1)}  \\
  u^{(2)}  \\
\end{pmatrix}_x=
\begin{pmatrix}
  0  \\
  0  \\
\end{pmatrix}
\end{equation}
with the following initial conditions:
\begin{equation*}
u^{(1)}(x,0)=-\frac{1}{2}, \quad  u^{(2)}(x,0)=\frac{1}{2}.    
\end{equation*}
The boundary conditions are given by 
\begin{equation*}
u^{(1)}(0,t)= \kappa  u^{(1)}(1,t), \quad u^{(2)}(1,t)= \kappa  u^{(2)}(0,t).
\end{equation*}
For the initial probability distributions of the parameter $\kappa$, we consider the following two cases:
\begin{itemize}
  \item Case I:  $\kappa \in [-2,2]$ and ${\cal P}(\kappa)= \frac{1}{4}$; 
  \item Case II: $\kappa \in [-5,5]$ and ${\cal P}(\kappa)= \frac{1}{\sigma \sqrt{2 \pi}}e^{-\frac{(\kappa-\mu)^2}{2 \sigma^2}}$, with $\sigma=\mu=1$. 
\end{itemize}

We first normalize the probability distributions and then compute the numerical results on a uniform mesh with $N_x=100$ and $N_\kappa=800$. The obtained numerical results are presented in Figure \ref{fig1}. Here, the values of the probability distributions are updated based on the corresponding Lyapunov function 
$$
{\cal L}\big(\xbar u^{(1)}, \xbar u^{(2)}\big) =\dx \sum_{j=1}^{N_x}e^{-\mu_1 x_j} \Big( \xbar u^{(1)}_j \Big)^2 +\dx \sum_{j=1}^{N_x}e^{\mu_2 x_j} \Big(\xbar u^{(2)}_j\Big)^2,
$$
where $\mu_1$ and $\mu_2$ are two positive constants. For more details, we refer the reader to \cite[(33)]{banda2013numerical}. As expected, the stability interval $(-1,1)$ is recovered independently of the chosen prior, which is consistent with \cite[Theorem 2.1]{banda2013numerical} and validates the proposed methodology.

\begin{figure}[ht!]
\centerline{\includegraphics[trim=0.9cm 0.4cm 1.2cm 0.2cm, clip, width=6.cm]{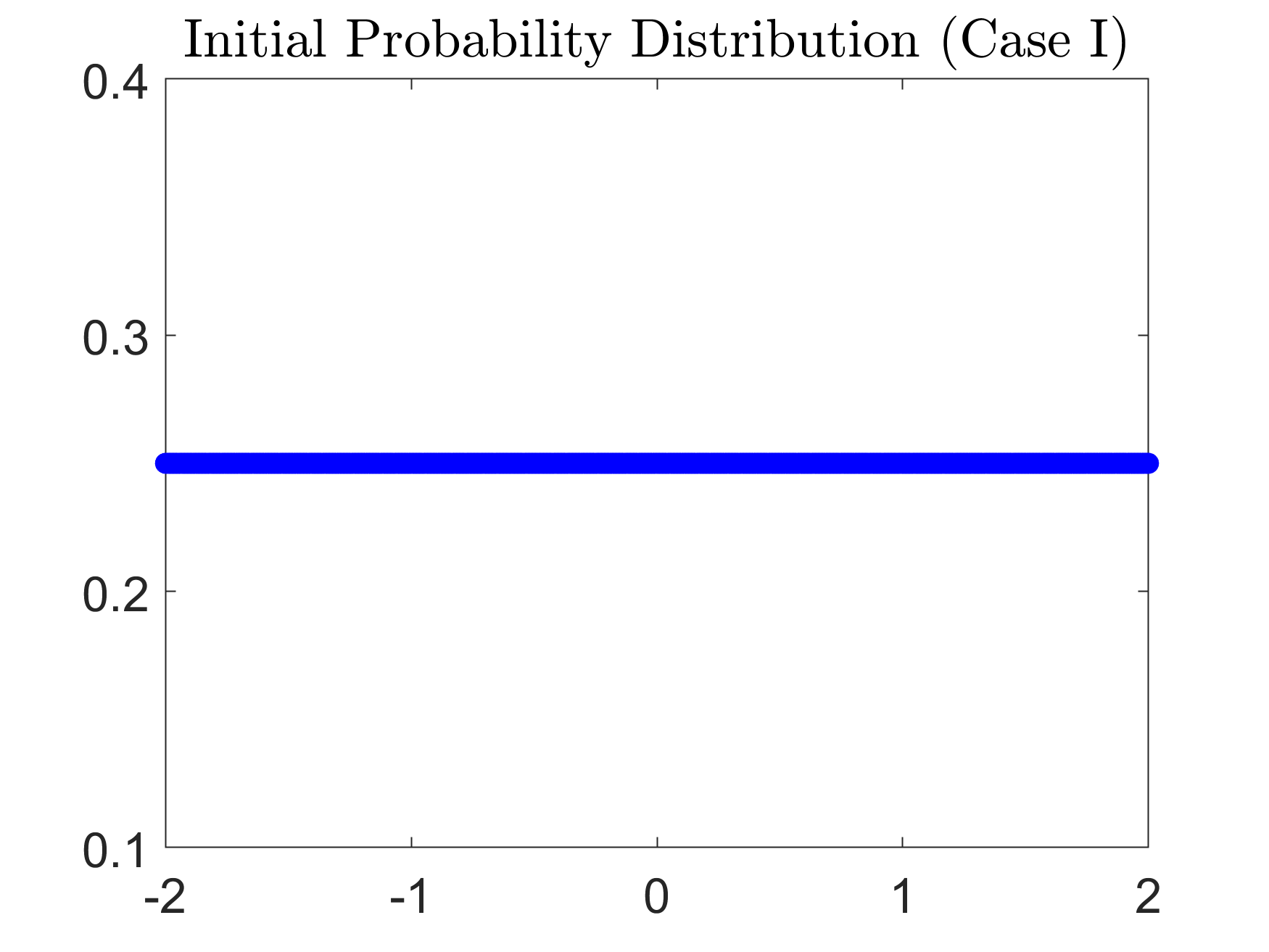}\hspace*{1cm}
            \includegraphics[trim=0.9cm 0.4cm 1.2cm 0.2cm, clip, width=6.cm]{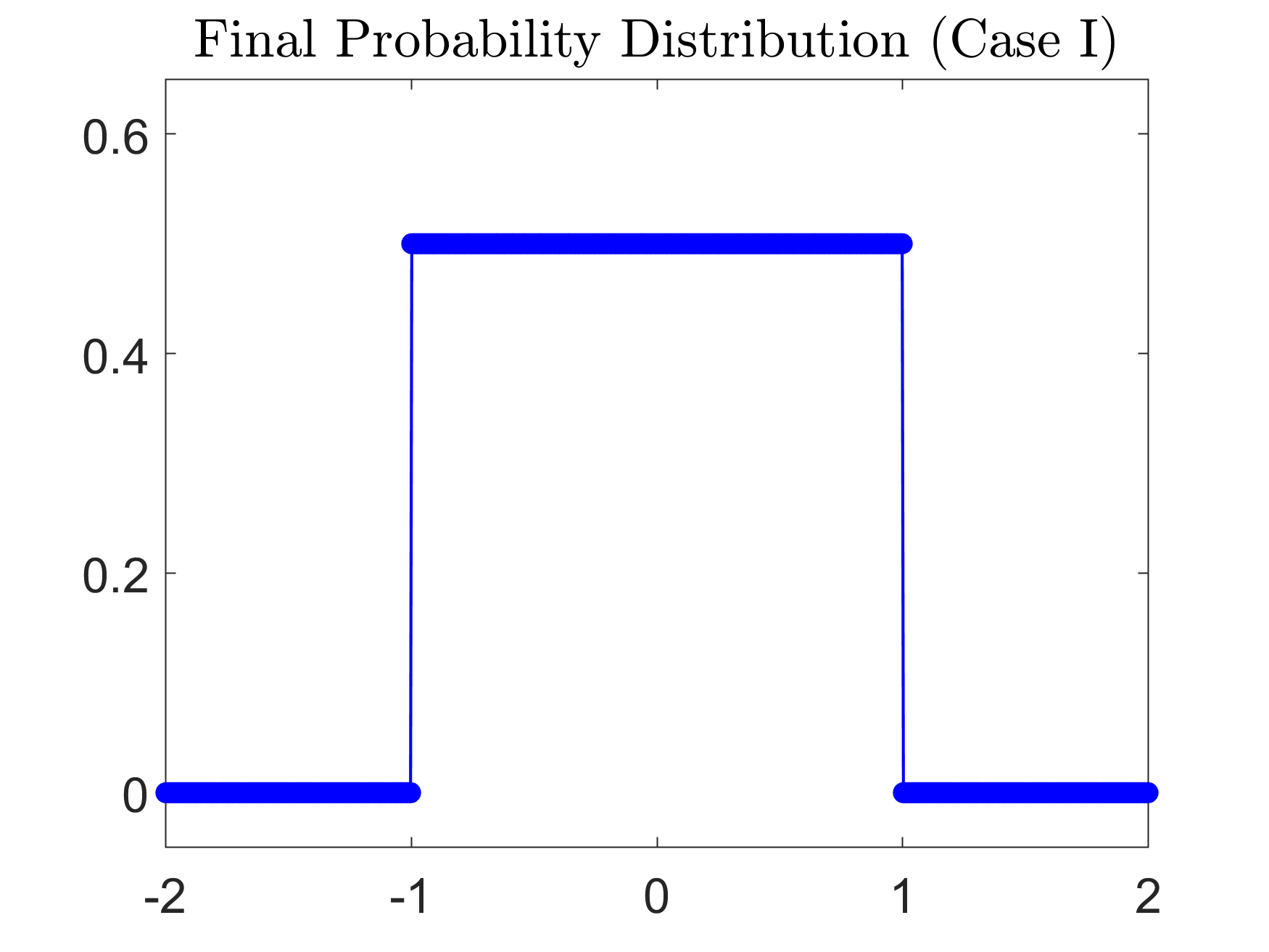}}
\vskip 12pt 
\centerline{\includegraphics[trim=0.9cm 0.4cm 1.2cm 0.2cm, clip, width=6.cm]{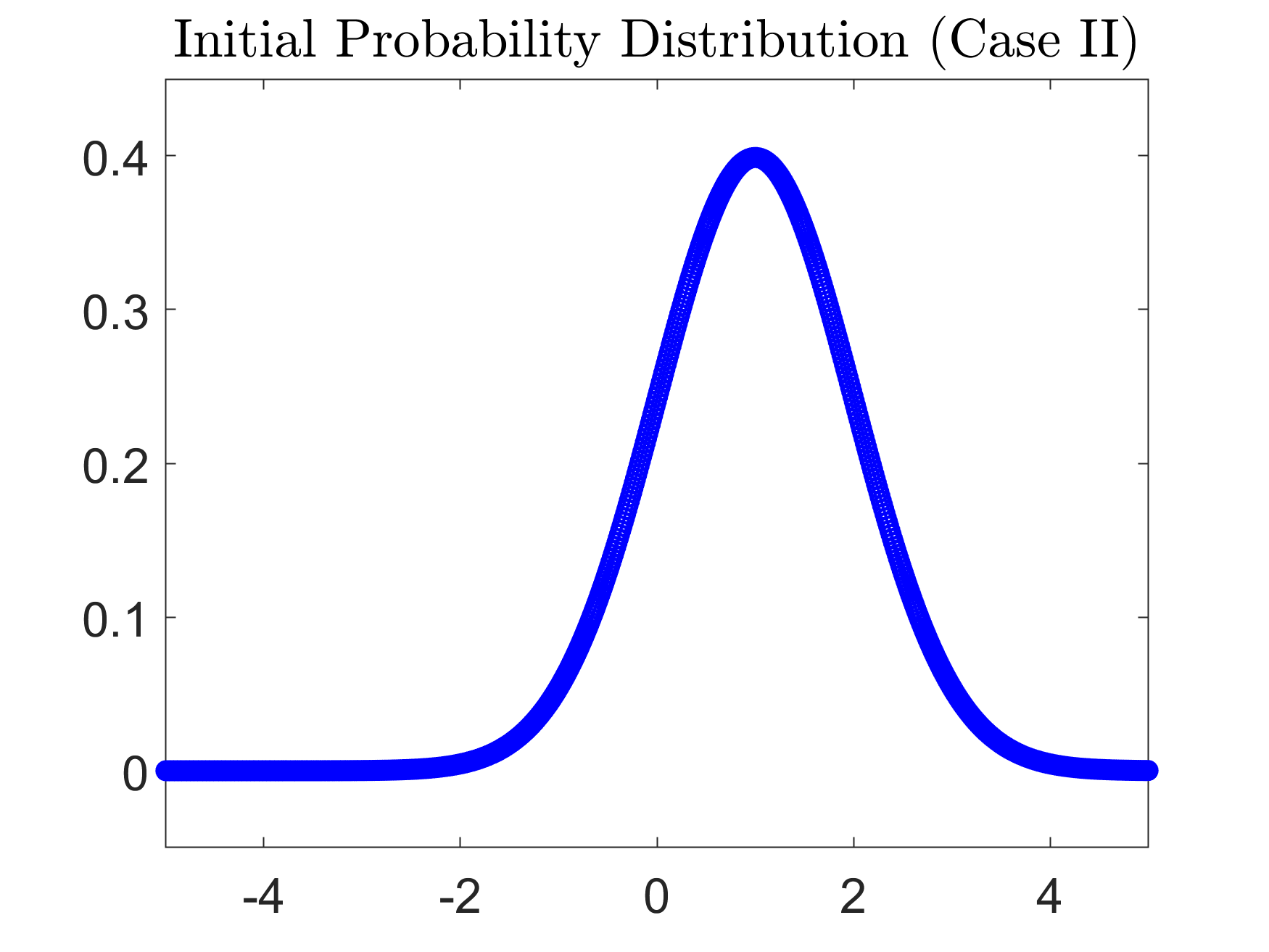}\hspace*{1cm}
            \includegraphics[trim=0.9cm 0.4cm 1.2cm 0.2cm, clip, width=6.cm]{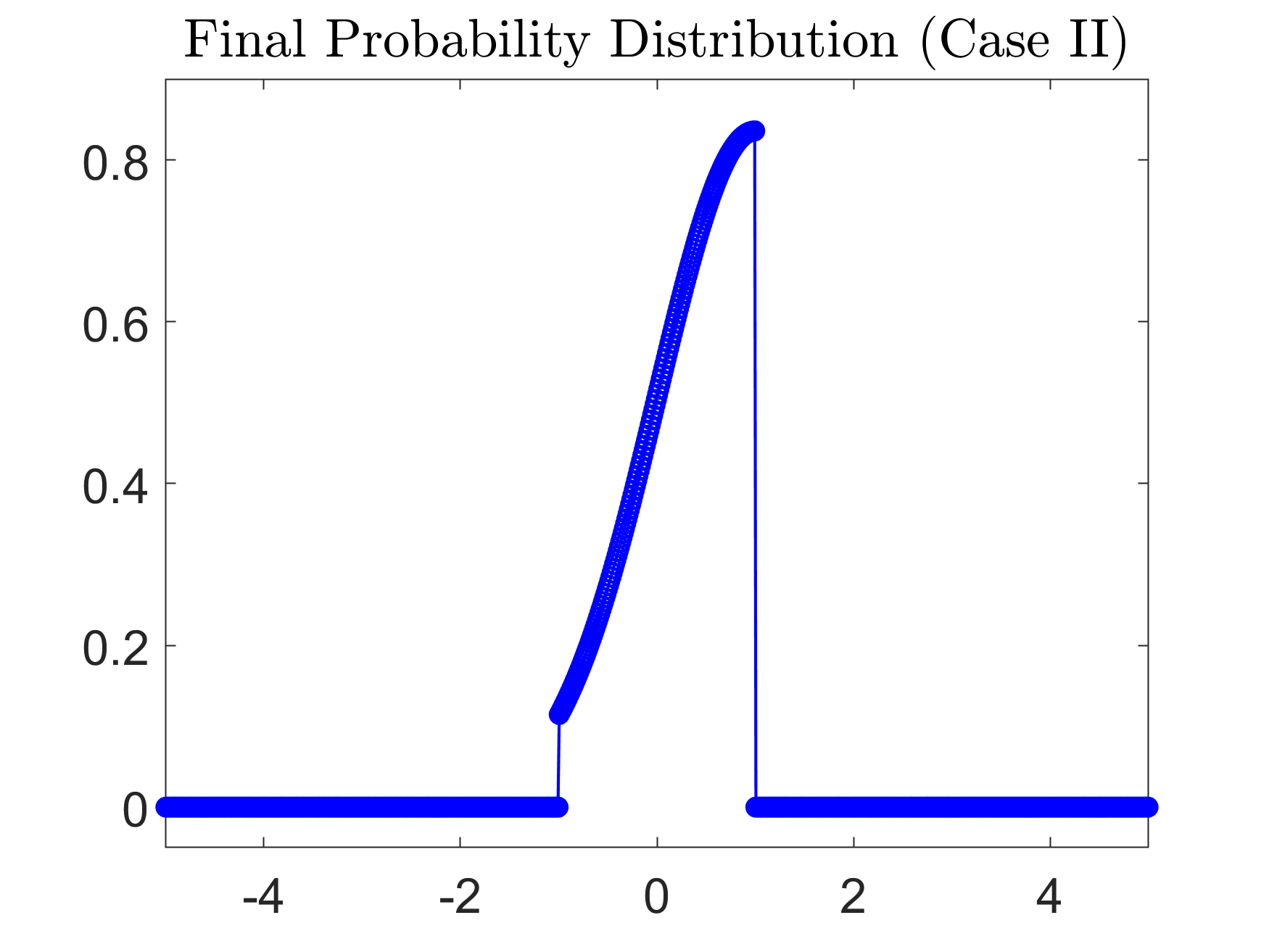}}            
\caption{\sf Example 1: Initial (left) and final probability distributions for Cases I (top row) and II (bottom row).\label{fig1}}
\end{figure}

As discussed above, one may also use other Lyapunov-type indicators to update the probability distribution. In this example, we also show the numerical results by measuring the corresponding discrete energy
\begin{equation}\label{3.1}
  {\cal L}\big(\xbar u^{(1)}, \xbar u^{(2)}\big)=\dx \sum_{j=1}^{N_x} \Big[  \Big( \xbar u^{(1)}_j \Big)^2+  \Big( \xbar u^{(2)}_j \Big)^2 \Big].
\end{equation}
The numerical results are plotted in Figure \ref{fig3}, where one can see that the numerical results are consistent with those obtained using the Lyapunov functional. 
\begin{figure}[ht!]
\centerline{\includegraphics[trim=0.9cm 0.4cm 1.2cm 0.2cm, clip, width=6.cm]{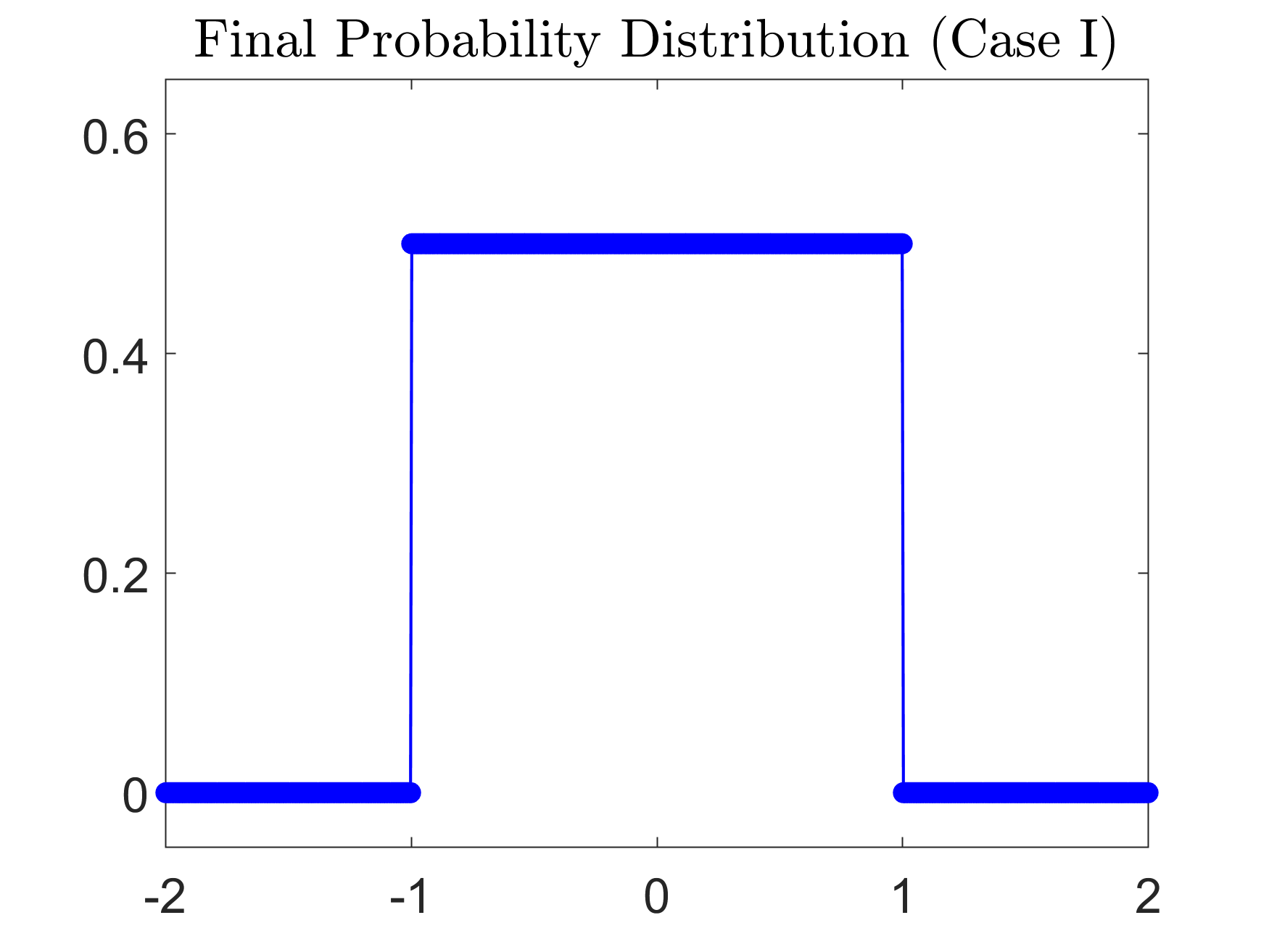}\hspace*{1cm}
            \includegraphics[trim=0.9cm 0.4cm 1.2cm 0.2cm, clip, width=6.cm]{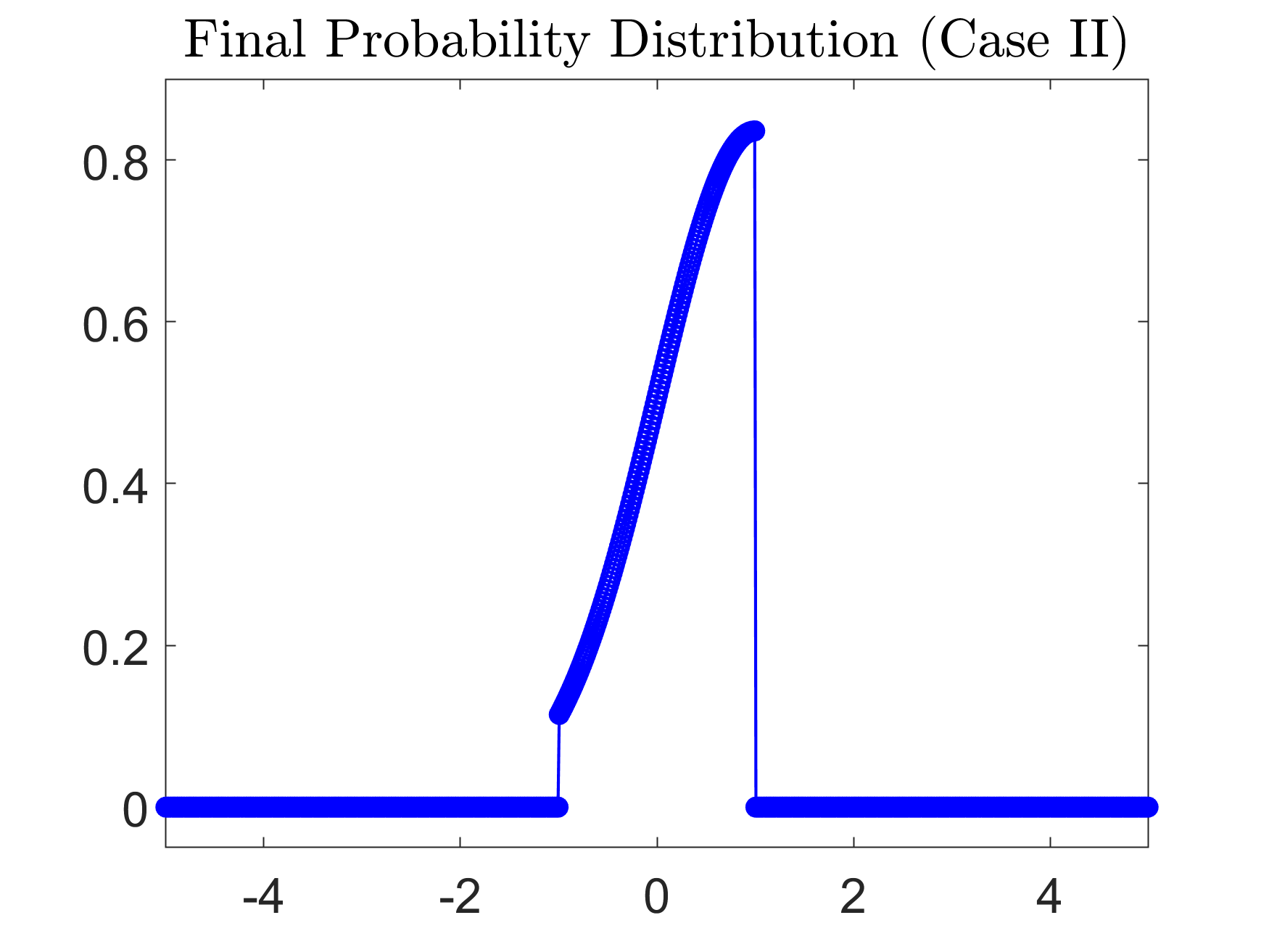}}
\caption{\sf Example 1: Final probability distributions for Cases I (left) and II (right).\label{fig3}}
\end{figure}

We also present the numerical results with $N_\kappa=100$ in Figure 3.3. The same stability interval is recovered, although with lower resolution.
\begin{figure}[ht!]
\centerline{\includegraphics[trim=0.9cm 0.4cm 1.2cm 0.2cm, clip, width=6.cm]{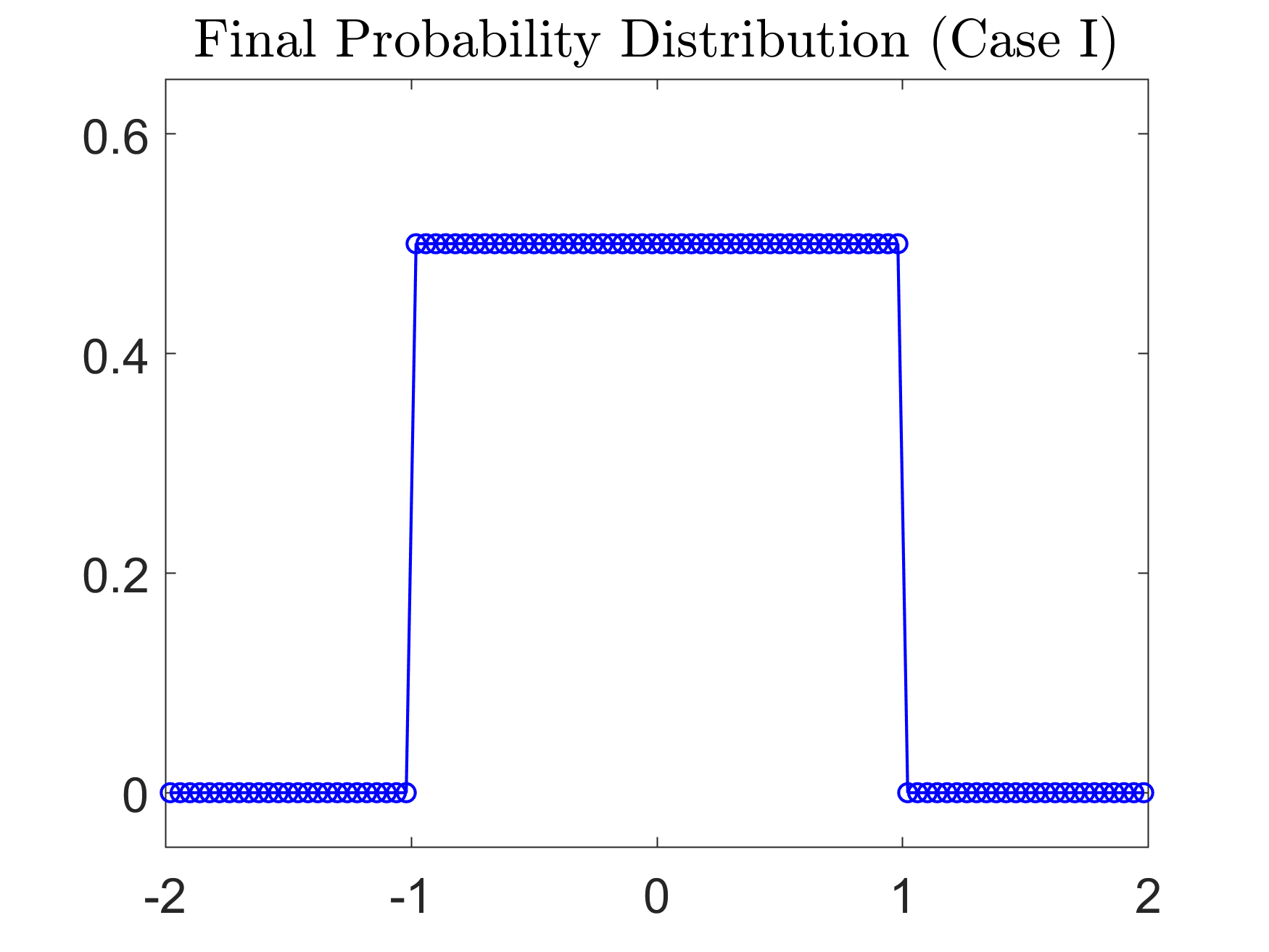}\hspace*{1cm}
            \includegraphics[trim=0.9cm 0.4cm 1.2cm 0.2cm, clip, width=6.cm]{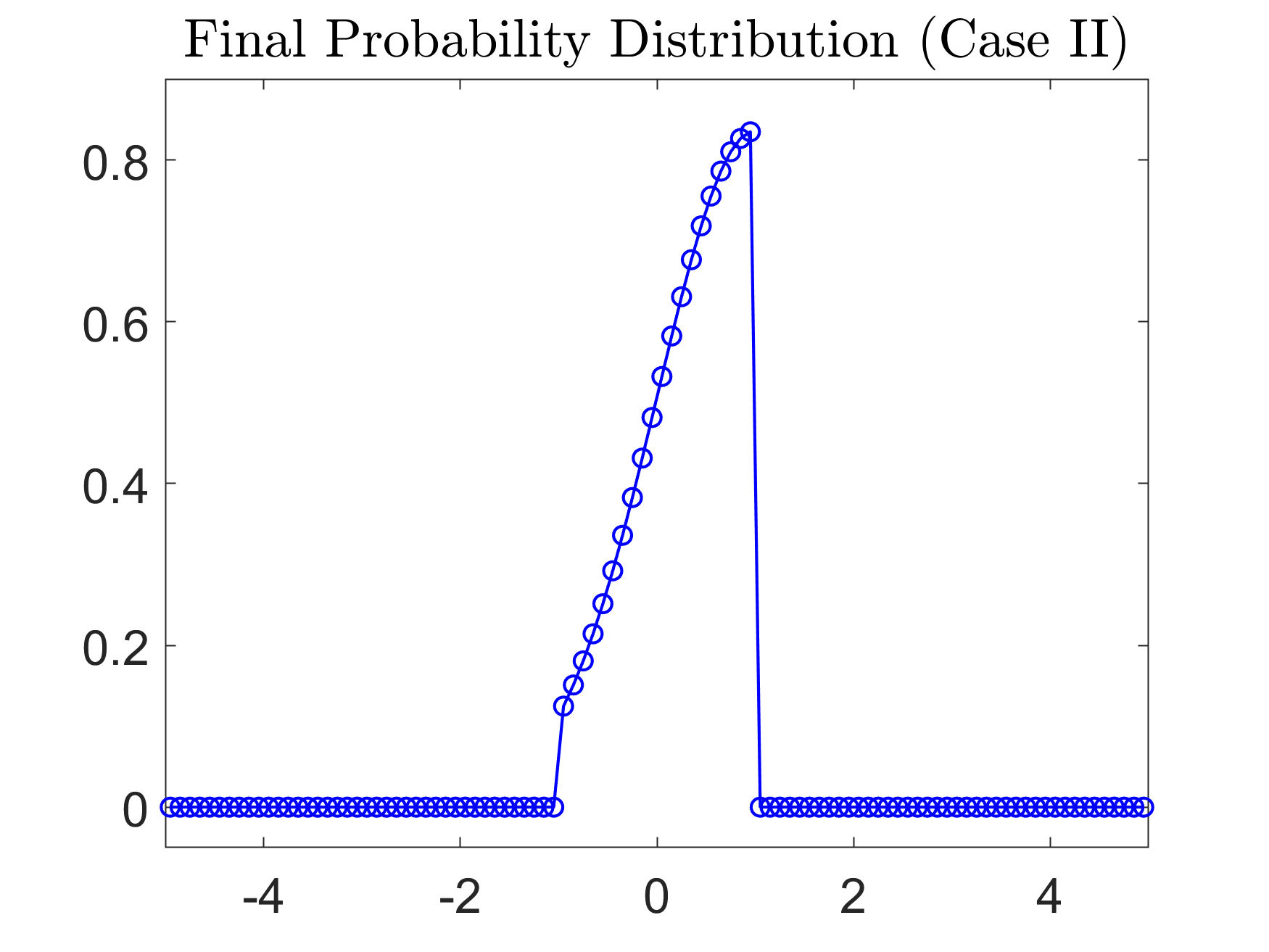}}
\caption{\sf Example 1: Final probability distributions for Cases I (left) and II (right) with $N_\kappa=100$.\label{fig3a}}
\end{figure}

\paragraph*{Example 2---Linearized Saint-Venant System.}
In the second example taken from \cite{banda2013numerical}, we consider the linearized Saint-Venant system of shallow-water equations, which is  motivated by recent advances in the continuous setting; see, e.g., \cite{Coron1999,bastin2007lyapunov,coron2007strict,deHalleux2003}. The objective is to regulate water depth and velocity in an open canal or in canal networks. External factors such as weather variations or uncontrolled inflows can cause fluctuations in water depth, which must be mitigated to maintain a predetermined target depth $\xbar h$ and velocity $\xbar v$. As in \cite{coron2007strict,deHalleux2003}, we neglect the source terms. The target water state is a constant with $\xbar h(x)=4$ and $\xbar v(x)=\frac{5}{2}$. To analyze small deviations from this steady state, the Saint-Venant system of shallow-water equations
\begin{equation*}
 \begin{pmatrix}
	 h \\
	 q
 \end{pmatrix} _t+ \begin{pmatrix}
	q  \\
	\dfrac{q^2}{h}+\dfrac{g}{2}h^2
 \end{pmatrix}_x=\begin{pmatrix}
	0 \\
	0 
 \end{pmatrix},
\end{equation*}
where $q=hv$, is linearized by introducing perturbations $(\delta h, \delta v)$ with $h=\xbar h+\delta h$ and $v = \xbar v +\delta v$. The resulting diagonalized equations for these perturbations are given by:
\begin{equation*}
 \begin{pmatrix}
	u^{(1)} \\
	u^{(2)}
 \end{pmatrix} _t+ \begin{pmatrix}
	\Lambda_1 & 0  \\
	0 & \Lambda_2
 \end{pmatrix}\begin{pmatrix}
	u^{(1)}\\
	u^{(2)} 
 \end{pmatrix} _x=\begin{pmatrix}
	0 \\
	0 
 \end{pmatrix},
\end{equation*}
where  
\begin{equation}\label{4.5a}
\Lambda_1=\xbar v + \sqrt{\mathstrut{g\,\xbar h}}, \quad \Lambda_2=\xbar v - \sqrt{\mathstrut{g\,\xbar h}}, \quad   u^{(1)}=\delta v+\sqrt{\frac{g}{\xbar h}}\delta h, \quad  {\rm and} \quad  u^{(2)}=\delta v-\sqrt{\frac{g}{\xbar h}}\delta h.
\end{equation}
In this example, we consider the following initial conditions:
\begin{equation*}
	\delta h(x,0)=\frac{1}{2}\sin(\pi x), \quad \delta v(x,0)=\frac{20}{8+\sin(\pi x)}-\frac{5}{2}, 
\end{equation*}
subject to the boundary conditions 
\begin{equation}\label{3.3a}
u^{(1)}(0,t) = \kappa \, u^{(1)}(1,t), \quad  {\rm and} \quad  u^{(2)}(1,t)  =\kappa\, u^{(2)}(0,t).
\end{equation}
For the initial probability distributions of the parameter $\kappa$, we consider the following two cases:
\begin{itemize}
  \item Case I:  $\kappa \in [-2,2]$ and ${\cal P}(\kappa)= \frac{1}{4}$; 
  \item Case II: $\kappa \in [-5,5]$ and ${\cal P}(\kappa)= \frac{1}{\sigma \sqrt{2 \pi}}e^{-\frac{(\kappa-\mu)^2}{2 \sigma^2}}$, with $\sigma=1, \mu=0$. 
\end{itemize}

We first normalize the probability distributions and then compute the numerical results on a uniform mesh with $N_x=100$ and $N_\kappa=800$. The values of the probability distributions are updated based on the discrete energy \eref{3.1}, and the  obtained numerical results are presented in Figure \ref{fig5}.  One can clearly see that the stability region of the boundary control parameter is $(-1,1)$, which is consistent with the one proved in \cite[Theorem 2.1]{banda2013numerical}, which validates the proposed methodology.

\begin{figure}[ht!]
\centerline{\includegraphics[trim=0.9cm 0.4cm 1.2cm 0.2cm, clip, width=6.cm]{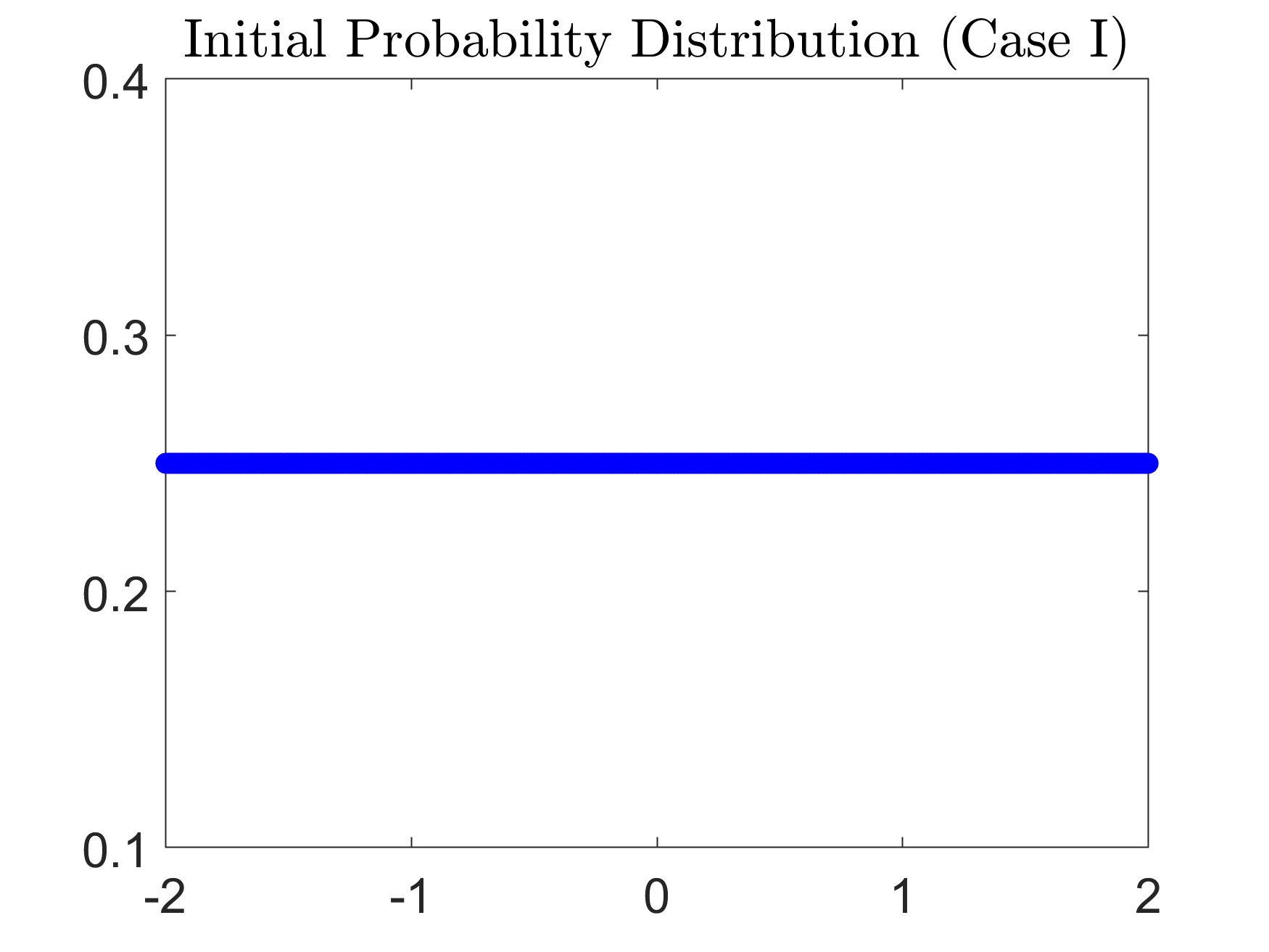}\hspace*{1cm}
            \includegraphics[trim=0.9cm 0.4cm 1.2cm 0.2cm, clip, width=6.cm]{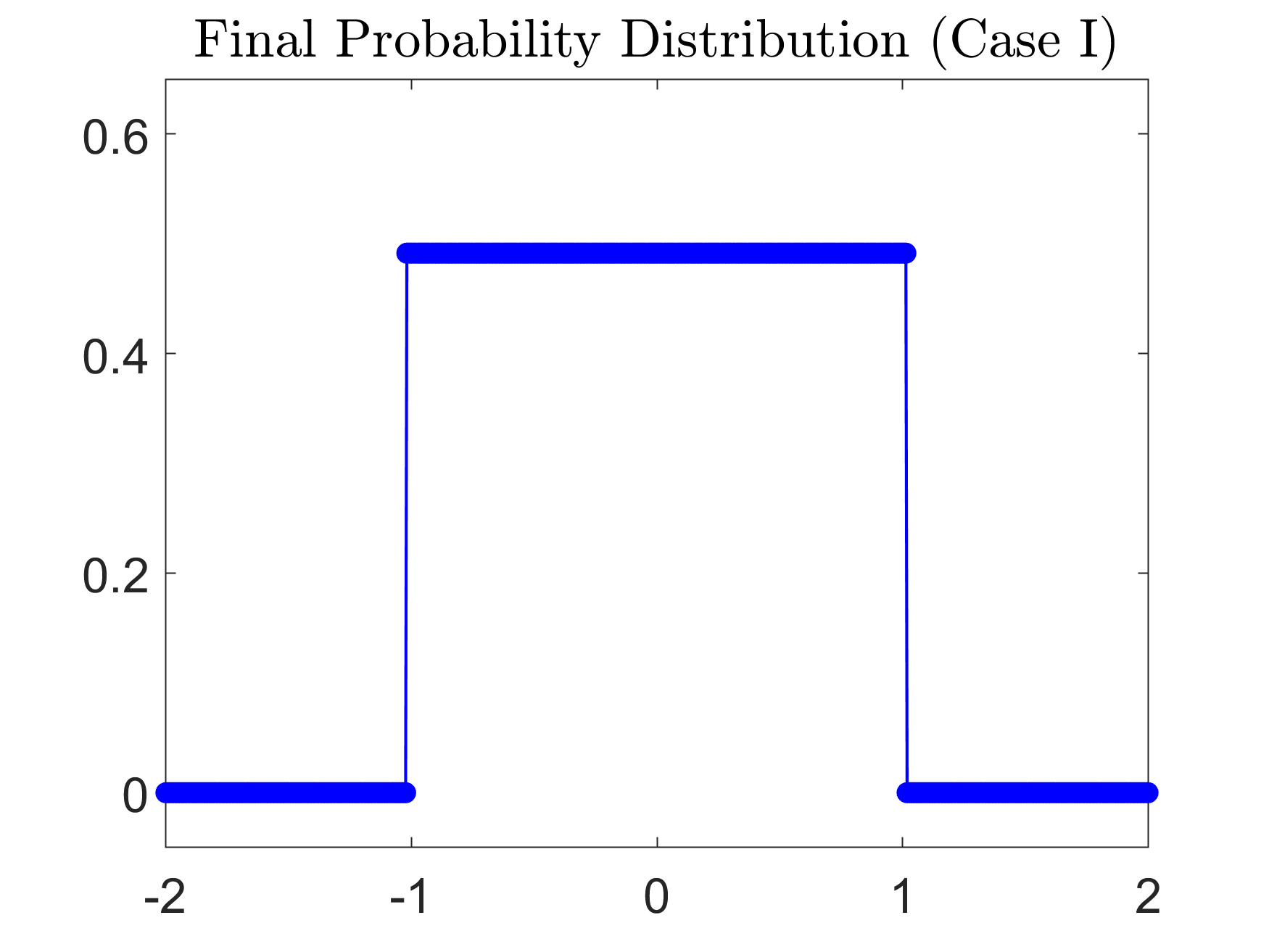}}
\vskip 12pt
\centerline{\includegraphics[trim=0.9cm 0.4cm 1.2cm 0.2cm, clip, width=6.cm]{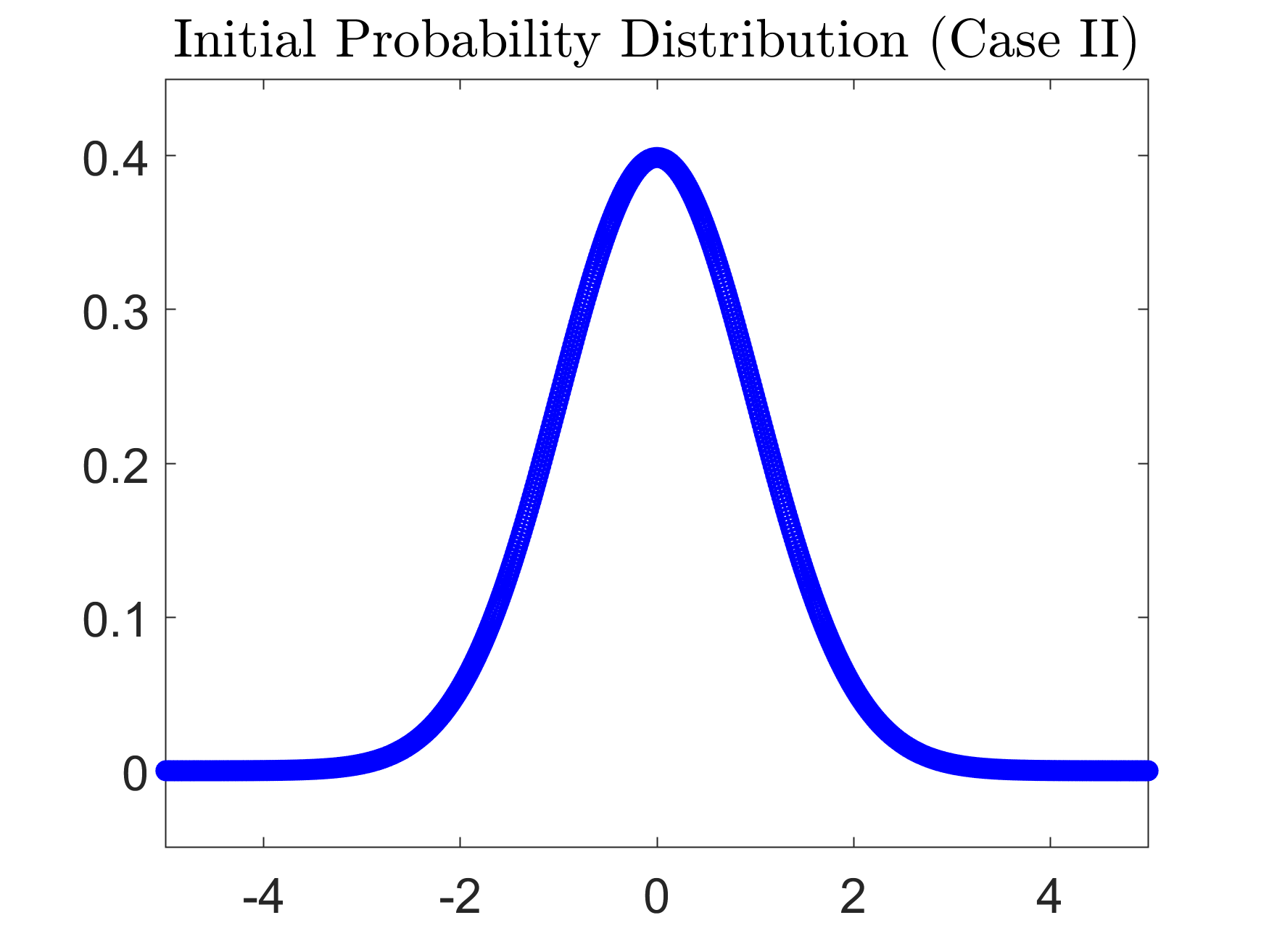}\hspace*{1cm}
            \includegraphics[trim=0.9cm 0.4cm 1.2cm 0.2cm, clip, width=6.cm]{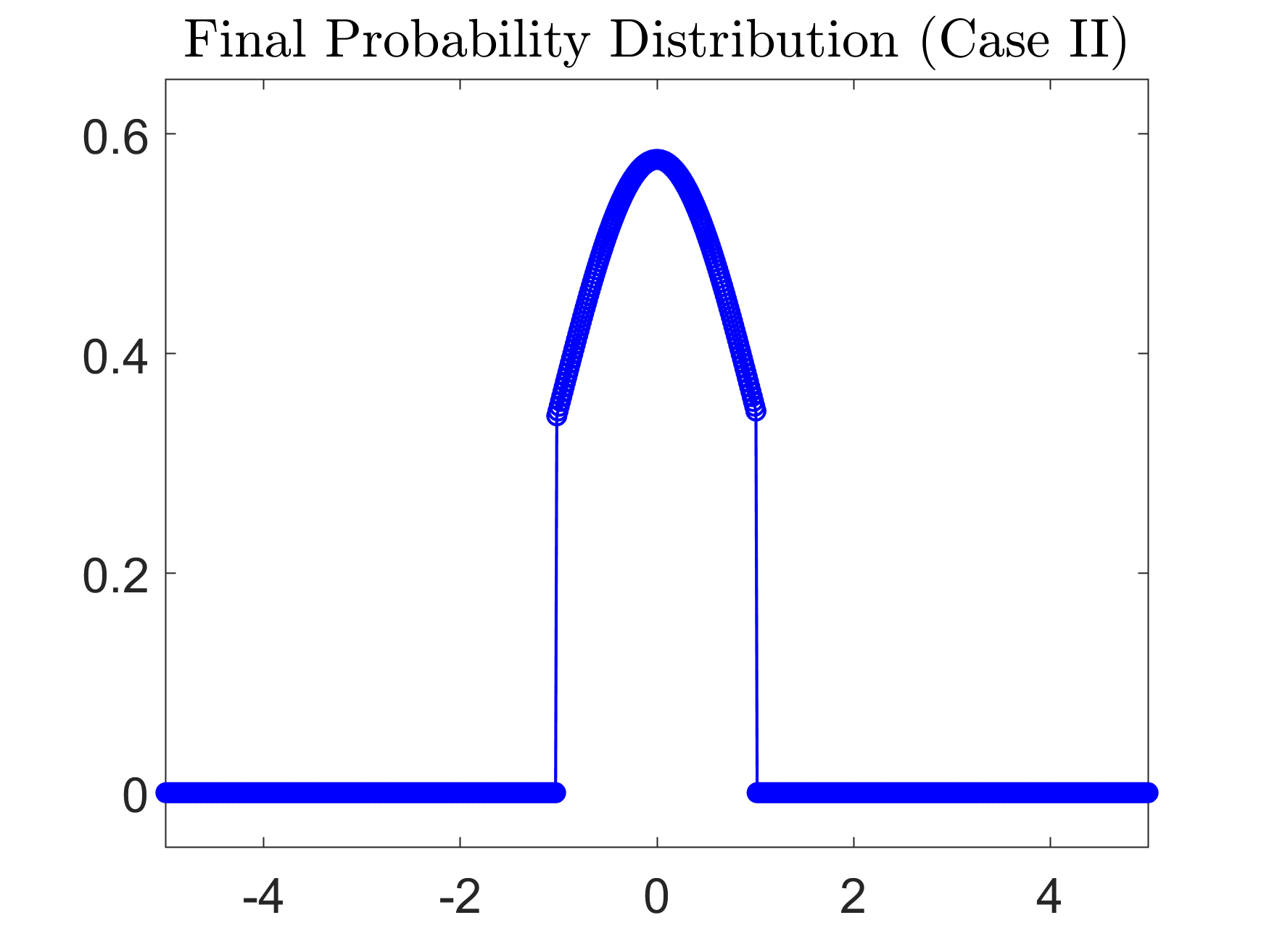}}
\caption{\sf Example 2: Initial (left) and final probability distributions for Cases I (top row) and II (bottom row) for the boundary conditions \eref{3.3a}.\label{fig5}}
\end{figure}

As in \cite{CHK_Stability,banda2013numerical}, we also compute the numerical results using the following boundary conditions:
\begin{equation}\label{3.4a}
u^{(1)}(0,t) = \kappa u^{(2)}(0,t), \quad  u^{(2)}(1,t) = \kappa u^{(1)}(1,t).
\end{equation} 
The obtained numerical results are reported in Figure \ref{fig7}, where one can see that stability region of the boundary control parameter is $(-0.67,0.67)$, which, once again, is consistent with the one proved in \cite[Theorem 2.1]{banda2013numerical}.

\begin{figure}[ht!]
\centerline{\includegraphics[trim=0.9cm 0.4cm 1.2cm 0.2cm, clip, width=6.cm]{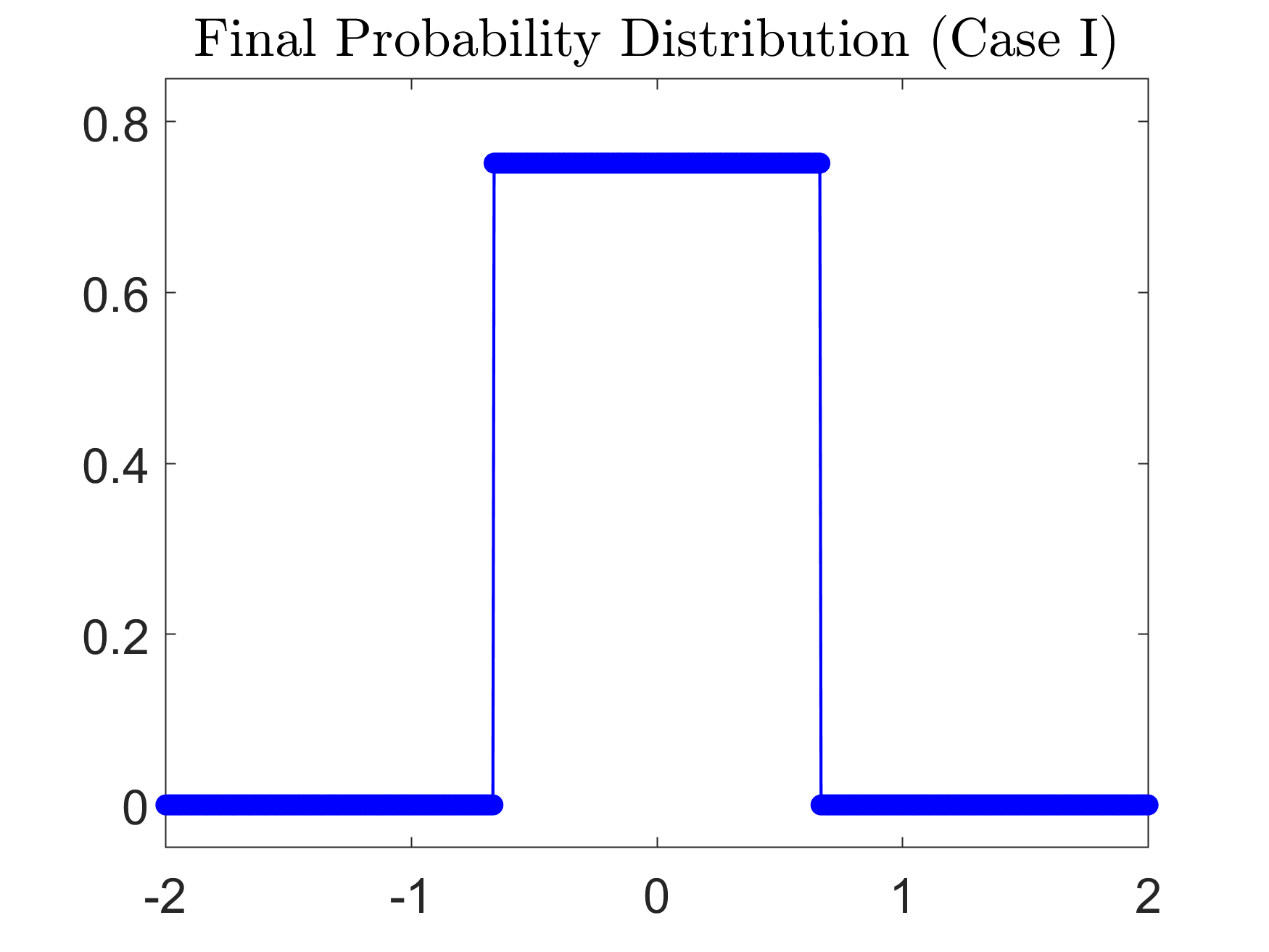}\hspace*{1cm}
            \includegraphics[trim=0.9cm 0.4cm 1.2cm 0.2cm, clip, width=6.cm]{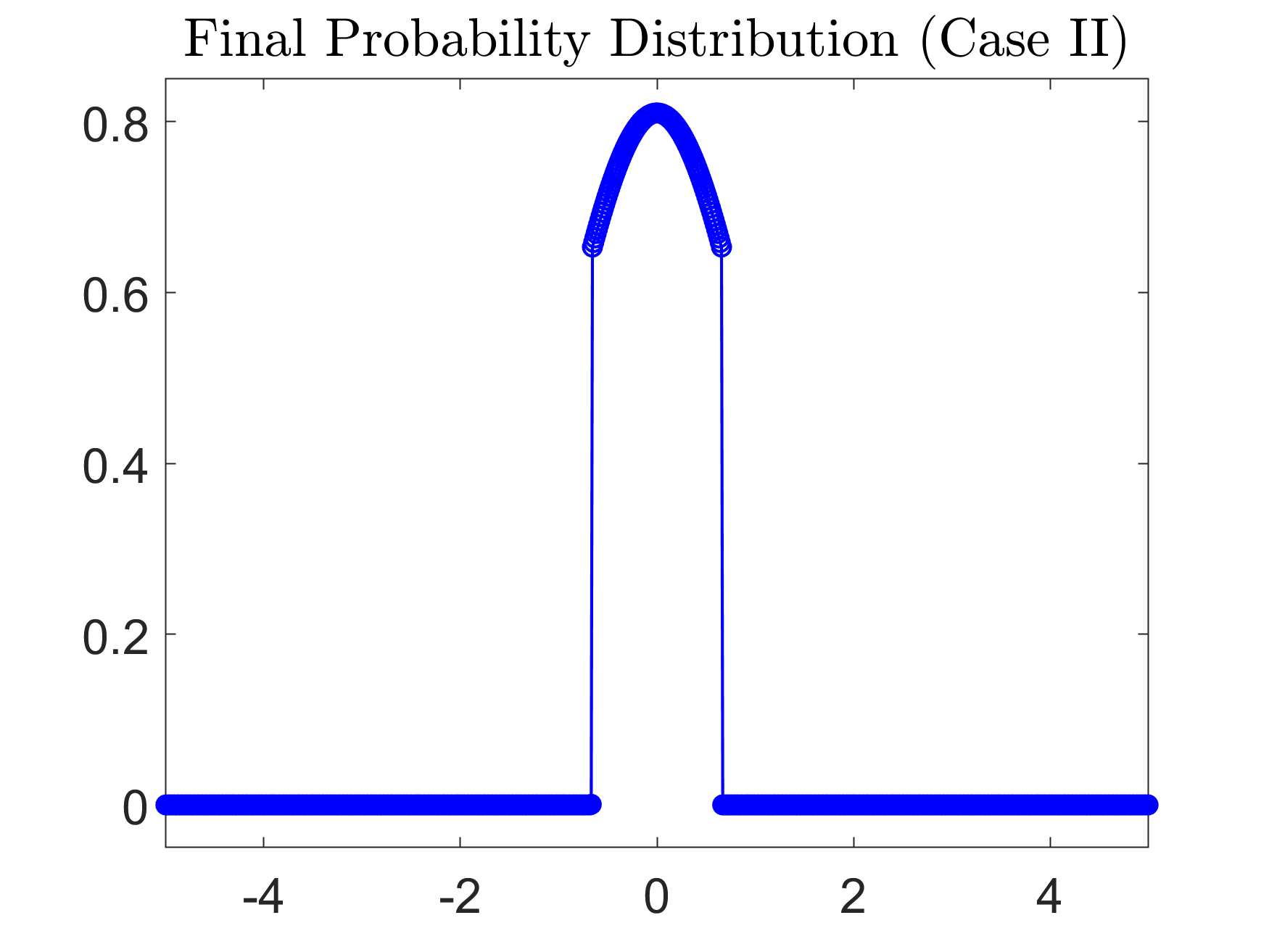}}
\caption{\sf Example 2: Final probability distributions for Cases I (left) and II (right) for the boundary conditions \eref{3.4a}.\label{fig7}}
\end{figure}

\subsection{Nonlinear Deterministic Examples}
In this section, we consider two nonlinear deterministic examples --- the nonlinear Saint-Venant system and Burgers equation and compute the stable domains. It is well known that determining stability domains for nonlinear systems is challenging, and we now show that the proposed method provides a feasible route to identify them.

\paragraph*{Example 3---Nonlinear Saint-Venant System.}
In the third example taken from \cite{CHK_Stability}, we consider the same setting as in Example 2, but with different equations for the perturbations:
\begin{equation}\label{5.5aa}
 \begin{pmatrix}
	u^{(1)} \\
	u^{(2)}
 \end{pmatrix} _t+ \begin{pmatrix}
	\xbar v +\delta v + \sqrt{g \big(\xbar h+\delta h\big)} & 0  \\
	0 & \xbar v+\delta v - \sqrt{g \big(\xbar h+\delta h\big)}
 \end{pmatrix}\begin{pmatrix}
	u^{(1)}\\
	u^{(2)}
 \end{pmatrix} _x=\begin{pmatrix}
	0 \\
	0 
 \end{pmatrix},
\end{equation}
where $u^{(1)}$ and $u^{(2)}$ are defined by \eref{4.5a}. Notice that the system \eref{5.5aa} is nonlinear.

We compute the numerical results on a uniform mesh with $N_x=100$ and $N_\kappa=800$. The values of the probability distributions are updated based on the discrete energy \eref{3.1}, and the obtained numerical results are presented in Figure \ref{fig9}. One can clearly see that the stability region of the boundary control parameter is approximately $(-1,1)$, slightly different from the one reported in Example~2.

\begin{figure}[ht!]
\centerline{\includegraphics[trim=0.9cm 0.4cm 1.2cm 0.2cm, clip, width=6.cm]{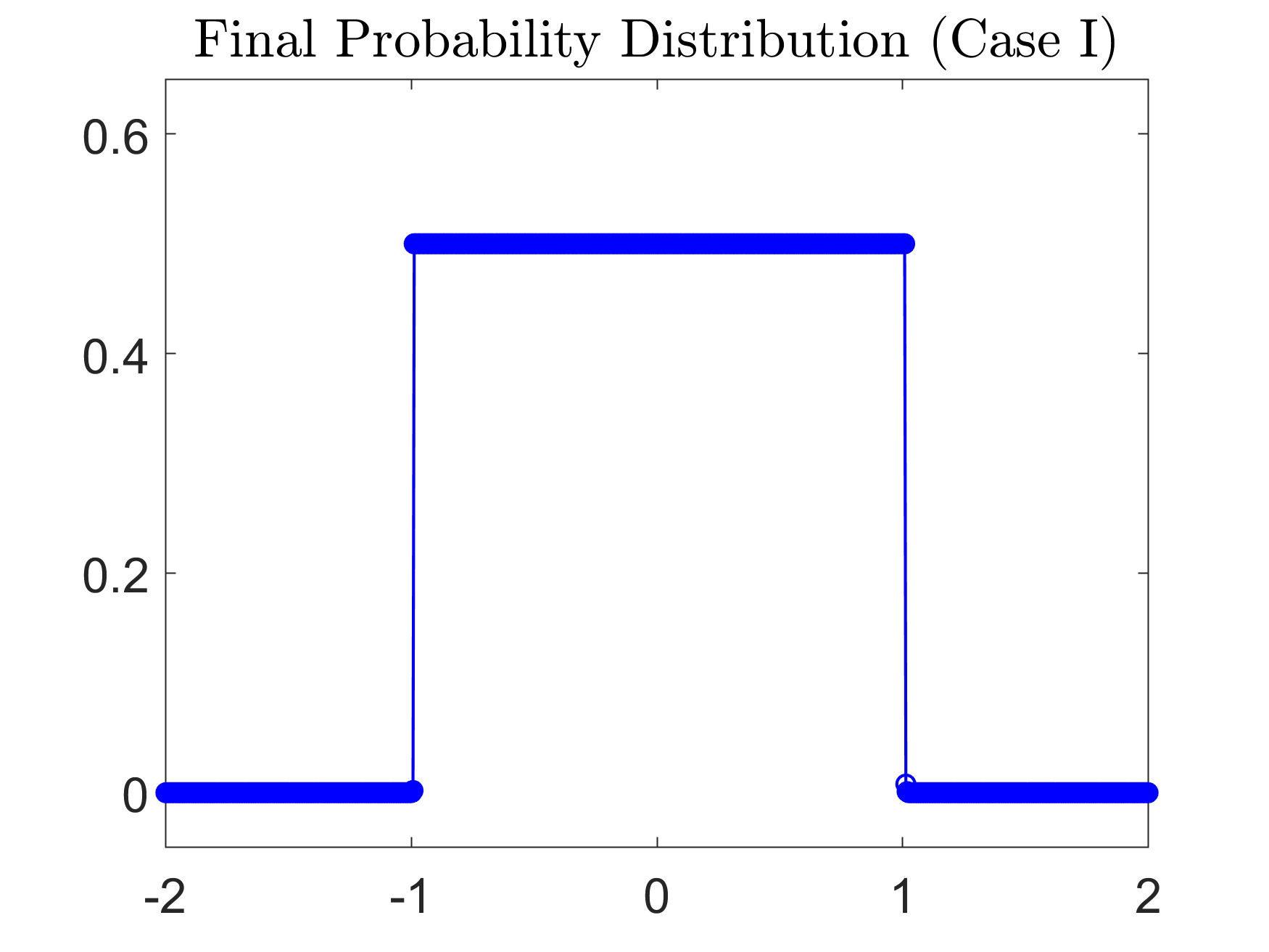}\hspace*{1cm}
            \includegraphics[trim=0.9cm 0.4cm 1.2cm 0.2cm, clip, width=6.cm]{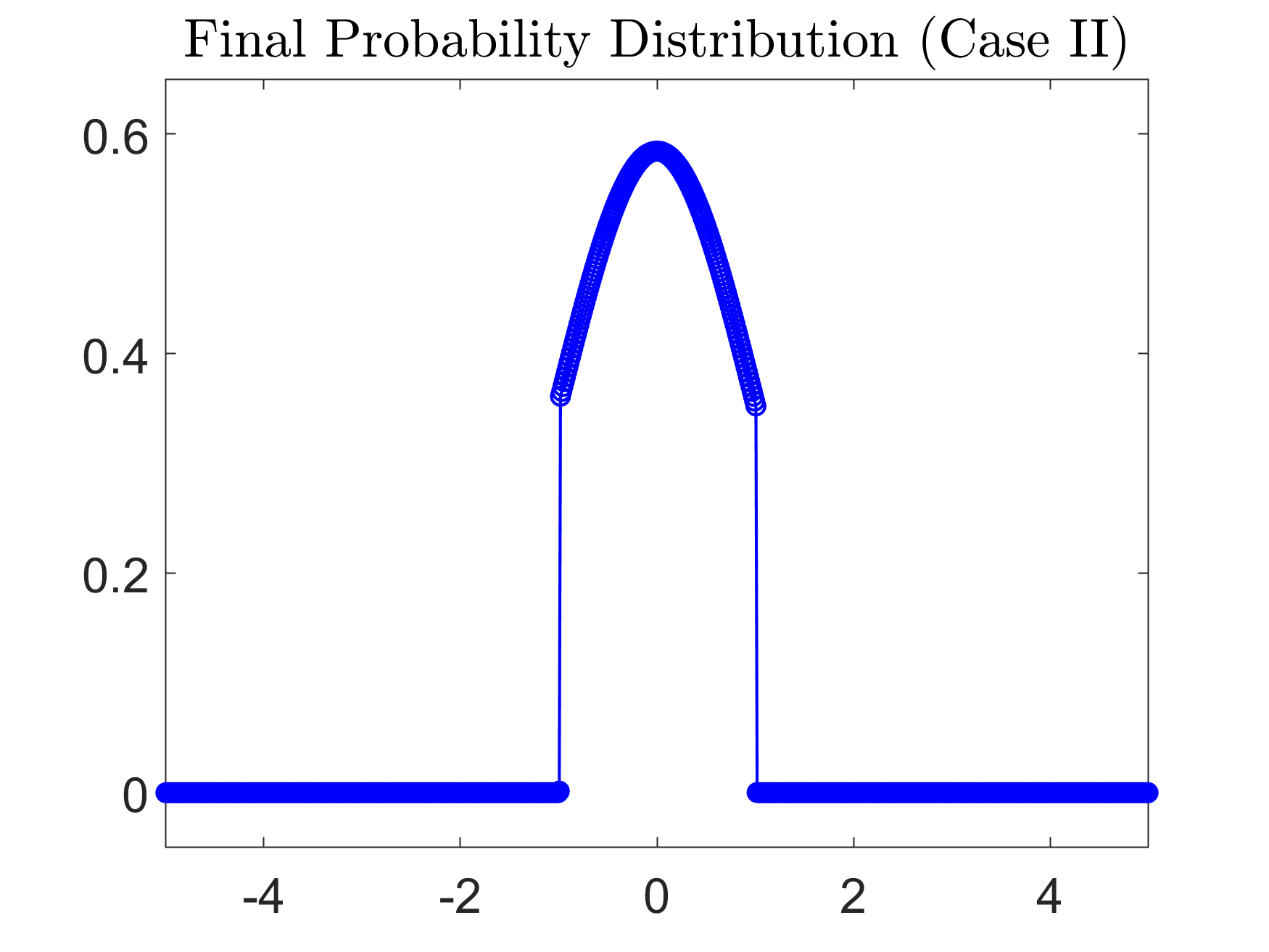}}
\caption{\sf Example 3: Final probability distributions for Cases I (left) and II (right).\label{fig9}}
\end{figure}

\paragraph{Example 4---Burgers Equation.} In this example, we consider the Burgers equation
\begin{equation}\label{3.5}
  u_t +\Big(\frac{u^2}{2} \Big)_x= 0
\end{equation}
with the initial conditions:
$$
\begin{aligned}
&{\rm Case \,\,I:} \quad u(x,0)=\begin{cases}
          0.3, & \mbox{if } x<0.3, \\
          0.2, & \mbox{if } 0.3 \le x \le 0.7,  \\
         -0.1, & \mbox{otherwise},
       \end{cases}\\[1.ex] 
&{\rm Case \,\,II:} \quad u(x,0)=\begin{cases}
          0.1, & \mbox{if } x<0.3, \\
          0.2, & \mbox{if } 0.3 \le x \le 0.7,  \\
          0.1, & \mbox{otherwise},
       \end{cases}
       \end{aligned}
$$
subject to the following boundary conditions:
\begin{equation}\label{3.6}
\begin{cases}
  u(0-,t)=\kappa u(1-,t), & \mbox{if $u(0+,t)>0$,} \\
  u(0-,t)= u(0+,t),       & \mbox{otherwise}, 
\end{cases}\qquad 
\begin{cases}
  u(1+,t)=\kappa u(0+,t), & \mbox{if $u(1-,t)<0$,} \\
  u(1+,t)= u(1-,t),       & \mbox{otherwise}. 
\end{cases}
\end{equation}

We compute the numerical results for Cases I and II on a uniform mesh with $N_x = 200$ and $N_\kappa = 400$, with the initial probability distribution defined on $\kappa \in [-2,2]$ by ${\cal P}(\kappa) = 1/4$. The values of the probability distributions are updated based on the discrete energy 
\begin{equation*}
  {\cal L}(\xbar u)=\dx \sum_{j=1}^{N_x} \xbar u^2_j,
\end{equation*}
and the obtained numerical results are presented in Figure \ref{fig11}. One can see that for the Burgers equation with the boundary conditions defined in \eref{3.6}, the stable domain of the parameter $\kappa$ depends on the initial conditions, and the stable domains are $(-1, 2)$ and $(-2, 1)$ for Cases I and II, respectively.

\begin{figure}[ht!]
\centerline{\includegraphics[trim=0.9cm 0.4cm 1.2cm 0.2cm, clip, width=6.cm]{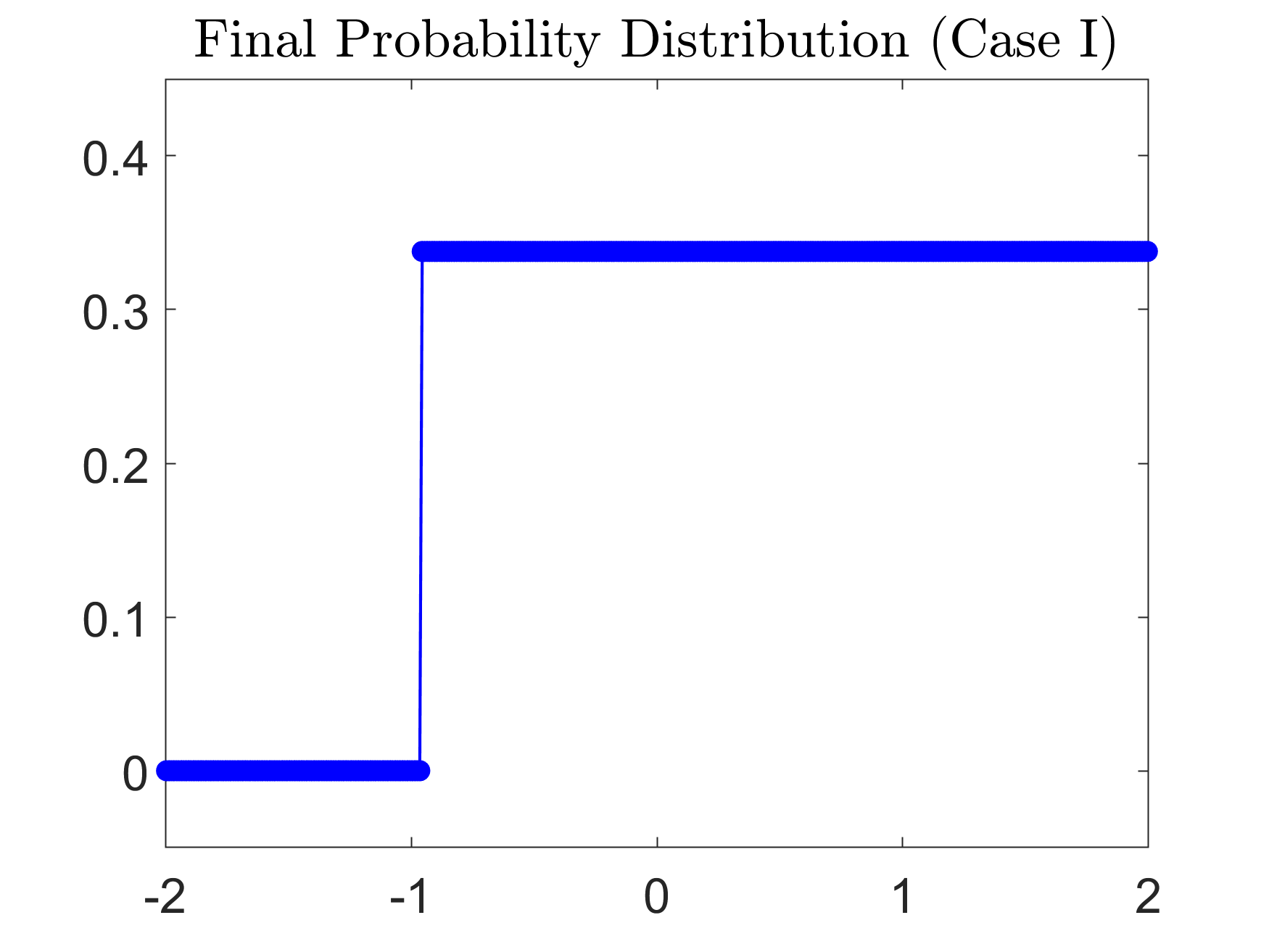}\hspace*{1cm}
            \includegraphics[trim=0.9cm 0.4cm 1.2cm 0.2cm, clip, width=6.cm]{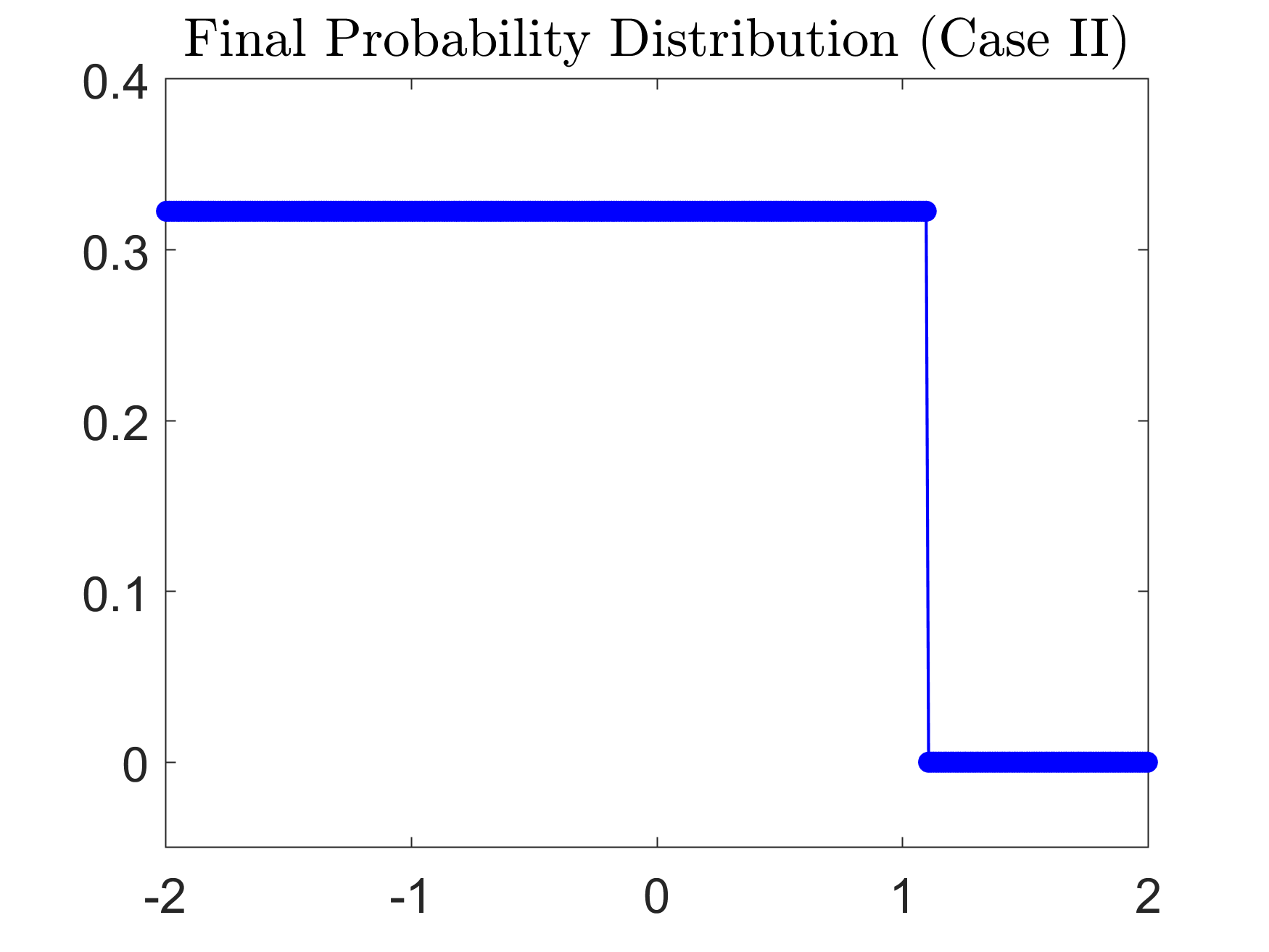}}
\caption{\sf Example 4: Final probability distributions for Cases I (left) and II (right).\label{fig11}}
\end{figure}

We also plot the three-dimensional numerical results for Cases I and II over time in Figure \ref{fig13}, where one can see that the stable domains are $(-1, 2)$ and $(-2, 1)$ for the Cases I and II, respectively, which further supports the proposed methodology for the two cases.

\begin{figure}[ht!]
\centerline{\includegraphics[trim=0.1cm 0.2cm 0.8cm 0.2cm, clip, width=7.cm]{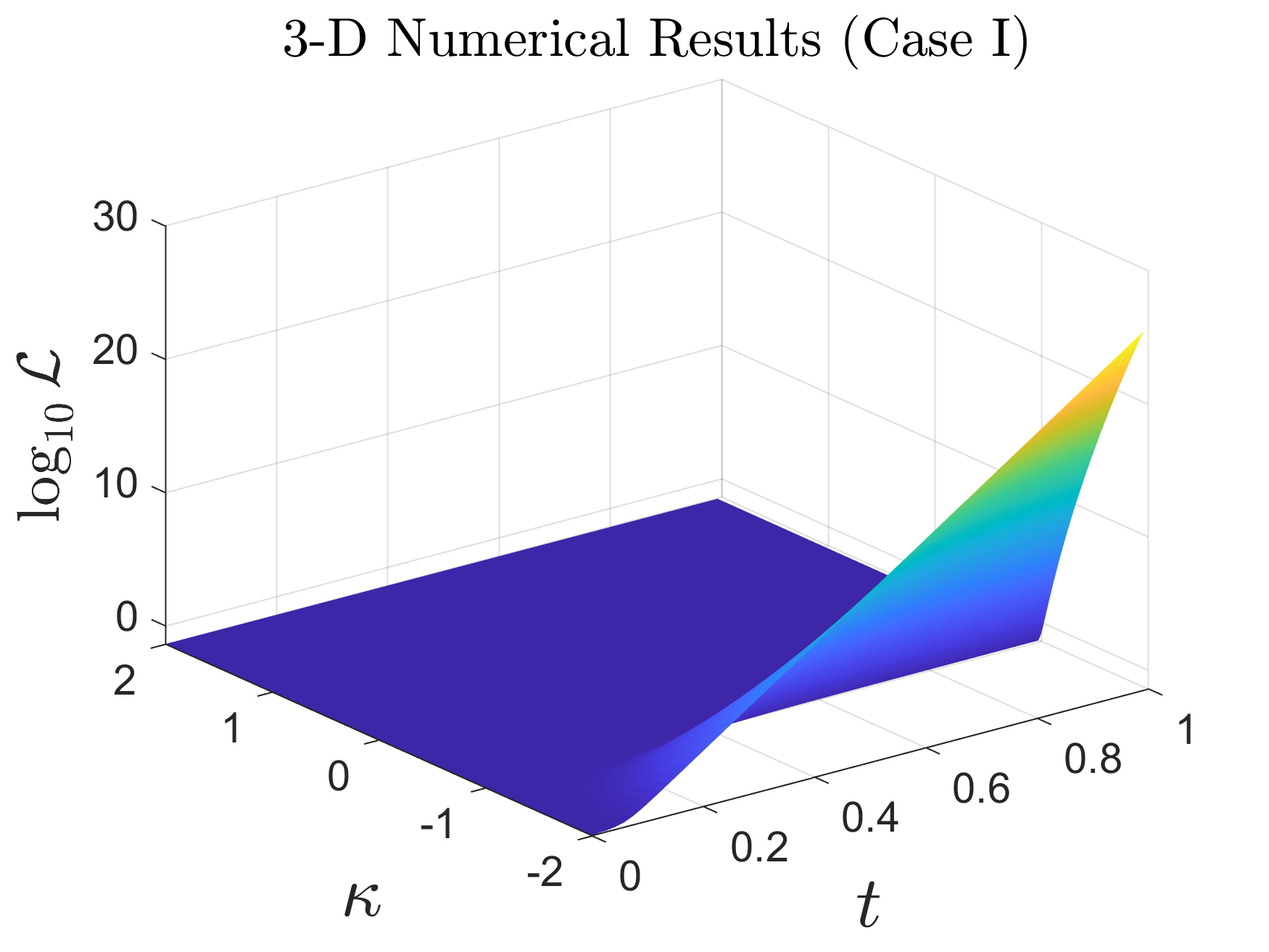}\hspace*{1cm}
            \includegraphics[trim=0.1cm 0.2cm 0.8cm 0.2cm, clip, width=7.cm]{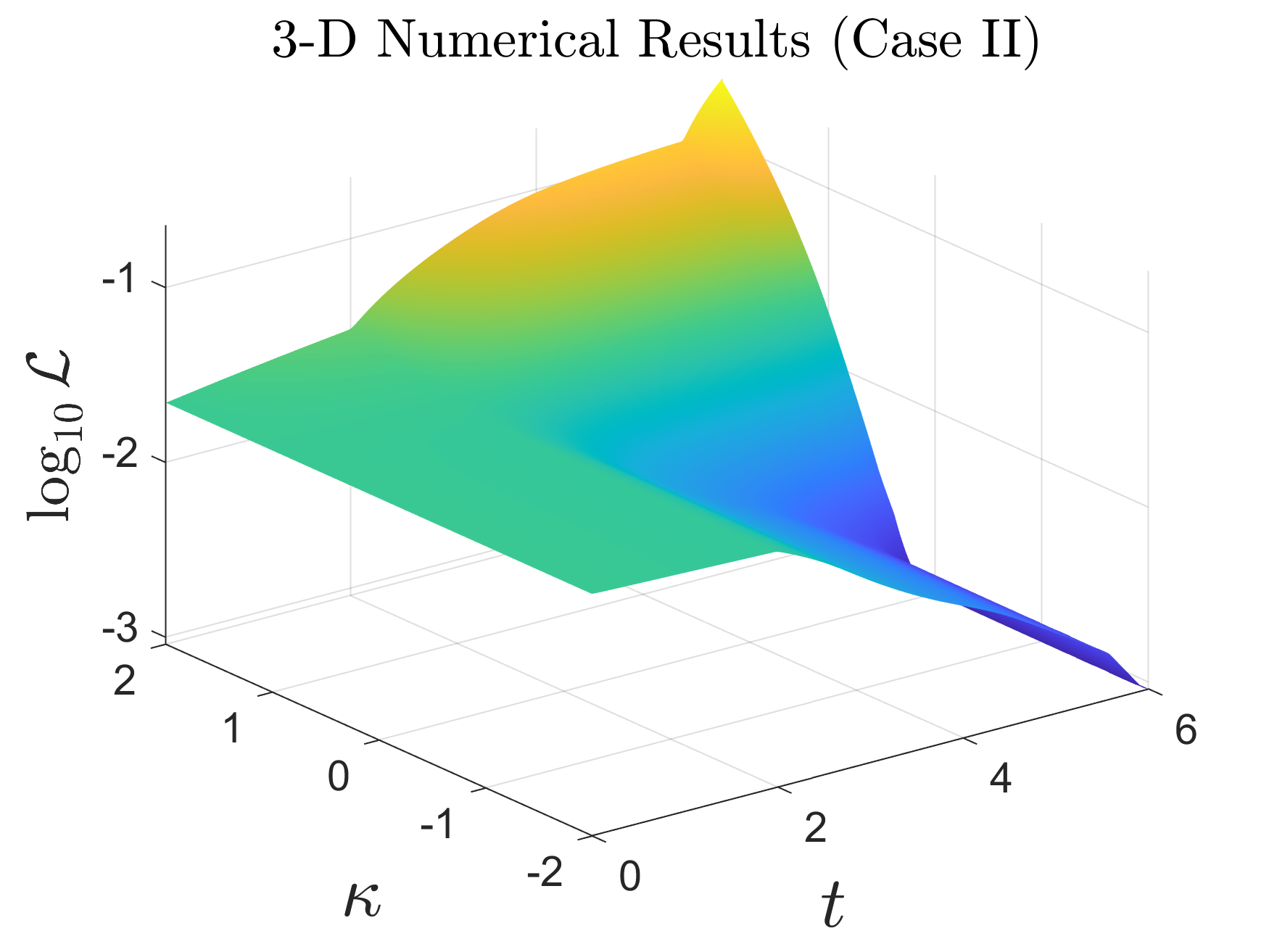}}
\caption{\sf Example 4: Evolution of $\log_{10}{\cal L}$ for Cases I (left) and II (right).\label{fig13}}
\end{figure}

\paragraph{Example 5--- 2-D Burgers Equation.} In this example, we consider the 2-D Burgers equation
\begin{equation*}
  u_t +\Big(\frac{u^2}{2} \Big)_x+\Big(\frac{u^2}{2} \Big)_y= -0.2 u
\end{equation*}
with the initial conditions:
$$
\begin{aligned}
&{\rm Case \,\,I:} \quad u(x,y,0)=\begin{cases}
          0.3, & \mbox{if } x<0.3, \\
          0.2, & \mbox{if } 0.3 \le x \le 0.7,  \\
         -0.1, & \mbox{otherwise},
       \end{cases}\\[1.ex] 
&{\rm Case \,\,II:} \quad u(x,y,0)=\begin{cases}
          0.1, & \mbox{if } x<0.3, \\
          0.2, & \mbox{if } 0.3 \le x \le 0.7,  \\
          0.1, & \mbox{otherwise},
       \end{cases}
       \end{aligned}
$$
with the following boundary conditions:
\begin{equation}\label{3.6a}
\hspace{-0.1cm}\begin{aligned}
&\begin{cases}
  u(0-,y,t)=\kappa u(1-,y,t), & \mbox{$u(0+,y,t)>0$,} \\
  u(0-,y,t)= u(0+,y,t),       & \mbox{$u(0+,y,t)\le 0$}, 
\end{cases}\quad 
\begin{cases}
  u(1+,y,t)=\kappa u(0+,y,t), & \mbox{$u(1-,y,t)<0$,} \\
  u(1+,y,t)= u(1-,y,t),       & \mbox{$u(1-,y,t)\ge0$},  
\end{cases}\\
&\begin{cases}
  u(x,0-,t)=\kappa u(x,1-,t), & \mbox{$u(x,0+,t)>0$,} \\
  u(x,0-,t)= u(x,0+,t),       & \mbox{$u(x,0+,t)\le 0$}, 
\end{cases}\quad 
\begin{cases}
  u(x,1+,t)=\kappa u(x,0+,t), & \mbox{$u(x,1-,t)<0$,} \\
  u(x,1+,t)= u(x,1-,t),       & \mbox{$u(x,1-,t)\ge0$},  
\end{cases}
\end{aligned} 
\end{equation}
subject to the computational domain $[0,1]\times[0,1]$. The initial data are chosen to be $x$-dependent in order to compare with Example 4, while the evolution and feedback treatment are fully two-dimensional. For simplicity, and in order to keep the feedback parameter space one-dimensional, the same feedback parameter $\kappa$ is used on all four sides of the computational domain.

We compute the numerical results for Cases I and II on a uniform mesh with $N_x = N_y=200$ and $N_\kappa = 400$, with the initial probability distribution defined on $\kappa \in [-2,2]$ by ${\cal P}(\kappa) = 1/4$. The values of the probability distributions are updated based on the discrete energy 
\begin{equation*}
  {\cal L}(\xbar u)=\dx \dy \sum_{j=1}^{N_x} \sum_{k=1}^{N_y} \xbar u^2_{j,k},
\end{equation*}
and the obtained numerical results are presented in Figure \ref{fig11a}. One can see that for the Burgers equation with the boundary conditions defined in \eref{3.6a}, the stable domain of the parameter $\kappa$ depends on the initial conditions, and the stable domains are approximately $(-1.4, 2)$ and $(-2, 1.6)$ for Cases I and II, respectively. 

\begin{figure}[ht!]
\centerline{\includegraphics[trim=0.9cm 0.4cm 1.2cm 0.2cm, clip, width=6.cm]{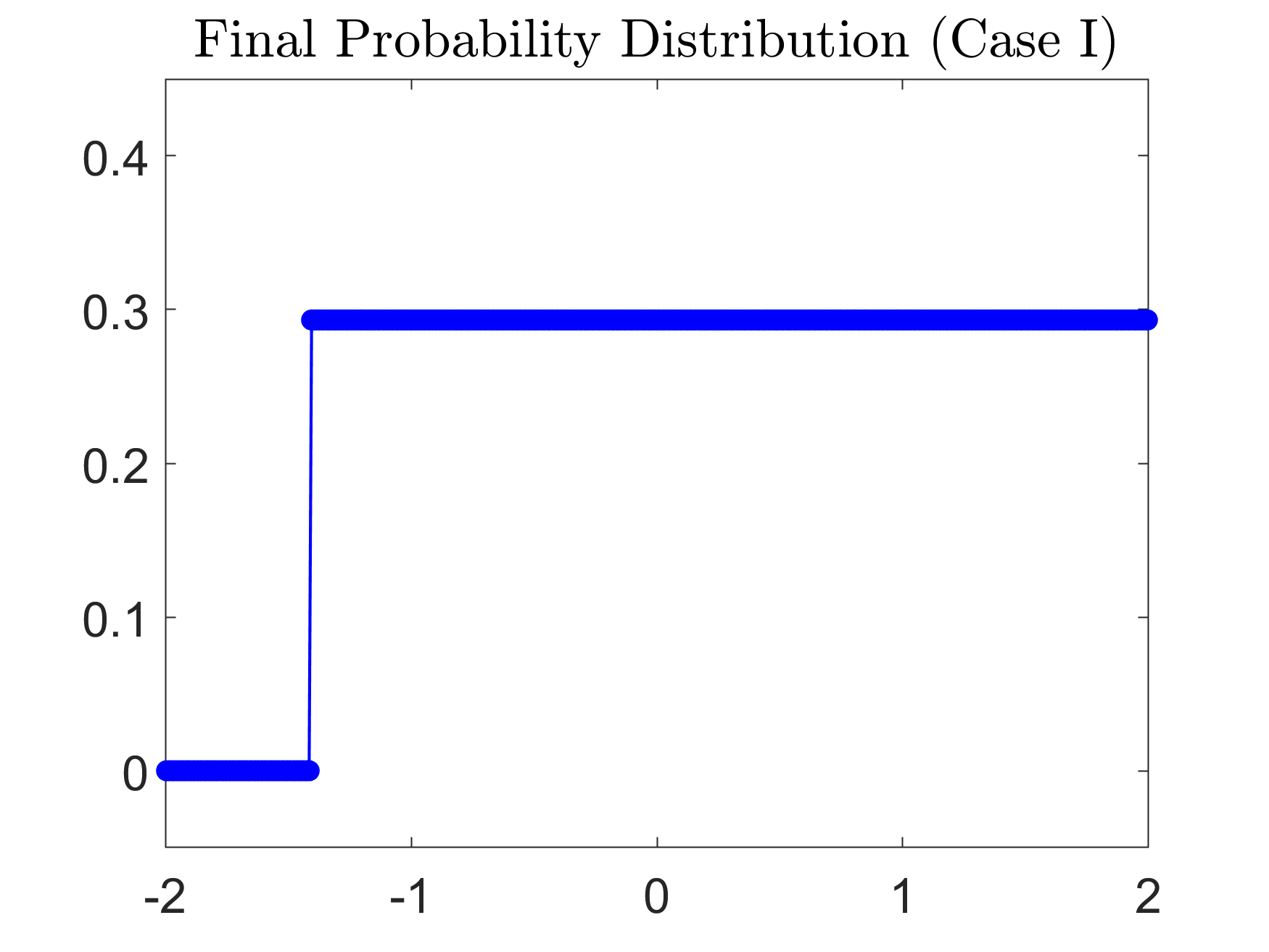}\hspace*{1cm}
            \includegraphics[trim=0.9cm 0.4cm 1.2cm 0.2cm, clip, width=6.cm]{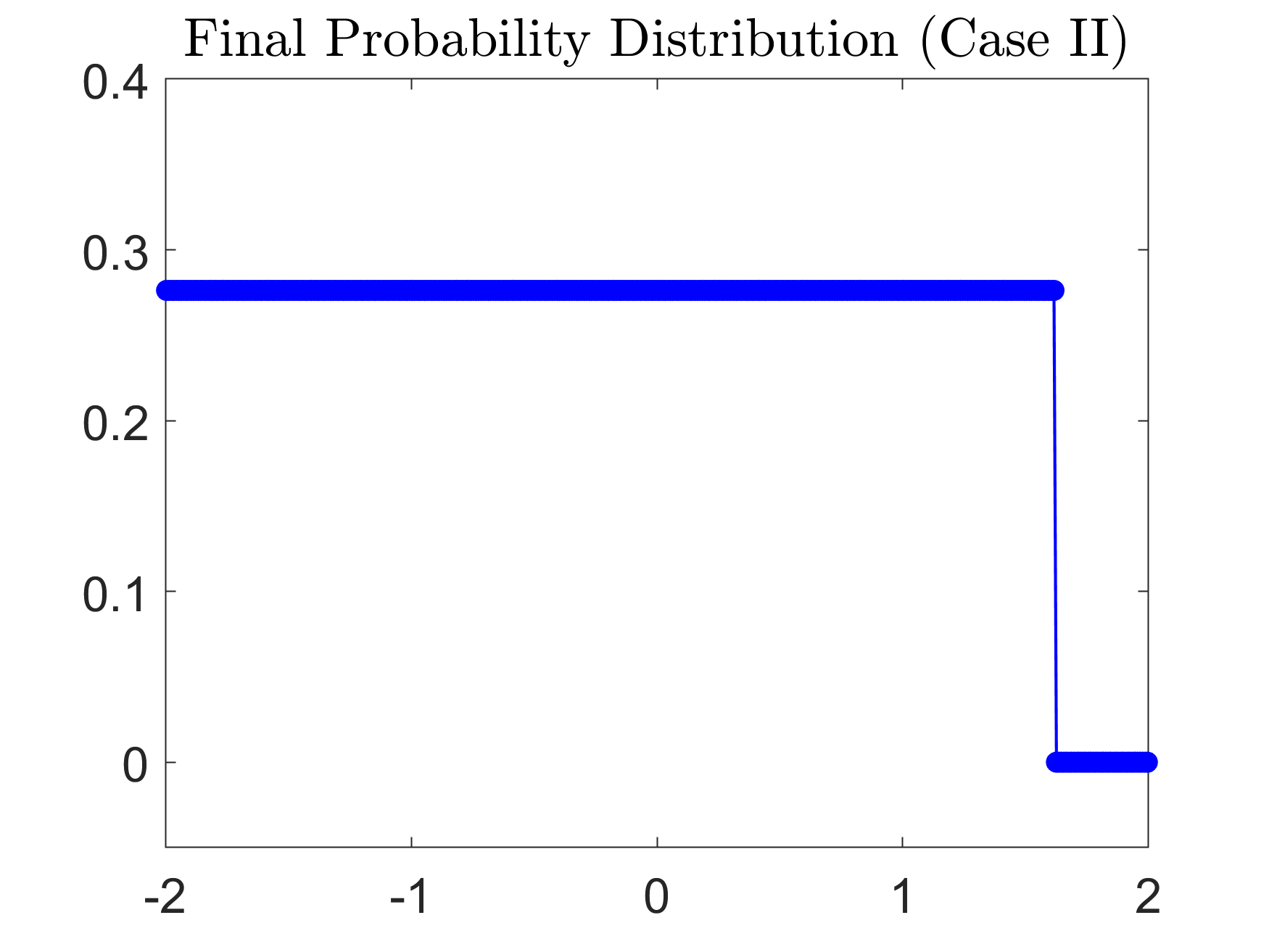}}
\caption{\sf Example 5: Final probability distributions for Cases I (left) and II (right).\label{fig11a}}
\end{figure}

We also plot the three-dimensional numerical results for Cases I and II over time in Figure \ref{fig13a}, where one can see that the stable domains are approximately $(-1.4, 2)$ and $(-2, 1.6)$ for the Cases I and II, respectively, which further supports the proposed methodology for the two cases. The results show that the method can be applied beyond one-dimensional test problems.

\begin{figure}[ht!]
\centerline{\includegraphics[trim=0.1cm 0.2cm 0.8cm 0.2cm, clip, width=7.cm]{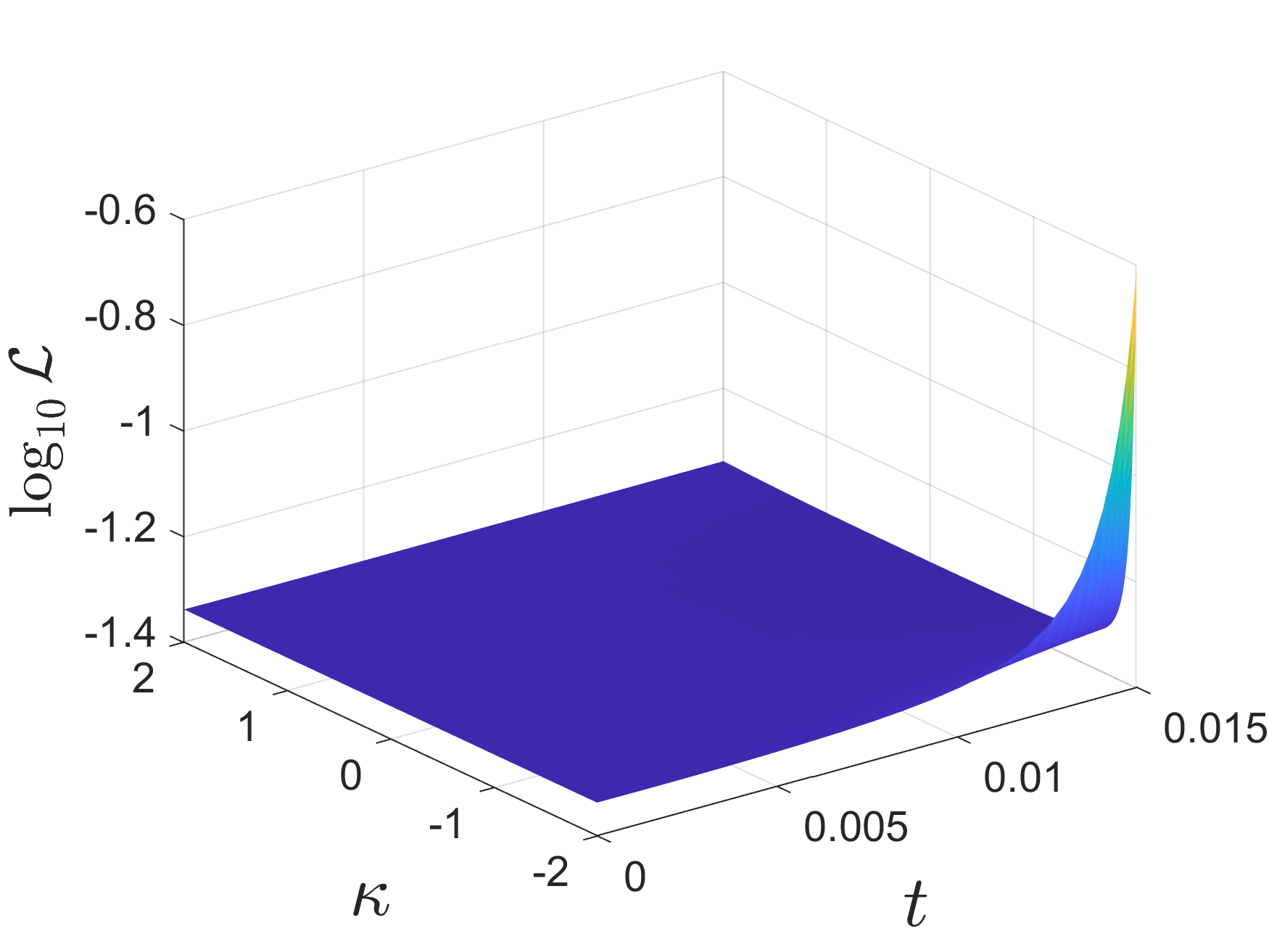}\hspace*{1cm}
            \includegraphics[trim=0.1cm 0.2cm 0.8cm 0.2cm, clip, width=7.cm]{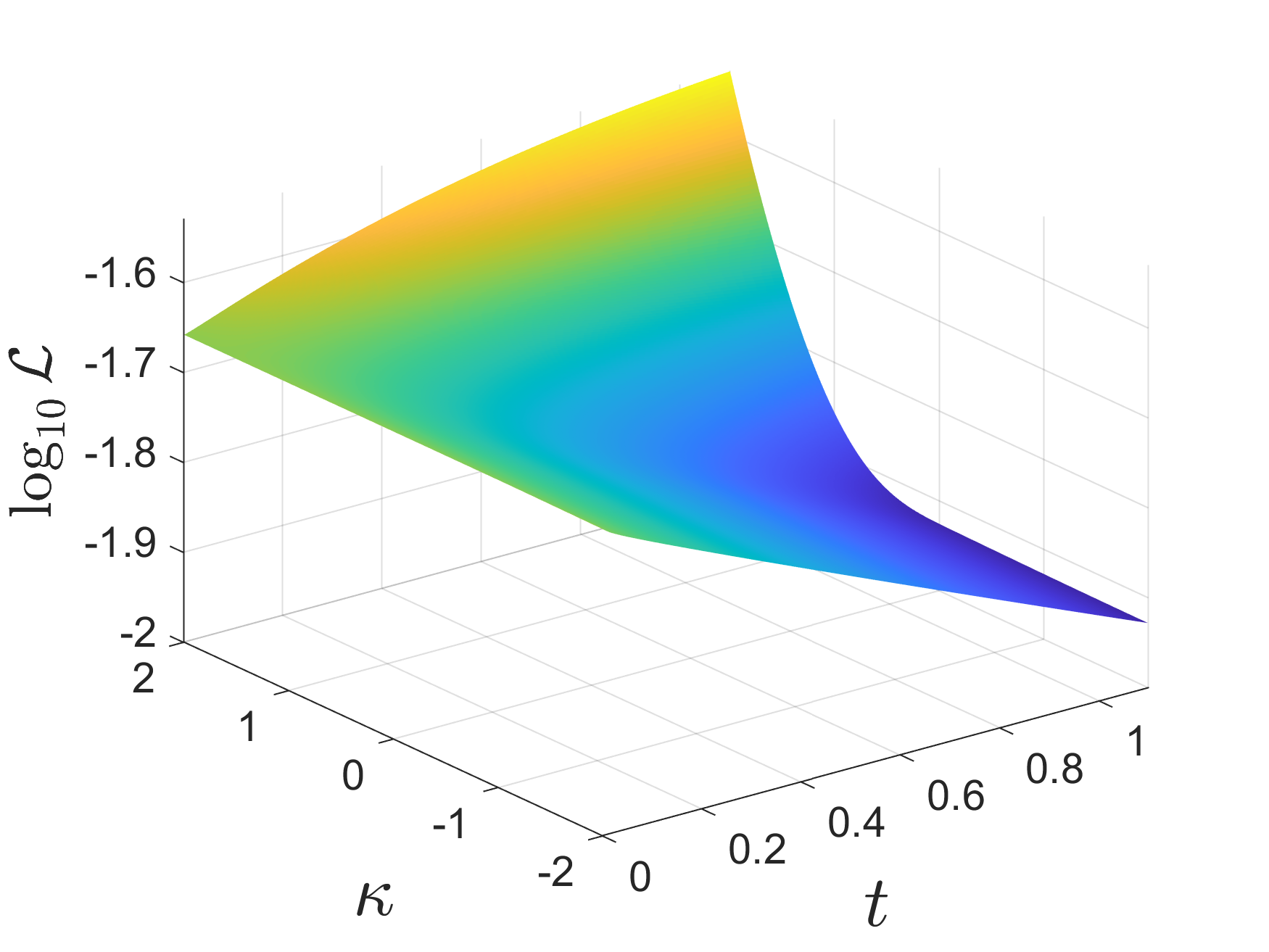}}
\caption{\sf Example 5: Evolution of $\log_{10}{\cal L}$ for Cases I (left) and II (right).\label{fig13a}}
\end{figure}

\subsection{Stochastic Examples}
In this section, we consider two stochastic examples to demonstrate that the proposed method works well for stochastic systems. At the same time, the stable domain for the stochastic linear wave equations computed by the proposed method is consistent with the theoretical validations in \cite{CHK_Stability}.

\paragraph*{Example 6---Stochastic Linear Wave Equations.}
In this example, we consider the linear wave equations \eref{3.1a} with the following initial conditions:
\begin{equation}\label{3.10}
u^{(1)}(x,0)=\frac{1}{4}-\frac{1}{2}\xi, \quad  u^{(2)}(x,0)=-\frac{1}{4}+\frac{1}{2}\xi.    
\end{equation}
The boundary conditions are given by 
\begin{equation}\label{3.10a}
u^{(1)}(0,t,\xi)= \kappa  u^{(1)}(1,t,\xi), \quad u^{(2)}(1,t,\xi)= \kappa  u^{(2)}(0,t,\xi),
\end{equation}
where $\xi$ is the random variable and $\xi \in [-0.5, 0.5]$.

We take the same initial probability distributions of the parameter $\kappa$ as in Example 1 and compute the numerical results on a uniform mesh with $N_x=100$, $N_\kappa=800$, and $N_\xi=100$. The obtained numerical results are presented in Figure \ref{fig14}. Here, the values of the probability distributions are updated based on the corresponding discrete energy 
\begin{equation}\label{3.9a}
{\cal L}\big(\xbar u^{(1)}, \xbar u^{(2)}\big)=\dx \dxi\sum_{j=1}^{N_x}\sum_{k=1}^{N_\xi} \bigg[  \Big( u^{(1)}_{j,k} \Big)^2+  \Big( u^{(2)}_{j,k} \Big)^2 \bigg].
\end{equation}
As one can see, the stability region of the boundary control parameter is approximately $(-1,1)$, which is consistent with the one proved in \cite[Theorem 4.1]{CHK_Stability}, which once again demonstrates the validity and robustness of the proposed method.

\begin{figure}[ht!]
\centerline{\includegraphics[trim=0.9cm 0.4cm 1.2cm 0.2cm, clip, width=6.cm]{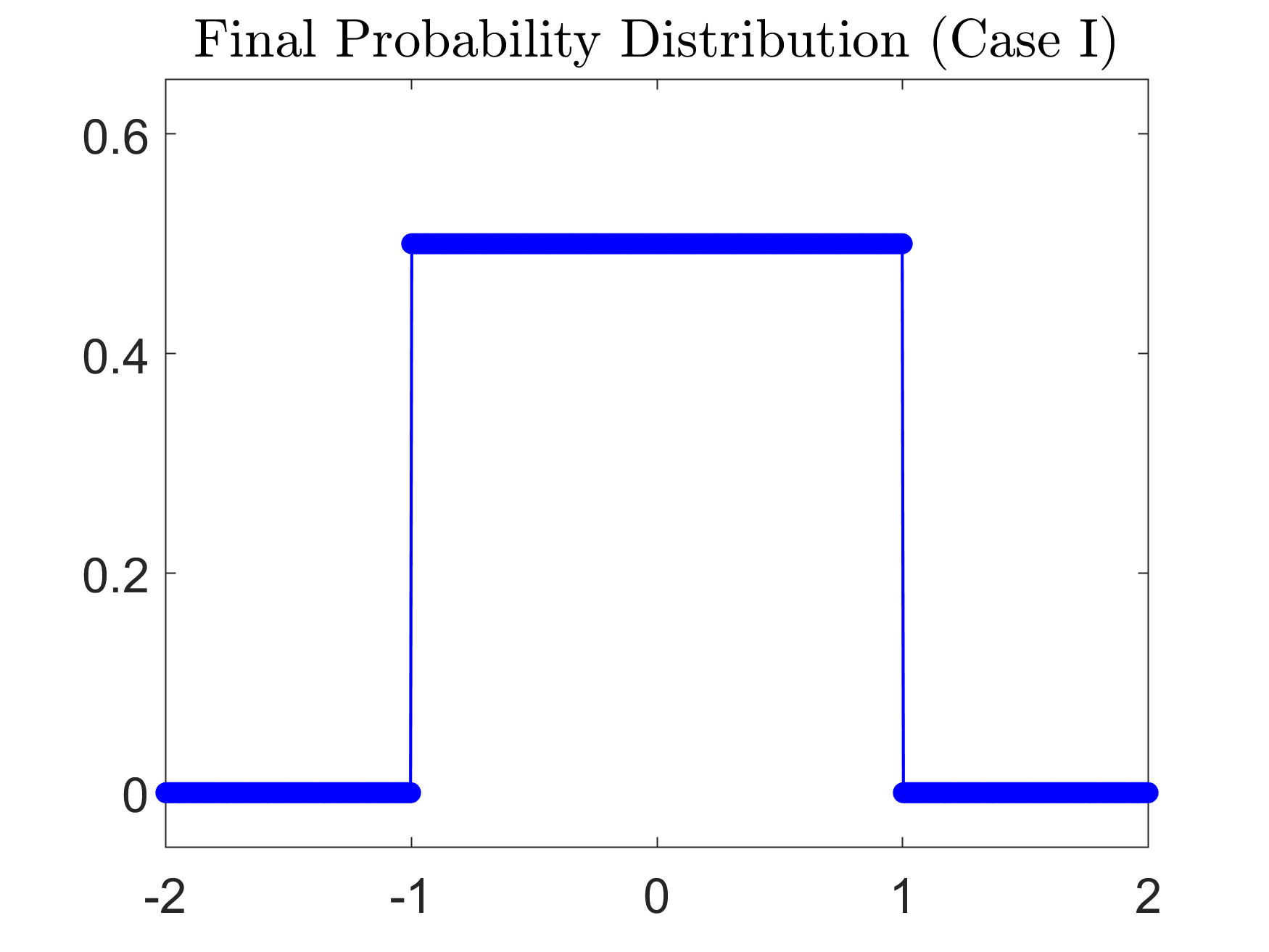}\hspace*{1cm}
            \includegraphics[trim=0.9cm 0.4cm 1.2cm 0.2cm, clip, width=6.cm]{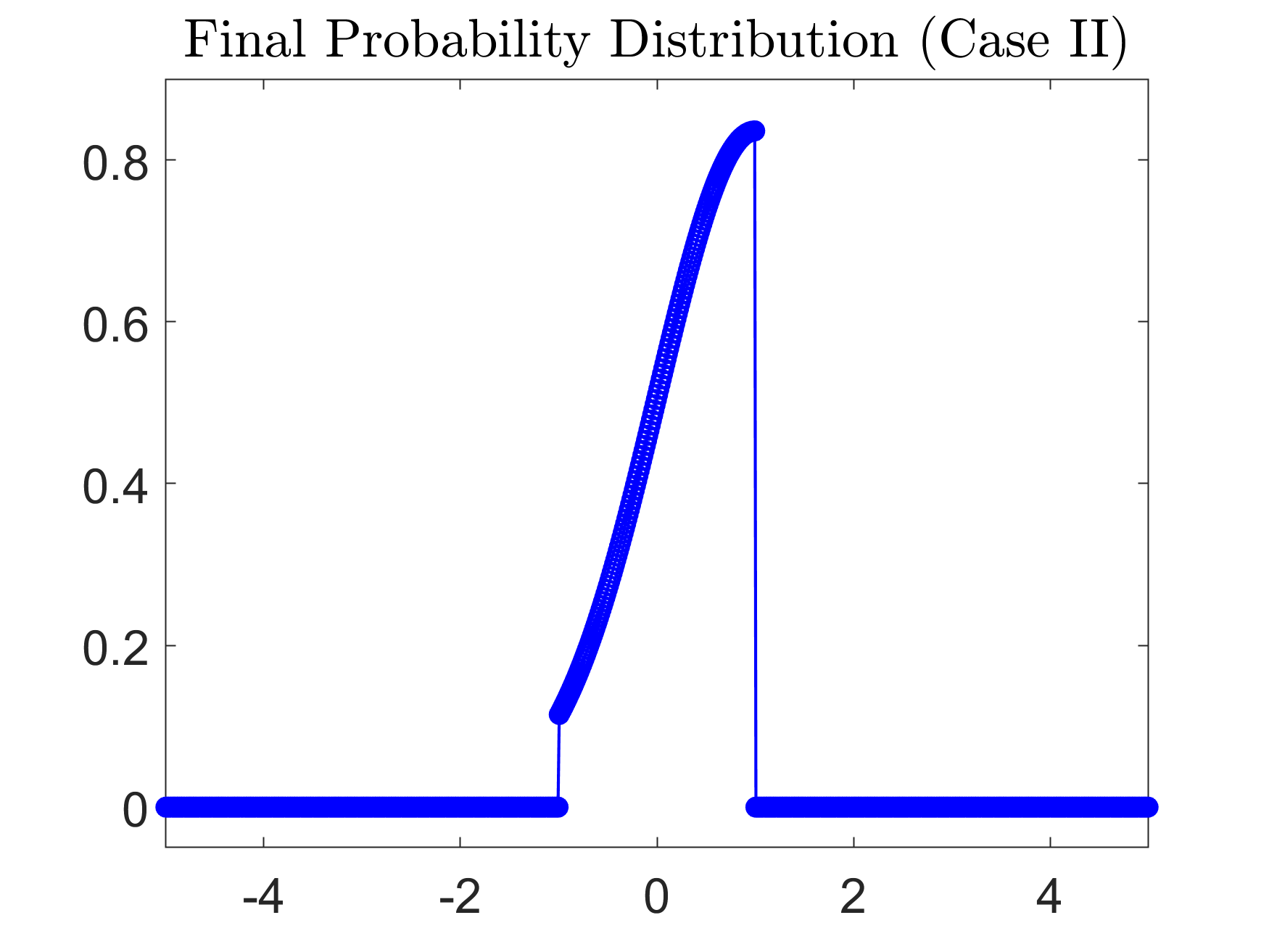}}
\caption{\sf Example 6: Final probability distributions for Cases I (left) and II (right).\label{fig14}}
\end{figure}

\begin{rmk}
In this example, the numerical solutions of the system \eref{3.1a}, \eref{3.10}--\eref{3.10a} are also computed by the first-order LLF scheme; see Appendix \ref{appa}.  
\end{rmk}

\paragraph*{Example 7---Stochastic Burgers Equation.}
In this example, we consider the Burgers equation \eref{3.5} with the following initial conditions:
\begin{equation*}
\begin{aligned}
&{\rm Case \,\,I:} \,\, u(x,0,\xi)=\begin{cases}
          0.3+\xi, & \mbox{if } x<0.3, \\
          0.2, & \mbox{if } 0.3 \le x \le 0.7,  \\
         -0.1, & \mbox{otherwise},
       \end{cases}\\[1.ex] 
&{\rm Case \,\,II:} \,\, u(x,0,\xi)=\begin{cases}
          0.1, & \mbox{if } x<0.3, \\
          0.2+0.1\xi, & \mbox{if } 0.3 \le x \le 0.7,  \\
          0.1, & \mbox{otherwise}.
       \end{cases}  
\end{aligned}
\end{equation*}
The boundary conditions are given by 
\begin{equation*}
\begin{cases}
  u(0-,t,\xi)=\kappa u(1-,t,\xi), & \mbox{if $u(0+,t,\xi)>0$,} \\
  u(0-,t,\xi)= u(0+,t,\xi),       & \mbox{otherwise}, 
\end{cases}\qquad 
\begin{cases}
  u(1+,t,\xi)=\kappa u(0+,t,\xi), & \mbox{if $u(1-,t,\xi)<0$,} \\
  u(1+,t,\xi)= u(1-,t,\xi),       & \mbox{otherwise},
\end{cases}
\end{equation*}
where $\xi$ is the random variable and $\xi \in [-0.5, 0.5]$.

We take the same initial probability distributions as in Example 4 and compute the numerical results for the Cases I and II on a uniform mesh with $N_x=200$, $N_\kappa=400$, and $N_\xi=100$. The values of the probability distributions are updated based on the discrete energy 
\begin{equation*}
  {\cal L}(\xbar u)=\dx \dxi\sum_{j=1}^{N_x}\sum_{k=1}^{N_\xi} u^2_{j,k},
\end{equation*}
and the obtained numerical results are presented in Figure \ref{fig16}. One can see that for the Burgers equation with stochastic initial conditions, the stable domains are 
similar to the ones reported in Example 4.

\begin{figure}[ht!]
\centerline{\includegraphics[trim=0.9cm 0.4cm 1.2cm 0.2cm, clip, width=6.cm]{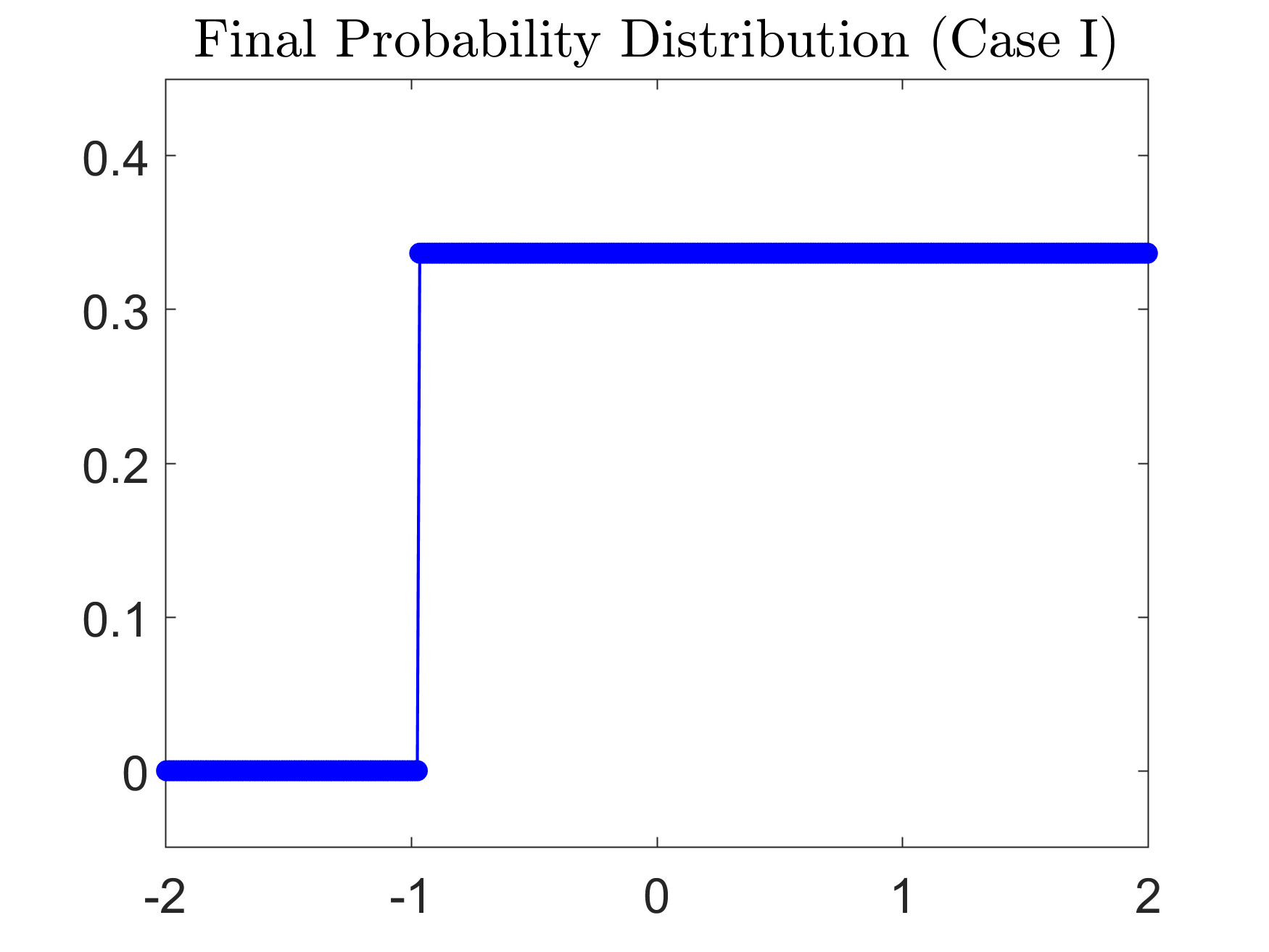}\hspace*{1cm}
            \includegraphics[trim=0.9cm 0.4cm 1.2cm 0.2cm, clip, width=6.cm]{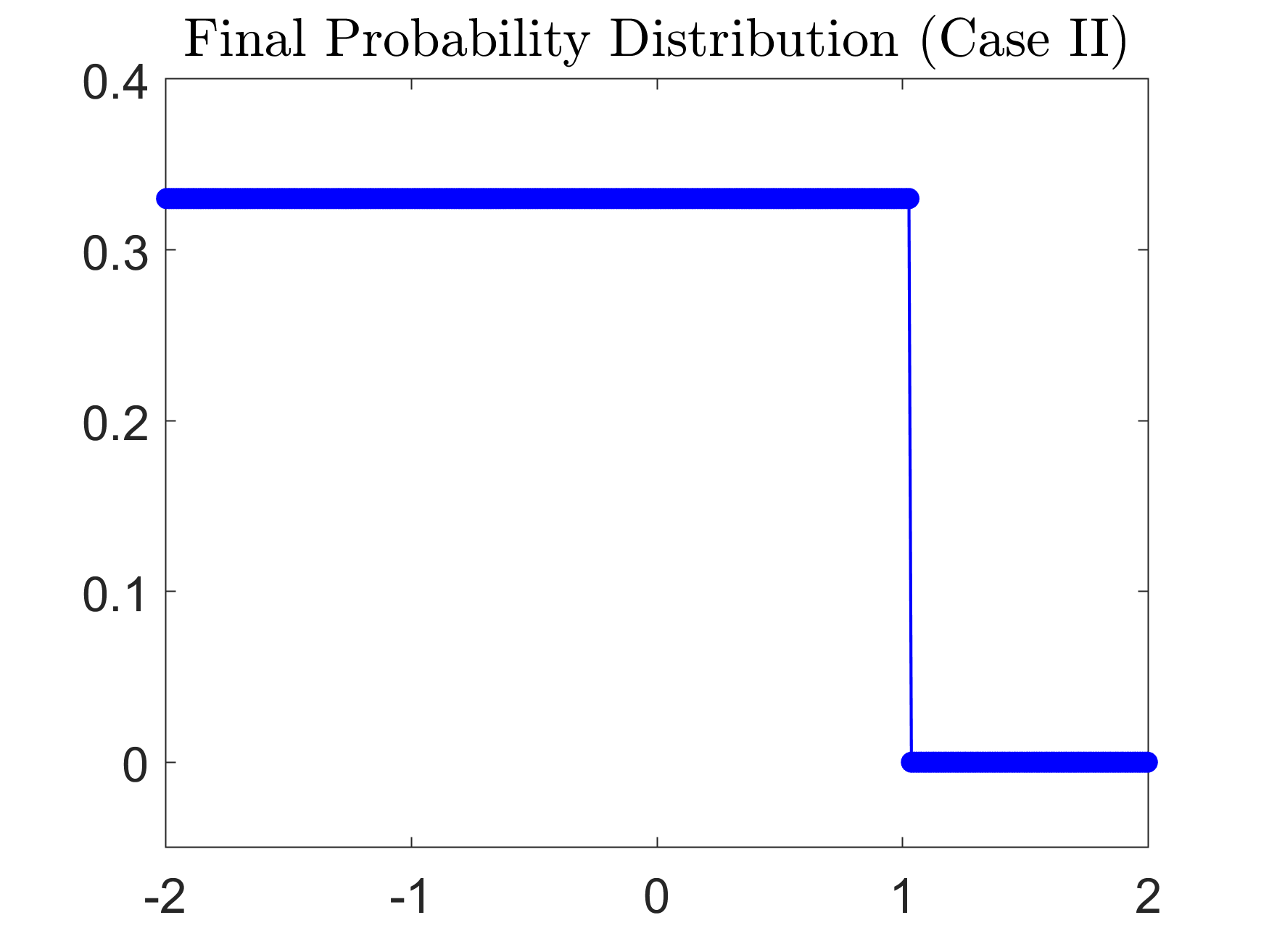}}
\caption{\sf Example 7: Final probability distributions for Cases I (left) and II (right).\label{fig16}}
\end{figure}

\subsection{Nonconservative System}
In this section, we consider a nonconservative system and show that the stability domain computed by the proposed method is consistent with the theoretical results in \cite{CHK_Stability}.

\paragraph*{Example 8---Linearized Saint-Venant System with Source Terms.}
In this example, we consider the boundary damping for the linearized Saint-Venant system with source terms. We take the same setting as in Example 2, but with different equations, which read as
\begin{equation*}
\begin{pmatrix}u^{(1)}\\u^{(2)}\end{pmatrix}_t+\begin{pmatrix}\Lambda_1&0\\0&\Lambda_2\end{pmatrix}
\begin{pmatrix}u^{(1)}\\u^{(2)}\end{pmatrix}_x=-\begin{pmatrix}0.1&0\\0&0.1\end{pmatrix}\begin{pmatrix}u^{(1)}\\u^{(2)}\end{pmatrix},
\end{equation*}
where $u^{(1)}$ and $u^{(2)}$ are defined by \eref{4.5a}.

We compute the numerical results on a uniform mesh with $N_x=100$ and $N_\kappa=800$. The values of the probability distributions are updated based on the discrete energy \eref{3.1}, and the obtained numerical results are presented in Figure \ref{fig24}.  One can clearly see that the stability region of the boundary control parameter is consistent with the theoretical results in \cite{CHK_Stability}.

\begin{figure}[ht!]
\centerline{\includegraphics[trim=0.9cm 0.4cm 1.2cm 0.2cm, clip, width=6.cm]{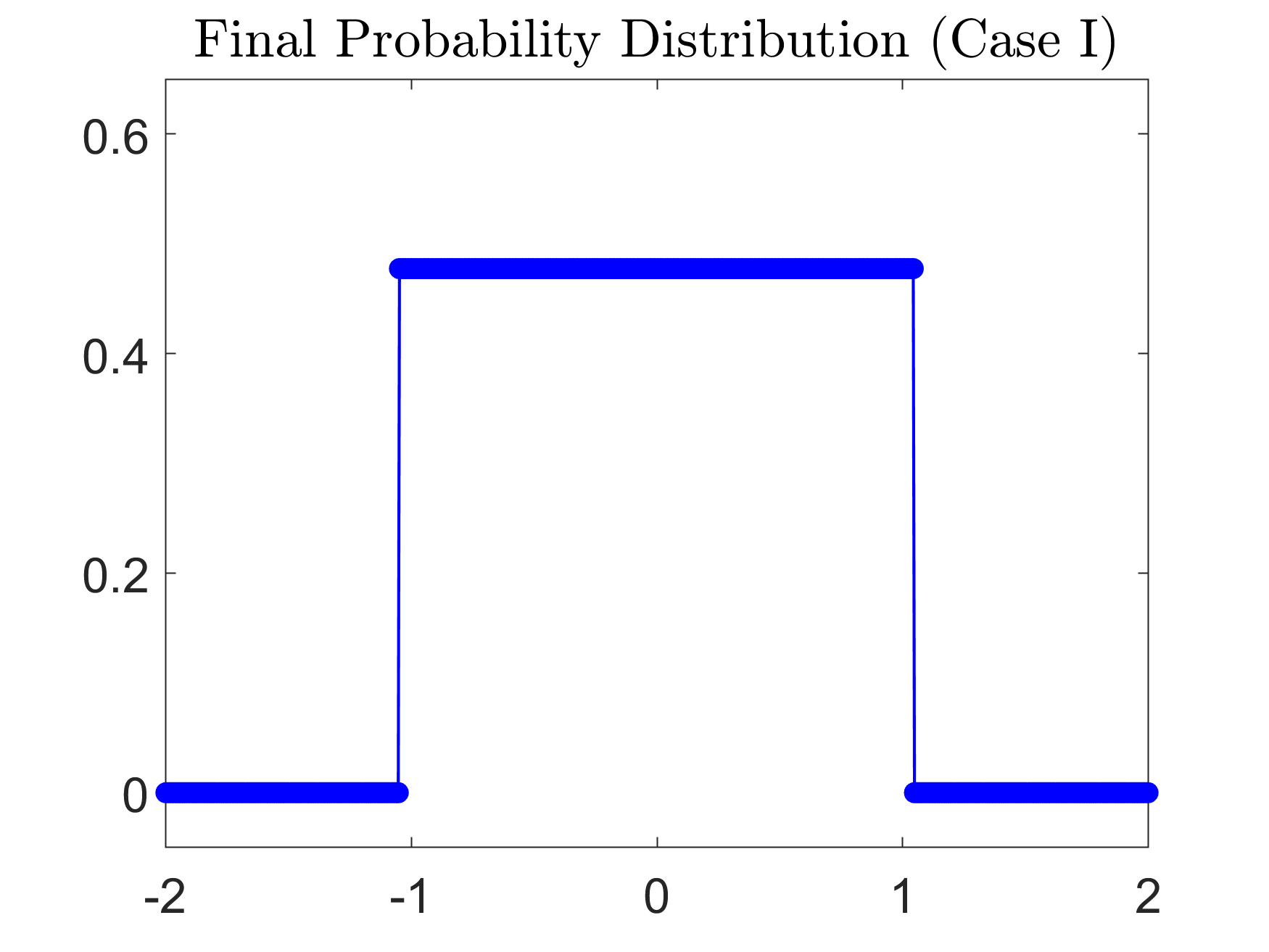}\hspace*{1cm}
            \includegraphics[trim=0.9cm 0.4cm 1.2cm 0.2cm, clip, width=6.cm]{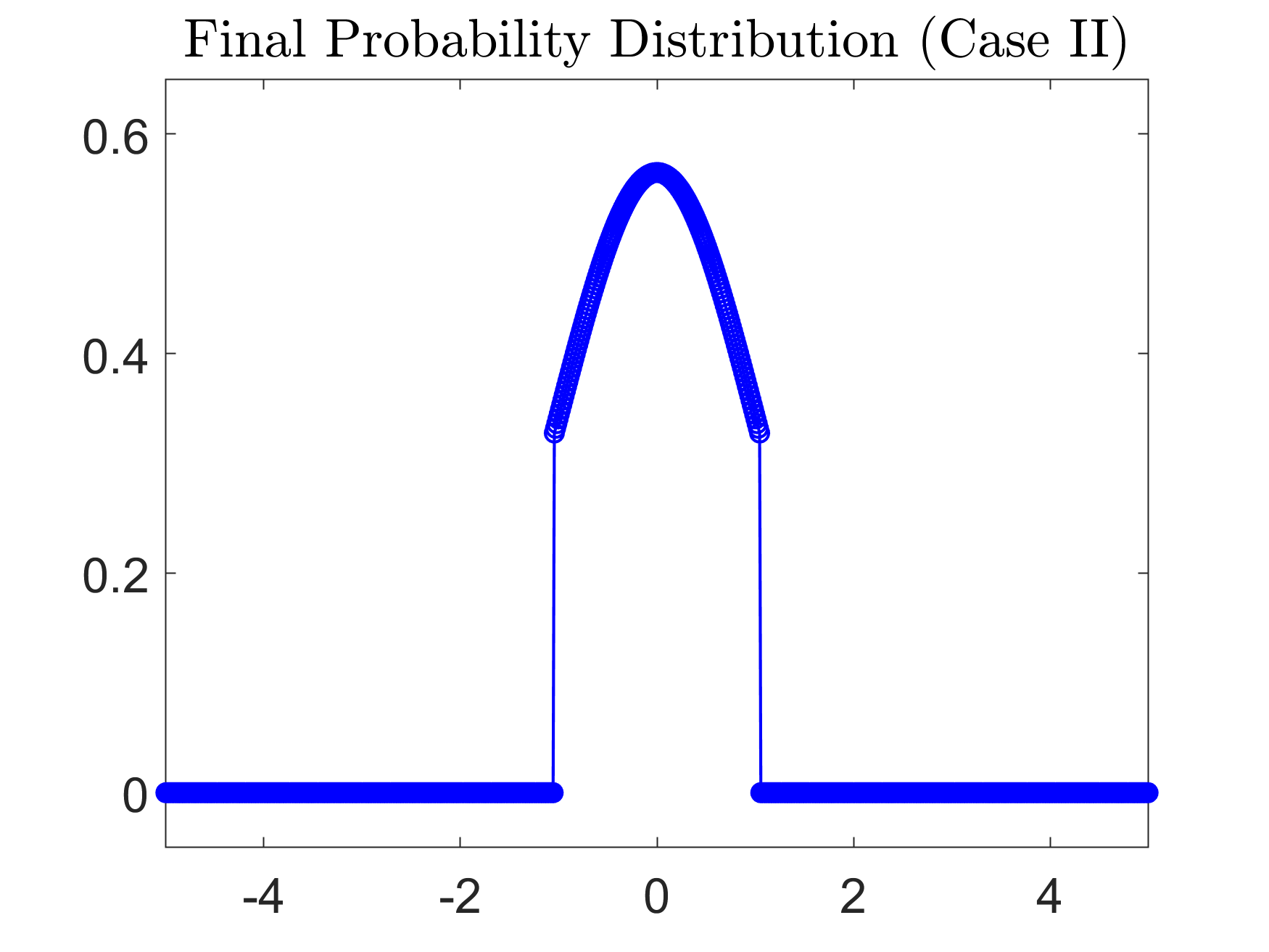}}
\caption{\sf Example 8: Final probability distributions of Case I (left) and II (right).\label{fig24}}
\end{figure}

\paragraph*{Example 9--- Linearized Stochastic Saint-Venant System with Source Terms.}
In this example, we take the same setting as in Example 7, but with the following initial conditions:
\begin{equation*}
	\delta h(x,0,\xi)=\frac{1}{2}\sin(\pi x)(-\frac{1}{2}+\xi), \quad \delta v(x,0,\xi)=\frac{20}{8+\sin(\pi x)(-\frac{1}{2}+\xi)}-\frac{5}{2}, \quad \xi \in [-\frac{1}{2}, \frac{1}{2}],
\end{equation*}
subject to the boundary conditions 
$$
u^{(1)}(0,t,\xi) = \kappa \, u^{(1)}(1,t,\xi), \quad  {\rm and} \quad  u^{(2)}(1,t,\xi)  =\kappa\, u^{(2)}(0,t,\xi),
$$ 
where $\xi$ is the random variable and $\xi \in [-0.5, 0.5]$.

We compute the numerical results on a uniform mesh with $N_x=100$, $N_\kappa=800$, and $N_\xi=200$. The values of the probability distributions are updated based on the discrete energy \eref{3.9a}, and the obtained numerical results are presented in Figure \ref{fig25}.  One can clearly see that the stability region of the boundary control parameter is approximately $(-1,1)$, which is consistent with the theoretical results in \cite{CHK_Stability}.

\begin{figure}[ht!]
\centerline{\includegraphics[trim=0.9cm 0.4cm 1.2cm 0.2cm, clip, width=6.cm]{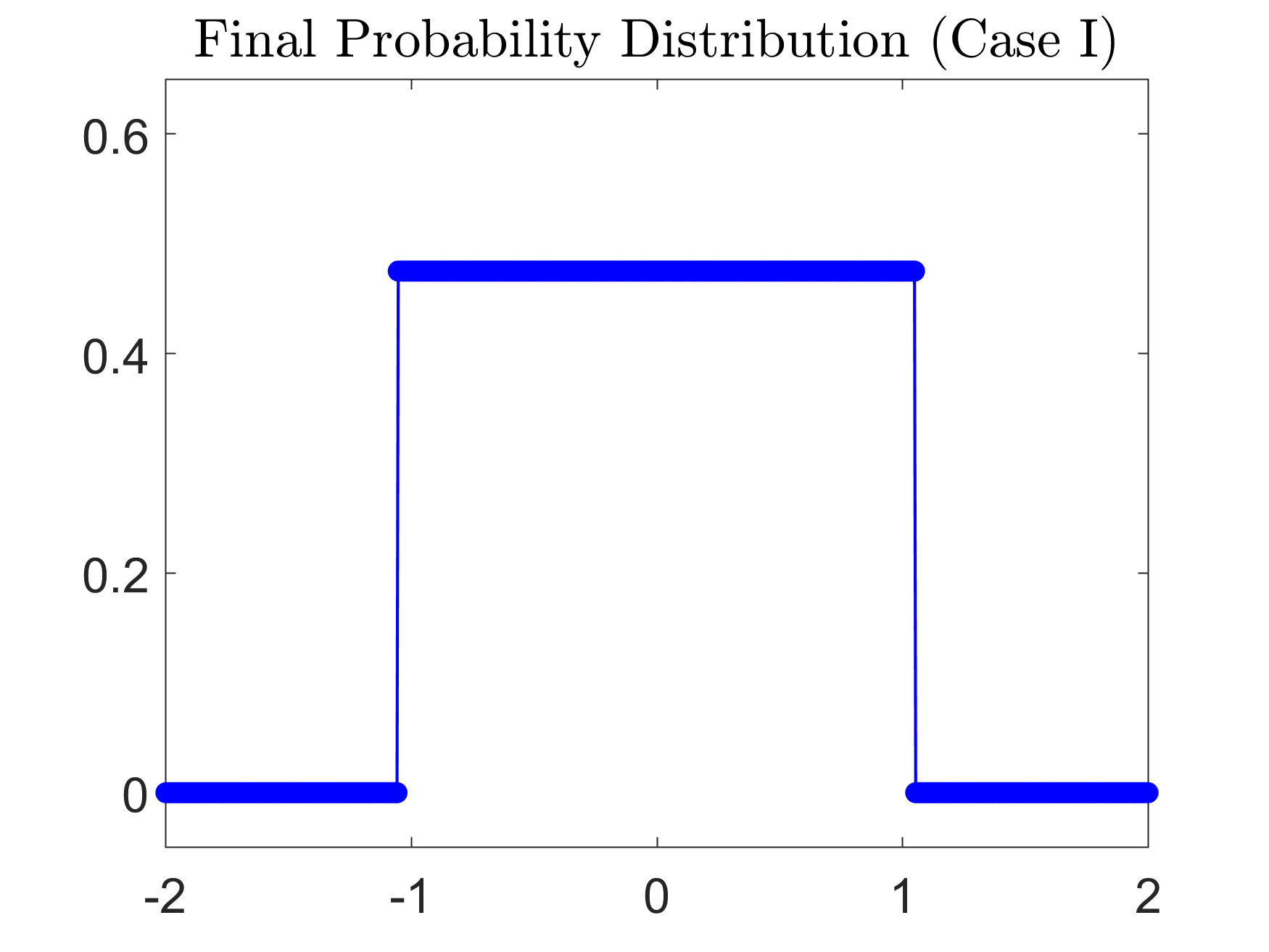}\hspace*{1cm}
            \includegraphics[trim=0.9cm 0.4cm 1.2cm 0.2cm, clip, width=6.cm]{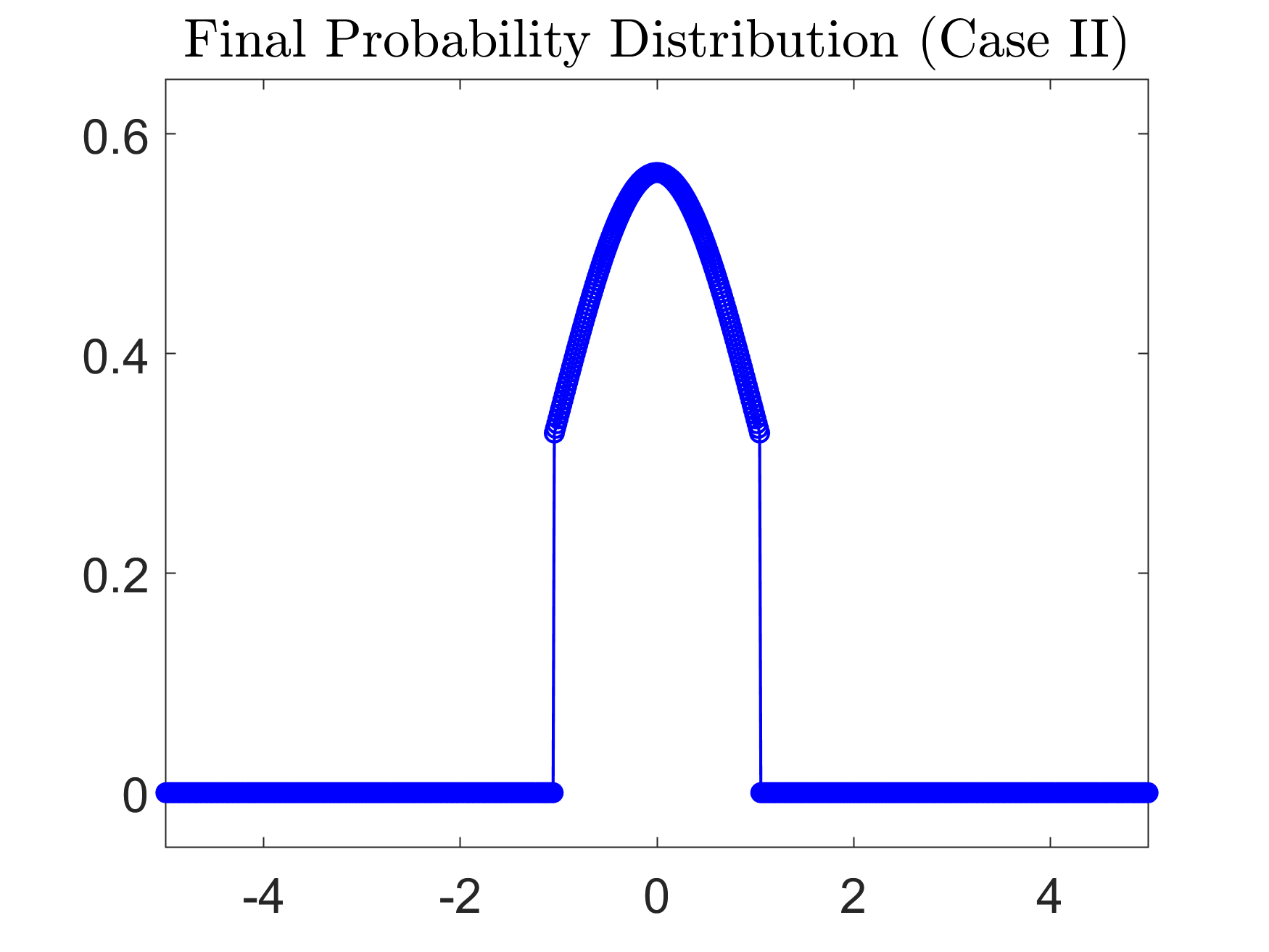}}
\caption{\sf Example 9: Final probability distributions of Case I (left) and II (right).\label{fig25}}
\end{figure}

\subsection{Second-Order Discretization}
We next employ a second-order semi-discrete LLF finite-volume scheme with piecewise linear reconstruction and SSP-RK3 time integration; see Appendix \ref{appb}. The resulting stability domains coincide with those obtained by the first-order LLF scheme.

\paragraph*{Example 10---Burgers Equation.}
In this example, we take the same setting as in Example 4, but compute the numerical results using the second-order LLF scheme. The obtained numerical results are plotted in Figure \ref{fig19}, where one can see that, as in the previous example, the computed stability domains coincide with those in Example 4. 
\begin{figure}[ht!]
\centerline{\includegraphics[trim=0.9cm 0.4cm 1.2cm 0.2cm, clip, width=6.cm]{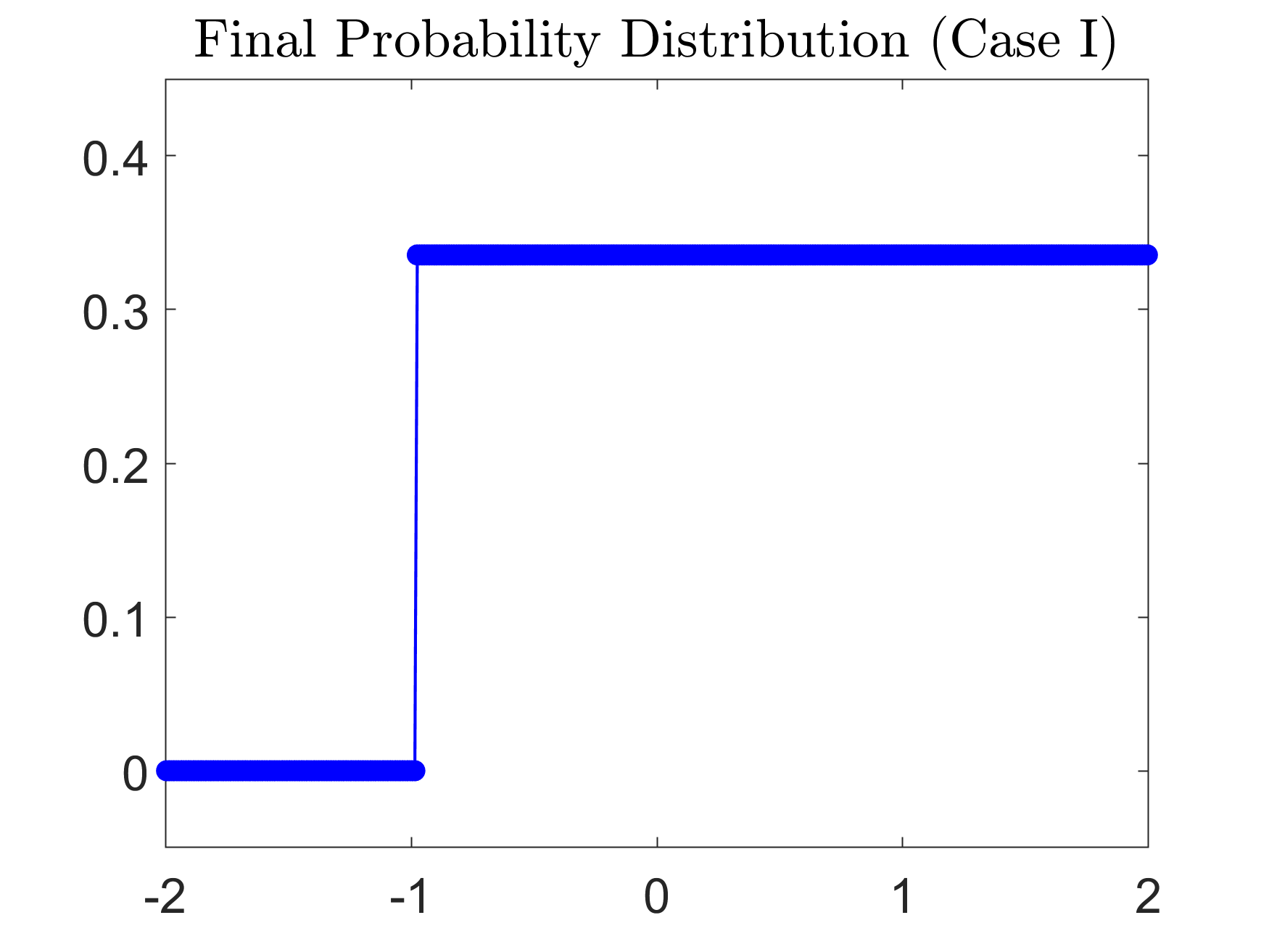}\hspace*{1cm}
            \includegraphics[trim=0.9cm 0.4cm 1.2cm 0.2cm, clip, width=6.cm]{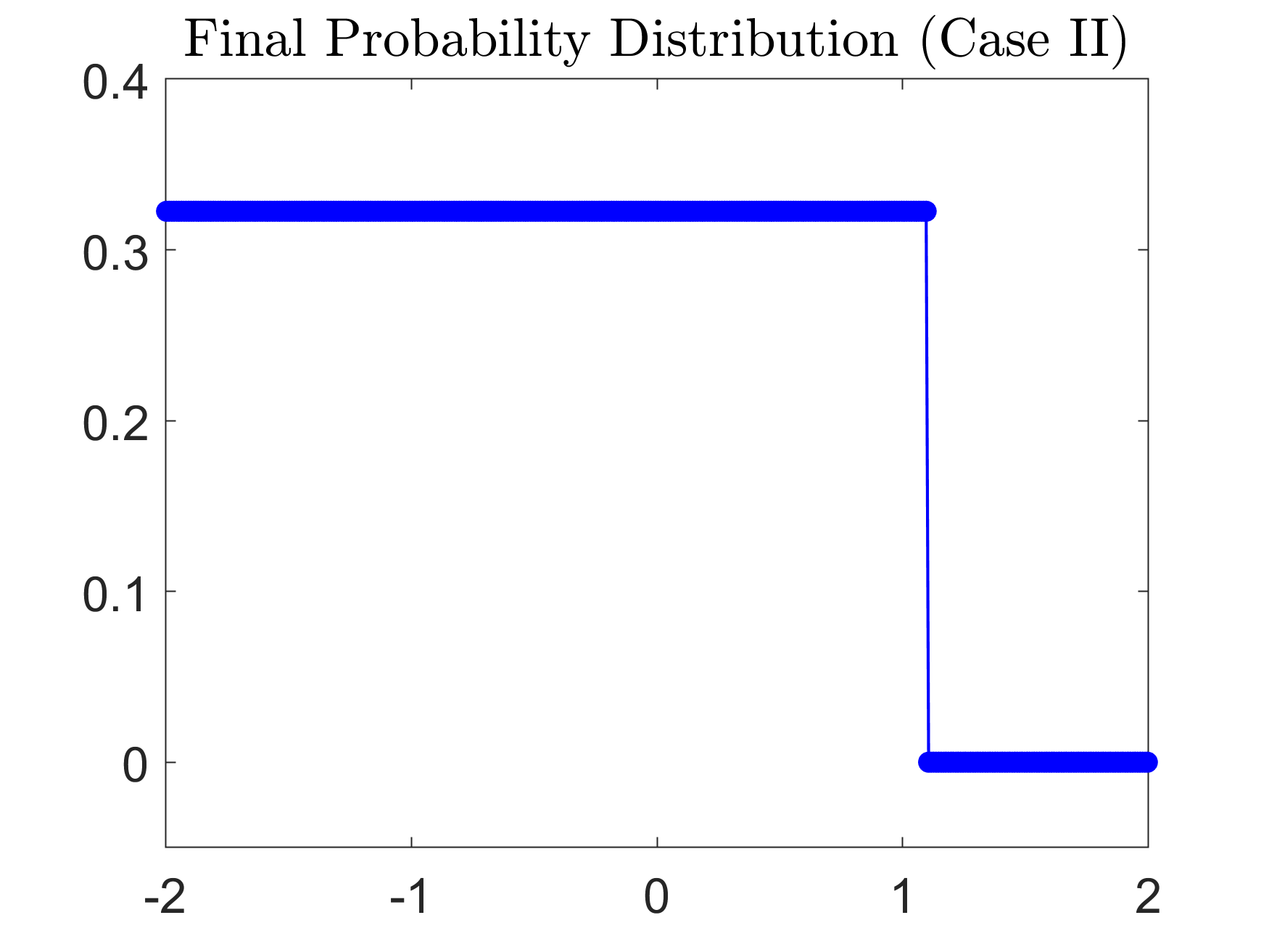}}
\caption{\sf Example 10: Final probability distributions of Case I (left) and II (right).\label{fig19}}
\end{figure}

We also plot the numerical results for Cases I and II over time in Figure \ref{fig20}, which validate the stable domains of the two cases for the studied system. 
\begin{figure}[ht!]
\centerline{\includegraphics[trim=0.1cm 0.2cm 0.8cm 0.2cm, clip, width=7.cm]{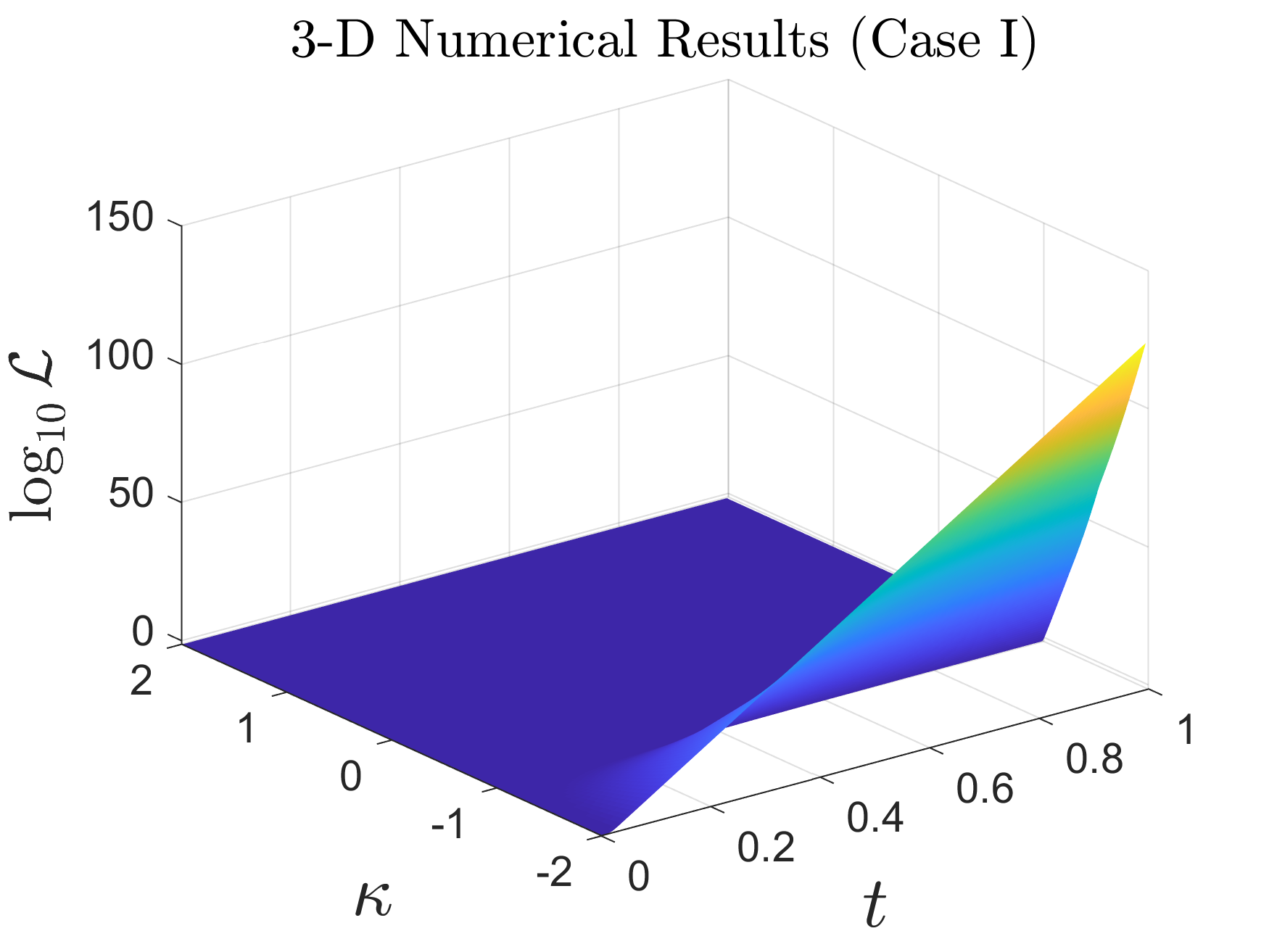}\hspace*{1cm}
            \includegraphics[trim=0.1cm 0.2cm 0.8cm 0.2cm, clip, width=7.cm]{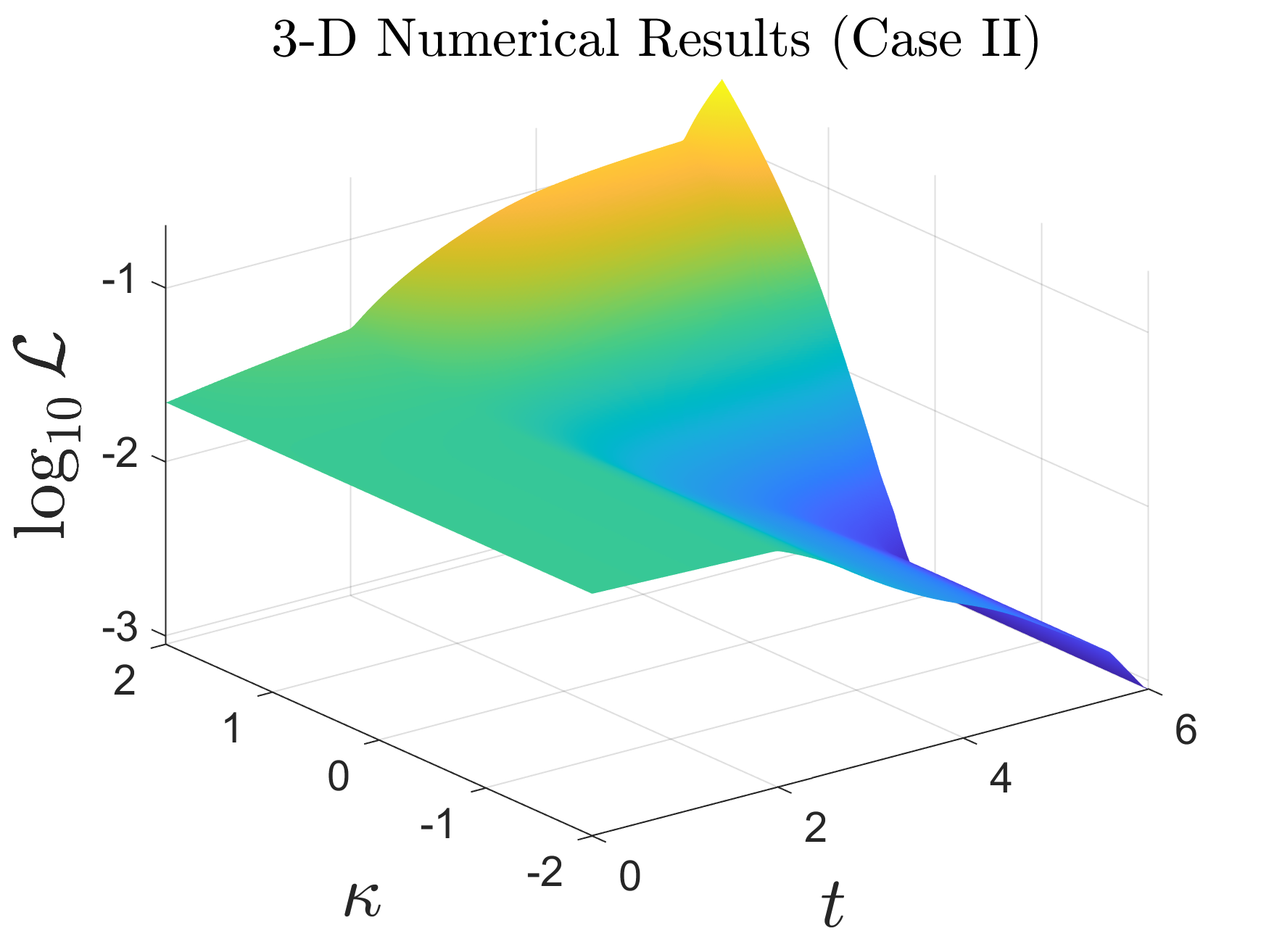}}
\caption{\sf Example 10: Numerical results of Cases I (left) and II (right).\label{fig20}}
\end{figure}

\subsection{Feedback Control of Two Parameters}
In this section, we consider a more complex system---heat transfer in laser powder bed fusion (LPBF)---to show that the proposed method can give the stable domain of complex systems.

\paragraph*{Example 11—Temperature Field Development in Laser Powder Bed Fusion with Feedback Power Regulation.} In this last example, we consider the evolution of the thermal field during a single-track scanning process in laser powder bed fusion (LPBF). LPBF is an additive manufacturing process in which a laser beam heats the powder-bed substrate and locally fuses the powder; further details of the process can be found in \cite{CHOWDHURY20222109}. Figure \ref{fig:LPBF} illustrates the single-track LPBF process considered in this study. One can see that the process is carried out on a cuboidal workpiece with dimensions
$$
(l_x,l_y,l_z)=(1.2\,\mathrm{mm},\,0.45\,\mathrm{mm},\,0.35\,\mathrm{mm}).
$$
During scanning, the laser beam center moves at a constant velocity $v=1\,\mathrm{m/s}$, with $(x_L(t),y_L(t))=(vt,0)$. The laser has a spot radius $R_L=25\,\mu\mathrm{m}$ and a Gaussian intensity distribution, with its power regulated by a feedback control law.
\begin{figure}
    \centering
    \includegraphics[width=0.6\linewidth]{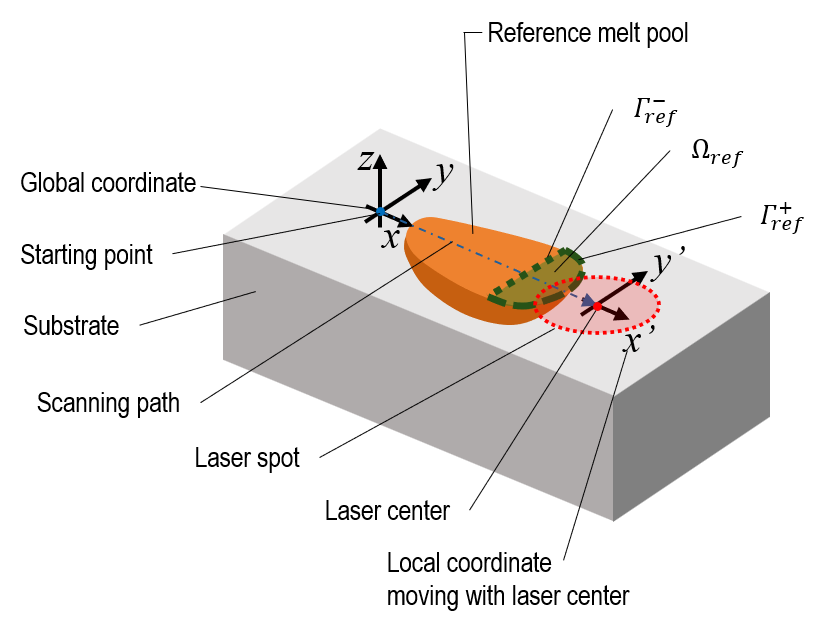}
    \caption{\sf Example 11: LPBF process and defined domains at the reference state for power control and Lyapunov function definition.}
    \label{fig:LPBF}
\end{figure}

For moderate laser energy input, heat transfer in the process is predominantly governed by thermal conduction. The temperature field therefore satisfies the heat conduction equation
\begin{equation}\label{3.12}
T_t(x,y,z,t)=\alpha\bigl(T_{xx}(x,y,z,t)+T_{yy}(x,y,z,t)+T_{zz}(x,y,z,t)\bigr)+P(t)\,f_L\bigl(x-x_L(t),y-y_L(t),z\bigr),
\end{equation}
where $\alpha$ denotes the thermal diffusivity, $P(t)$ is the laser power input, and $f_L$ represents the volumetric laser power distribution relative to the moving laser position. The laser power $P(t)$ is determined by the feedback control law
\begin{equation}\label{3.13}
P(t)=P_{\mathrm{ref}}+\kappa\,\alpha_0(t)\sqrt{\frac{B(t)}{A}},
\end{equation}
where $P_{\mathrm{ref}}=150\,\mathrm{W}$ is the reference laser power and $\kappa$ is the feedback gain. The quantities $\alpha_0(t)$, $B(t)$, and $A$ are defined in terms of the relative coordinates $(x',y')=(x-x_L(t),y-y_L(t))$ by
\begin{equation*}
\alpha_0(t)=\frac12\,\mathrm{sgn}\!\left(\int_{(x',y')\in\Gamma^+_{\mathrm{ref}}} \Bigl(T(vt_{\mathrm{ref}}+x',y',0,t_{\mathrm{ref}})-T(vt+x',y',0,t)\Bigr)\,dS\right),
\end{equation*}
\begin{equation*}
\begin{aligned}
B(t)&=\int_{(x',y')\in\Gamma^+_{\mathrm{ref}}}\Bigl(T_x(vt_{\mathrm{ref}}+x',y',0,t_{\mathrm{ref}})-T_x(vt+x',y',0,t)\Bigr)^2\\
    &+\Bigl(T_y(vt_{\mathrm{ref}}+x',y',0,t_{\mathrm{ref}})-T_y(vt+x',y',0,t)\Bigr)^2\,dS,
\end{aligned}
\end{equation*}
and
\begin{equation*}
A=\int_{\Gamma^-_{\mathrm{ref}}} v\,dS.
\end{equation*}
The subscript ``ref'' denotes quantities evaluated at the reference state, which corresponds to a stabilized thermal field in the relative coordinate system $(x',y')$. This reference state is generated by applying the constant power input $P(t)\equiv P_{\mathrm{ref}}$ for a sufficiently long time after the onset of scanning; in this study, the
reference time is taken as $t_{\mathrm{ref}}=800\,\mu\mathrm{s}$.

The domains $\Gamma^+_{\mathrm{ref}}$, $\Gamma^-_{\mathrm{ref}}$, and $\Omega_{\mathrm{ref}}$ are illustrated in Figure~\ref{fig:LPBF}. The melting interface on the top surface, $\Gamma^+_{\mathrm{ref}}$, is defined as
\begin{equation*}
\Gamma^+_{\mathrm{ref}}=\Bigl\{(x',y')\ \big|\ T(vt_{\mathrm{ref}}+x',y',0,t_{\mathrm{ref}})=T_m\ \text{and}\ T_t(x',y',0,t_{\mathrm{ref}})\ge 0\Bigr\},
\end{equation*}
where $T_m=1723\,\mathrm{K}$ is the melting temperature. The line segment $\Gamma^-_{\mathrm{ref}}$ separates the solid--liquid interface $T=T_m$ into melting and solidification regions and is defined as
\begin{equation*}
\begin{aligned}
\Gamma^-_{\mathrm{ref}} =\Bigl\{ & (x',y')\ \Big|\ x'=\min\bigl(x''\in(x'',y'')\in\Gamma^+_{\mathrm{ref}}\bigr),\\
 & \min\bigl(y''\in(x'',y'')\in\Gamma^+_{\mathrm{ref}}\bigr)\le y'\le\max\bigl(y''\in(x'',y'')\in\Gamma^+_{\mathrm{ref}}\bigr)\Bigr\}.
\end{aligned}
\end{equation*}
The domain $\Omega_{\mathrm{ref}}$ is the region on the top surface enclosed by $\Gamma^+_{\mathrm{ref}}\cup\Gamma^-_{\mathrm{ref}}$. The volumetric laser heat source is given by
\begin{equation*}
f_L\bigl(x-x_L(t),y-y_L(t),z\bigr)=\frac{2 {\cal A}\,d_{\mathrm{opt}}}{\pi R_L^2}\exp\!\left(-\frac{2\bigl((x-vt)^2+y^2\bigr)}{R_L^2}+\frac{z}{d_{\mathrm{opt}}}\right),
\end{equation*}
where ${\cal A}=0.27$ is the laser absorptivity and $d_{\mathrm{opt}}=34\,\mu\mathrm{m}$ is the optical penetration depth.

The initial condition is set as a uniform temperature distribution $T(x,y,z,0)=300\,\mathrm{K}$, subject to the specified following boundary conditions. The bottom face is maintained at a constant temperature $T(x,y,-l_z,t)=300\,\mathrm{K}$, the four lateral faces are assumed to be adiabatic $T_x(x,0,z,t)=T_x(x,l_x,z,t)=T_y(x,\pm 0.5l_y,z,t)=0$, and the top surface is subject to convective and radiative heat losses to the surrounding atmosphere 
$$
k\,T_z(x,y,0,t)=-\Bigl(h\bigl(T(x,y,0,t)-T_\infty\bigr)+\varepsilon\sigma\bigl(T(x,y,0,t)^4-T_\infty^4\bigr)\Bigr),
$$
where $k=20.9\,\mathrm{W\,m^{-1}K^{-1}}$ is the thermal conductivity, $h=20\,\mathrm{W\,m^{-2}K^{-1}}$ is the convective heat transfer coefficient, $T_\infty=300\,\mathrm{K}$ is the ambient temperature, $\varepsilon=0.36$ is the emissivity, and $\sigma$ is the Stefan--Boltzmann constant.

The resulting initial--boundary value problem is solved numerically using a finite-difference method on a uniform mesh with $(N_x,N_y,N_z)=(128,48,32)$. Time integration is performed using a Runge--Kutta scheme with a time step $\Delta t=40\,\mathrm{ns}$. The objective of the feedback power regulation is to control the Lyapunov-type indicator
\begin{equation*}
\begin{aligned}
L(t)=\int_{(x',y')\in\Omega_{\mathrm{ref}}} & \Bigl(T_x(vt_{\mathrm{ref}}+x',y',z,t_{\mathrm{ref}})-T_x(vt+x',y',z,t)\Bigr)^2 \\ +&\Bigl(T_y(vt_{\mathrm{ref}}+x',y',z,t_{\mathrm{ref}})-T_y(vt+x',y',z,t)\Bigr)^2\,dx'\,dy',
\end{aligned}
\end{equation*}
which quantifies the discrepancy between the temperature gradients in the $x$- and $y$-directions over a moving domain that maintains a fixed relative position with respect to the laser spot. We then apply the Bayesian update procedure of \S2.2 (Algorithm~1), using $L(t)$ as the observed indicator, to update the discrete probability mass function over the chosen parameter grid.

In this example, we consider a two-parameter setting in which the Bayesian damping factor $\alpha$  in \eref{3.12} and  the feedback gain $\kappa$ in \eref{3.13} are treated as the control parameters; thus, the Bayesian procedure updates the joint distribution of $(\kappa, \alpha)$, as illustrated in Figure~\ref{fig26}.

\begin{figure}[ht!]
\centerline{\includegraphics[trim=0.7cm 0.2cm 0.9cm 0.1cm, clip, width=6.cm]{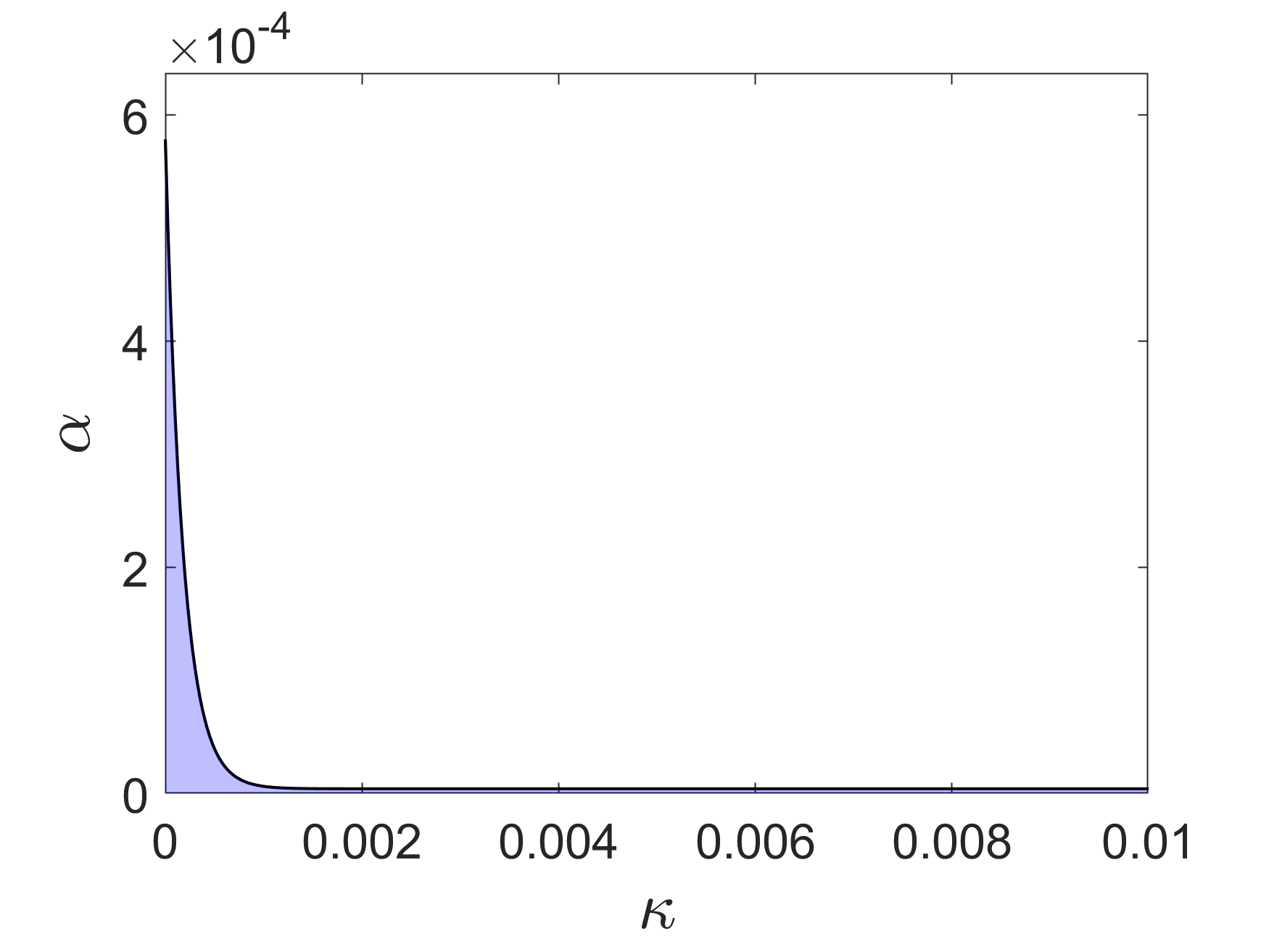}\hspace*{1cm}}
\caption{\sf Example 11: Prior and posterior (purple color) probability distributions of the feedback parameters.\label{fig26}}
\end{figure}

To further quantify the decay behaviour of the indicator, we fit the discrete Lyapunov functions from the subregion $[0, 1e-4]\times[0, 3.3e-4]$ by a power-law model $L(t)\approx C\,t^{-a}$ by a nonlinear least-squares procedure, and record the fitted parameters $(a,C)$. The corresponding top-view plots of the fitted values are shown in Figure~\ref{fig27}.
\begin{figure}[H]
\centerline{\includegraphics[trim=0.4cm 0.1cm 0.7cm 0.1cm, clip, width=6cm]{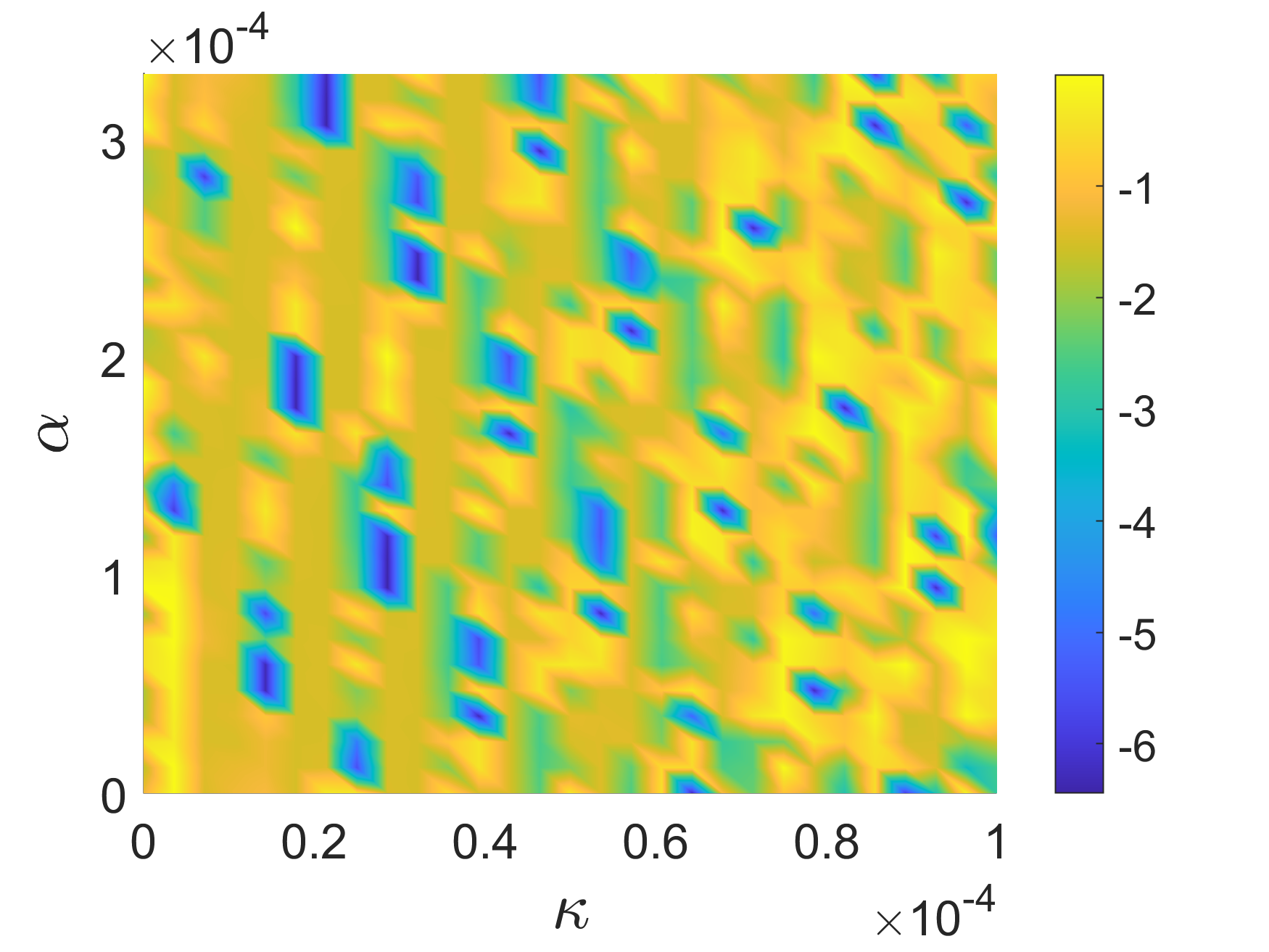}\hspace*{0.5cm}
            \includegraphics[trim=0.4cm 0.1cm 0.7cm 0.1cm, clip, width=6cm]{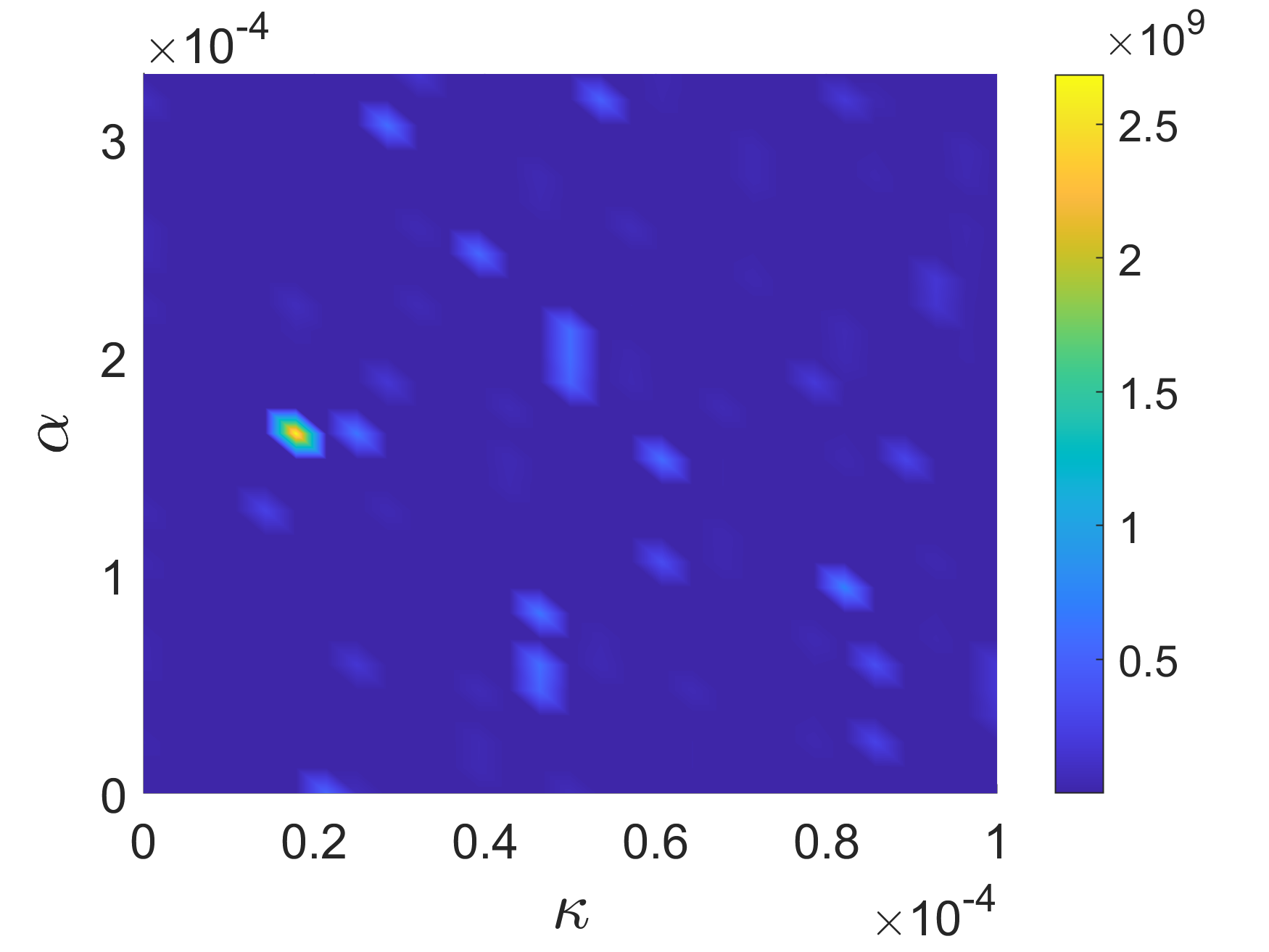}}
\caption{\sf Example 11: Least-squares power-law fit parameters: fitted decay exponent $a$ (left) and fitted prefactor $C$ (right).\label{fig27}}
\end{figure}

To assess the fit quality, we compute the standard deviation
\begin{equation*}
e_{\mathrm{deviation}}=\sqrt{\frac{1}{N}\sum_{i=1}^{N}\bigl(L_i-\xbar L\bigr)^2},
\end{equation*}
where $L_i=L(t_i)$ and $\xbar L=\frac{1}{N}\sum^{i=N}_{i=1}  L_i$. The values of $e_{\mathrm{deviation}}$ and $\log_{10}\bigl(e_{\mathrm{deviation}}\bigr)$ are plotted in Figure~\ref{fig28}.
\begin{figure}[H]
   \centerline{\includegraphics[trim=0.4cm 0.1cm 0.7cm 0.1cm, clip, width=6cm]{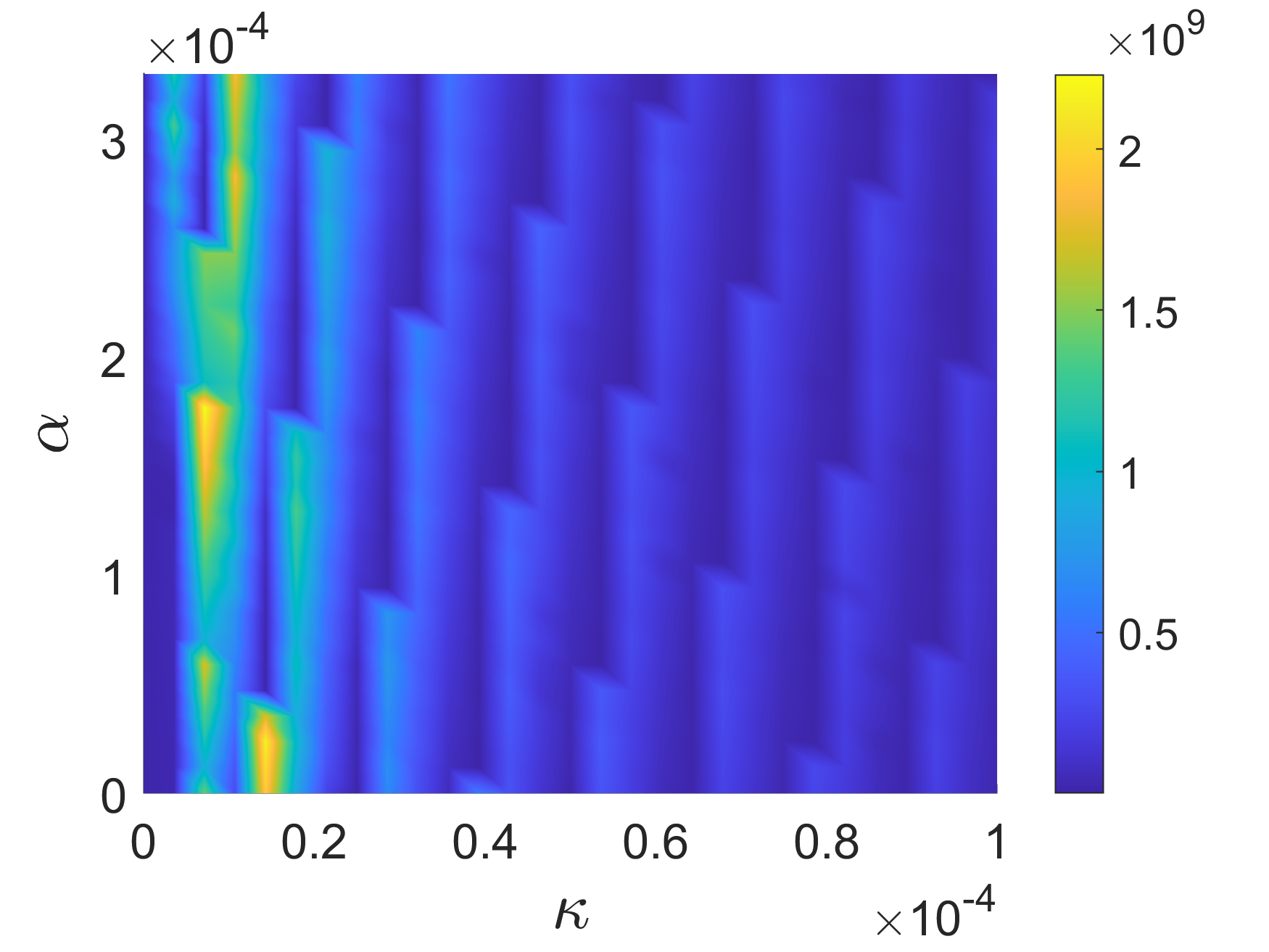}\hspace*{1cm}
               \includegraphics[trim=0.4cm 0.1cm 0.7cm 0.1cm, clip, width=6.cm]{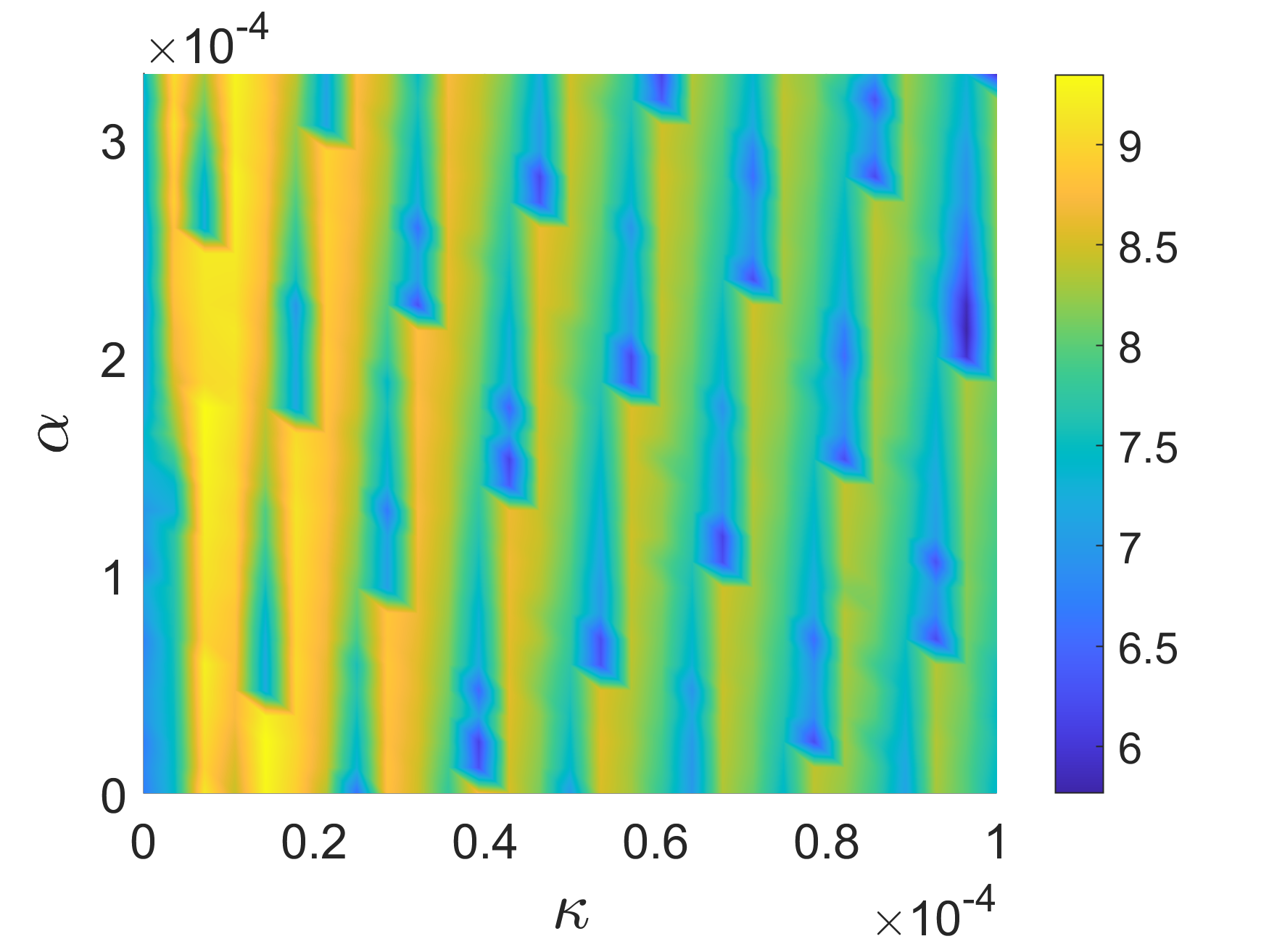}}
   \caption{\sf Example 11: Fit-quality diagnostics for the power-law model: the standard deviation $e_{\mathrm{deviation}}$ (left) and $\log_{10}(e_{\mathrm{deviation}})$ (right).\label{fig28}}
\end{figure}

Finally, Figure~\ref{fig29} shows thresholded regions (black) where $e_{\mathrm{deviation}}$ falls below prescribed tolerances, providing an additional diagnostic for identifying parameter regimes that exhibit a consistent decay trend.
\begin{figure}[H]
   \centerline{\includegraphics[trim=0.6cm 0.1cm 1.1cm 0.1cm, clip, width=6.cm]{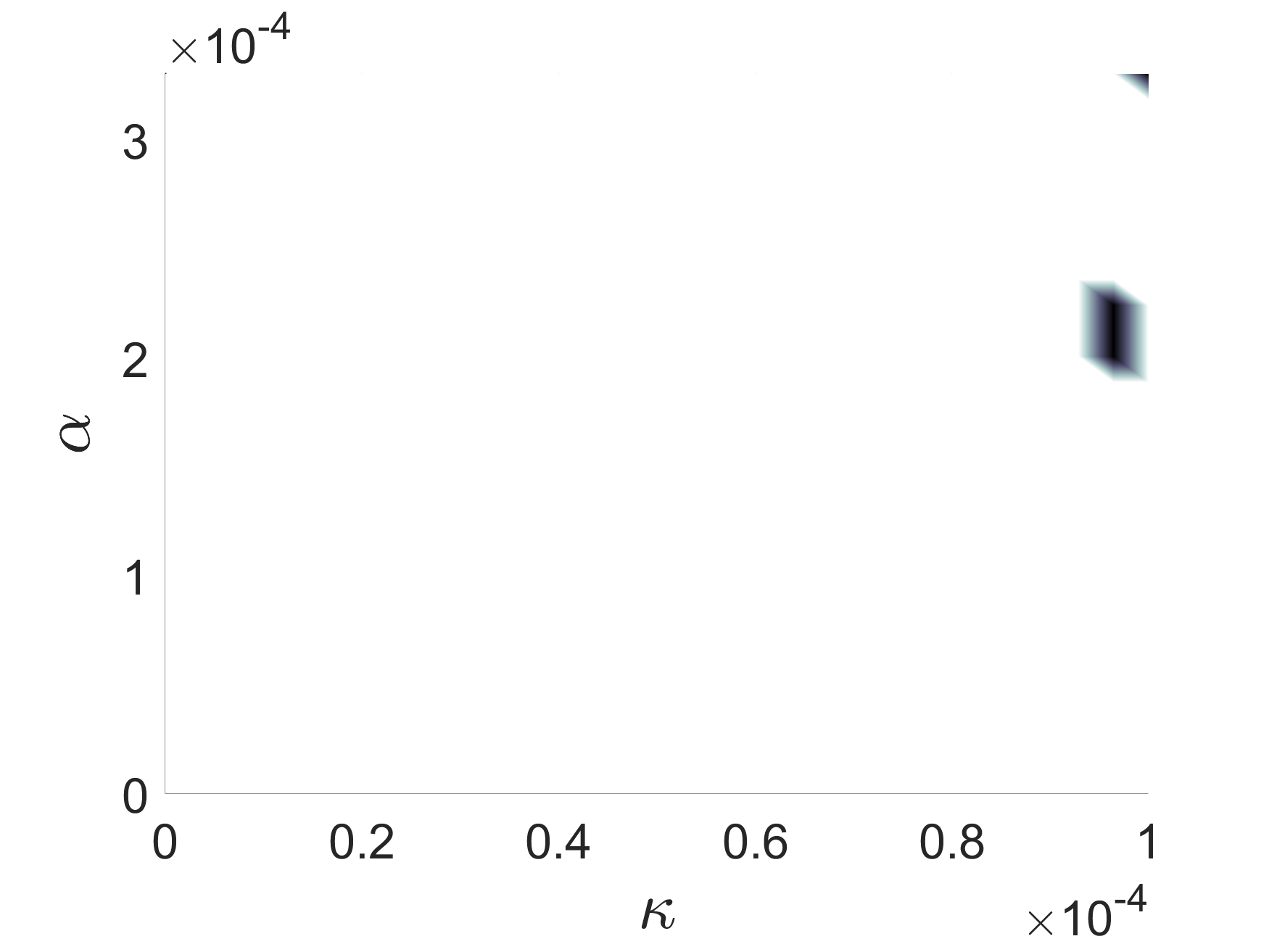}\hspace*{1cm}
               \includegraphics[trim=0.6cm 0.1cm 1.1cm 0.1cm, clip, width=6.cm]{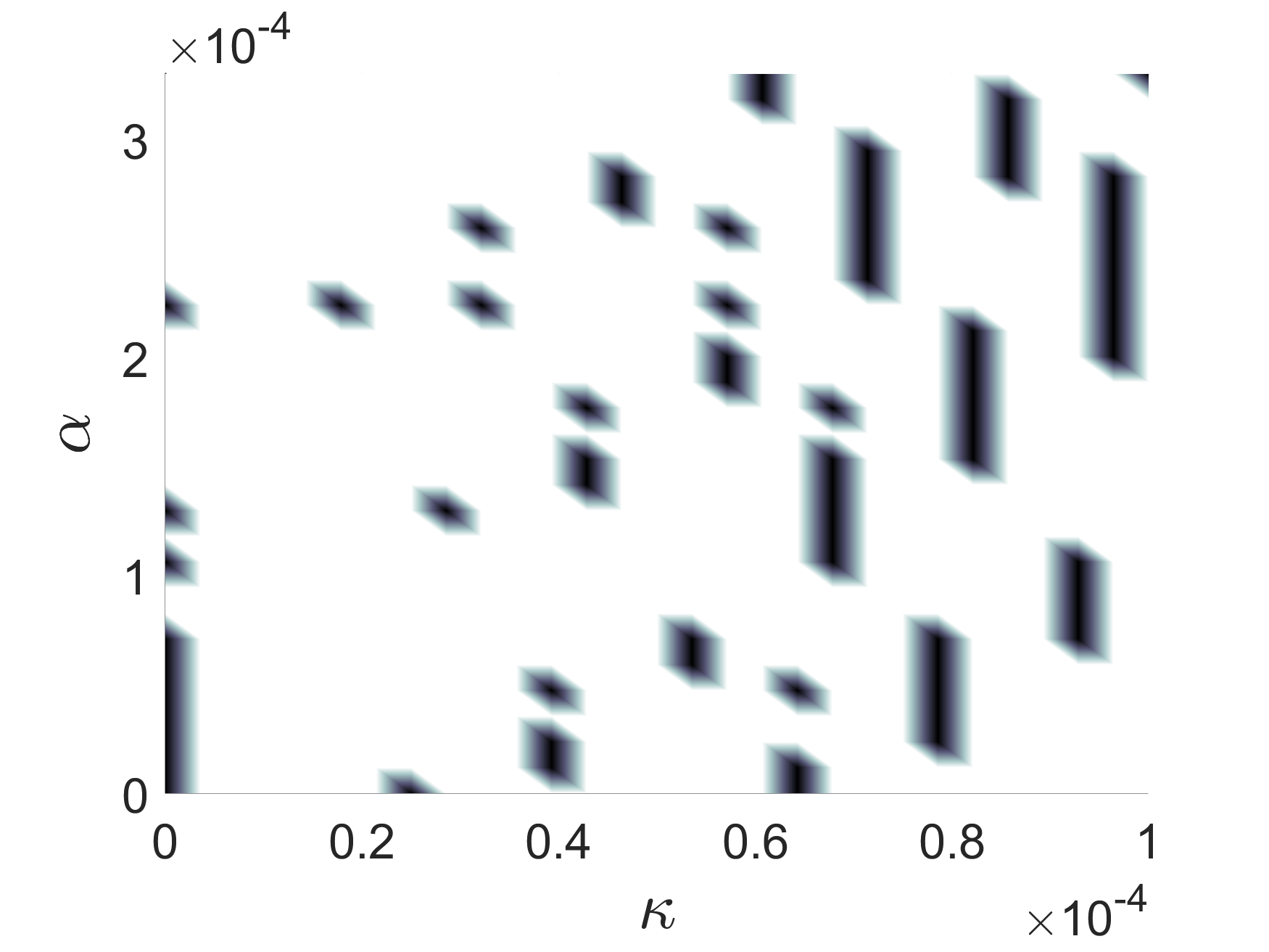}}
   \vskip 12pt 
   \centerline{\includegraphics[trim=0.6cm 0.1cm 1.1cm 0.1cm, clip, width=6.cm]{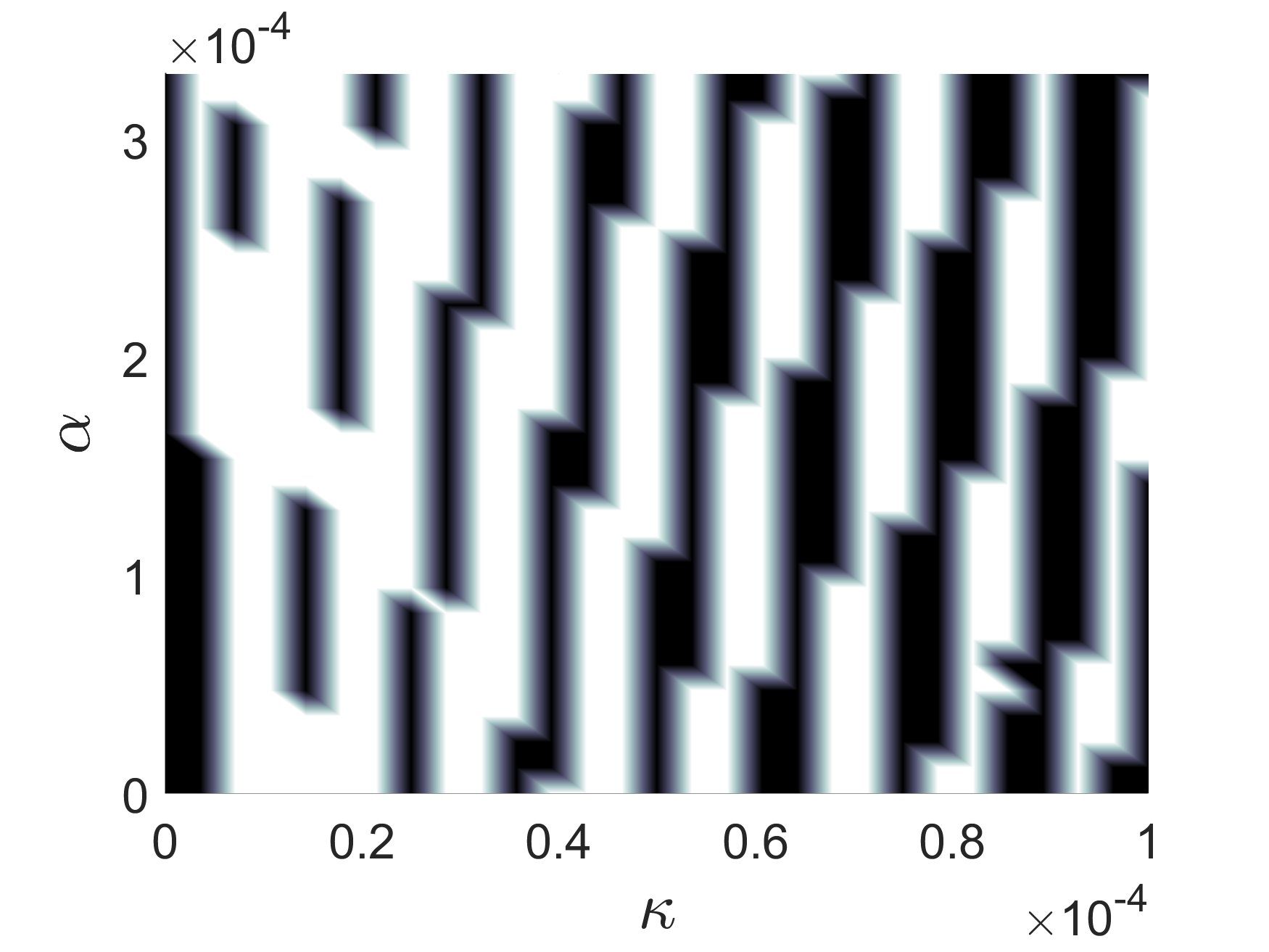}\hspace*{1cm}
               \includegraphics[trim=0.6cm 0.1cm 1.1cm 0.1cm, clip, width=6.cm]{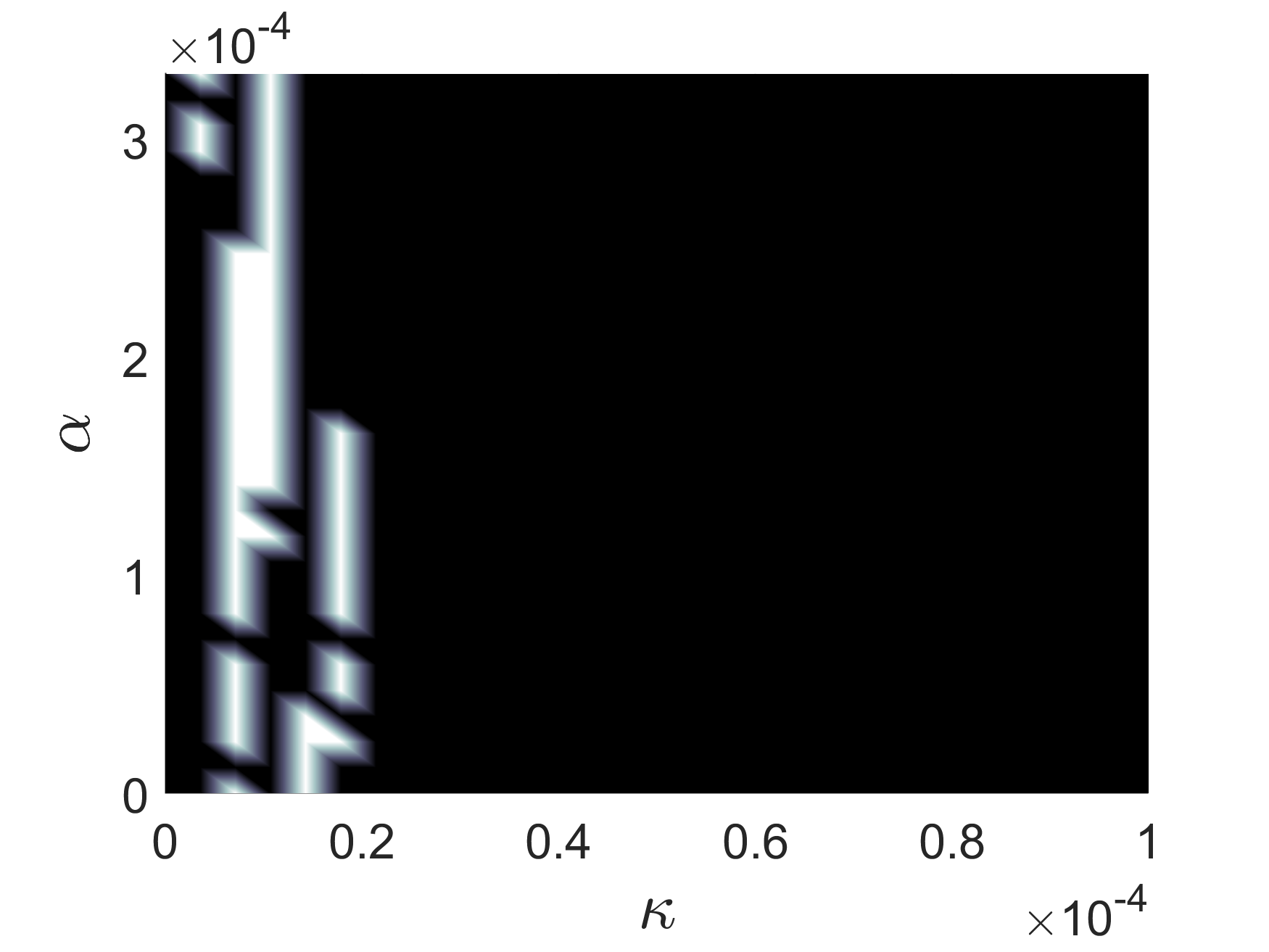}}
   \caption{\sf Example 11: Thresholded regions (black) where the power-law fit error satisfies $e_{\mathrm{deviation}}<10^{-6}$ (top left), $e_{\mathrm{deviation}}<10^{-7}$ (top right), $e_{\mathrm{deviation}}<10^{-8}$ (bottom left), and $e_{\mathrm{deviation}}<10^{-9}$ (bottom right).\label{fig29}}
\end{figure}

These results indicate that the proposed algorithm remains effective in a two-parameter system, with $\kappa$ and $\alpha$ treated as the unknown parameters to be identified via the Bayesian update.

\section{Conclusions}\label{sec4}
In this paper, we have introduced a Bayesian feedback framework for identifying stabilizing boundary controls in hyperbolic balance laws. The core idea is to iteratively update a prior over control parameters using solution-driven diagnostics (Lyapunov functionals or discrete energies), coupled with a damping and normalization step that concentrates probability on stabilizing regimes. This yields an efficient, nonintrusive post-processing layer that can be placed atop standard finite-volume solvers. On linear benchmarks (advection/wave equations and the linearized Saint-Venant system), the method reproduces analytically known stability domains, including mixed boundary couplings. On nonlinear and stochastic tests, including the nonlinear Saint-Venant system, one- and two-dimensional Burgers equations, and Burgers equations with random inputs, the method remains robust and reveals the dependence of stabilizing regions on the initial data and boundary feedback configuration. In particular, the 2-D Burgers example demonstrates that
the Bayesian update can be applied to a multidimensional balance law with feedback boundary conditions imposed in both spatial directions. Extensions to a second-order LLF discretization and the evolution of the thermal field during a single-track scanning process in laser powder bed fusion indicate that the approach is scheme-independent and applicable beyond purely conservative models. Future work will focus on a rigorous analysis of the probability iteration, including posterior concentration on the stabilizing set and its relation to decay rates of the indicator, and on extensions to multi-parameter and constrained boundary controls, as well as generalizations to networks and multidimensional settings.

\appendix
\setcounter{equation}{0}
\section{First-Order Local Lax-Friedrichs Scheme for Stochastic Hyperbolic Systems}\label{appa}

We consider the first-order LLF finite-volume scheme for the stochastic system. Assume the computational domain is covered with uniform cells $C_{j,k}:=[x_\jmh,x_\jph]\times [\xi_\kmh, \xi_\kph]$ with $x_\jph-x_\jmh\equiv\dx$ and $\xi_\kph-\xi_\kmh\equiv\dxi$ centered at $(x_j,\xi_k)=\big((x_\jmh+x_\jph)/2,(\xi_\kmh+\xi_\kph)/2\big)$, $\,j=1,\ldots,N_x$, $k=1,\ldots,N_\xi$ and the cell average values 
\begin{equation*}
  \xbar\mU_{j,k}(t):\approx\frac{1}{\dx \dxi}\int\limits_{C_{j,k}}\mU(x,t)\,{\rm d}x{\rm d}\xi
\end{equation*}
are available at a certain time level $t$.

The computed cell averages $\xbar \mU_{j,k}$ of the 1-D system \eref{1.1} are evolved in time by the first-order LLF scheme:
\begin{equation*}
  \xbar \mU^{n+1}_{j,k}=  \xbar \mU^{n}_{j,k} - \dfrac{\dt}{\dx}\Big( \bm{{\cal F}}_{\jph,k} - \bm{{\cal F}}_{\jmh,k} \Big)+\dt\,\mS(\xbar \mU^n_{j,k}),
\end{equation*}
where $\bm{{\cal F}}_{\jph,k}$ stands for the finite-volume numerical flux
\begin{equation*}
  \bm{{\cal F}}_{\jph,k}=\frac{1}{2} \big( \mF (\xbar \mU_{j,k}) +\mF(\xbar \mU_{j+1,k}) \big)-\frac{\alpha_{\jph,k}}{2}\big(\xbar \mU_{j+1,k}-\xbar \mU_{j,k} \big), 
\end{equation*}
with the local speeds of propagation estimated using the eigenvalues $\lambda_1({\cal A}) \le \ldots \le \lambda_d({\cal A})$ of the matrix ${\cal A}=\frac{\partial \mF}{\partial \mU}$:
\begin{equation*}
\alpha_{\jph,k} = \max \big\{|\lambda_1 (\xbar \mU_{j,k})|, |\lambda_1 (\xbar \mU_{j+1,k})|, \ldots, |\lambda_d (\xbar \mU_{j,k})|, |\lambda_d (\xbar \mU_{j+1,k})|      \big\}. 
\end{equation*}
The time step $\dt$ is computed by the CFL condition, $\dt = {\rm CFL} \frac{\dx}{ \max\limits_{j,k} \{\alpha_{\jph,k}\} }$.

\section{Second-Order Local Lax-Friedrichs Scheme}\label{appb}
In the second-order LLF scheme, the cell averages $\xbar \mU_j$ of the 1-D system \eref{1.1} are evolved in time by solving the following system of ODEs:
\begin{equation}\label{B1a}
  \frac{{\rm d}\xbar \mU_j  }{ {\rm d} t} = -\dfrac{ \bm{{\cal F}}_\jph - \bm{{\cal F}}_\jmh }{\dx}+\mS(\xbar \mU_j),
\end{equation}
where $\bm{{\cal F}}_\jph$ stand for the second-order finite-volume numerical fluxes 
\begin{equation*}
  \bm{{\cal F}}_\jph=\frac{1}{2} \big( \mF (\mU^-_\jph) +\mF (\mU^+_\jph)  \big)-\frac{\alpha_\jph}{2}\big(\mU^+_\jph-\mU^-_\jph \big), 
\end{equation*}
with the local speeds of propagation estimated using the eigenvalues $\lambda_1({\cal A}) \le \ldots \le \lambda_d({\cal A})$ of the matrix ${\cal A}=\frac{\partial \mF}{\partial \mU}$:
\begin{equation*}
\alpha_\jph = \max \big\{|\lambda_1 (\mU^-_\jph)|, |\lambda_1 (\mU^+_\jph)|, \ldots, |\lambda_d (\mU^-_\jph)|, |\lambda_d (\mU^+_\jph)|      \big\}. 
\end{equation*}
Here, $\mU^\pm_\jph$ are the one-sided point values of $\mU$ at the cell interface $x=x_\jph$, estimated using a piecewise linear interpolant
\begin{equation}\label{B1}
\widetilde\mU(x)=\,\xbar\mU_j+(\mU_x)_j(x-x_j),\quad x\in C_j,
\end{equation}
which leads to
\begin{equation}\label{B2}
\mU^-_\jph=\,\xbar\mU_j+\frac{\dx}{2}(\mU_x)_j,\quad\mU^+_\jph=\,\xbar\mU_{j+1}-\frac{\dx}{2}(\mU_x)_{j+1}.
\end{equation}
In order to ensure the reconstruction \eref{B1}--\eref{B2} is non-oscillatory, one needs to compute the slopes $(\mU_x)_j$ in \eref{B1} with the help of a 
nonlinear limiter. In all of the numerical experiments reported in \S\ref{sec3}, we have used a generalized minmod limiter
\cite{lie03,Nessyahu90,Sweby84}:
\begin{equation}
(\mU_x)_j={\rm minmod}\left(\theta\frac{\,\xbar\mU_j-\,\xbar\mU_{j-1}}{\dx},\,\frac{\,\xbar\mU_{j+1}-\,\xbar\mU_{j-1}}{2\dx},\,
\theta\frac{\,\xbar\mU_{j+1}-\,\xbar\mU_j}{\dx}\right),\quad\theta\in[1,2],
\label{B3}
\end{equation}
applied in a component-wise manner. Here, the minmod function is defined as
\begin{equation*}
{\rm minmod}(z_1,z_2,\ldots):=\begin{cases}
\min_j\{z_j\}&\mbox{if}~z_j>0\quad\forall\,j,\\
\max_j\{z_j\}&\mbox{if}~z_j<0\quad\forall\,j,\\
0            &\mbox{otherwise}.
\end{cases}
\end{equation*}
The parameter $\theta$ in \eref{B3} is used to control the amount of numerical viscosity present in the resulting scheme, and larger values of $\theta$ correspond to sharper but, in general, more oscillatory reconstructions. In this paper, we use $\theta= 1.3$.

Finally, we numerically integrate the ODE systems \eref{B1a} by the three-stage third-order strong stability preserving Runge-Kutta (SSP RK3) method (see, e.g., \cite{Gottlieb11,Gottlieb12}), and the time step $\dt$ is computed by the CFL condition, $\dt={\rm CFL} \frac{\dx}{\max\limits_j\{\alpha_\jph\}}$.

\begin{DA}
\paragraph{Funding.} The work of S. Chu and  M. Herty was funded by the Deutsche Forschungsgemeinschaft (DFG, German Research Foundation)-SPP 2410 Hyperbolic Balance Laws in Fluid Mechanics: Complexity, Scales, Randomness (CoScaRa) within the Project(s) HE5386/26-1 (Numerische Verfahren für gekoppelte Mehrskalenprobleme, 525842915) and (Zufällige kompressible Euler Gleichungen: Numerik und ihre Analysis, 525853336) HE5386/27-1, and  the Deutsche Forschungsgemeinschaft (DFG, German Research Foundation) - SPP 2183: Eigenschaftsgeregelte Umformprozesse with the Project(s) HE5386/19-2,19-3 Entwicklung eines flexiblen isothermen Reckschmiedeprozesses für die eigenschaftsgeregelte Herstellung von Turbinenschaufeln aus Hochtemperaturwerkstoffen (424334423). Support through HE5386/33-1 Control of Interacting Particle Systems, and Their Mean-Field, and Fluid-Dynamic Limits (560288187) and HE5386/34-1 Partikelmethoden für unendlich dimensionale Optimierung ( 561130572) is acknowledged.

\paragraph{Conflicts of interest.} On behalf of all authors, the corresponding author states that there is no conflict of interest.

\paragraph{Data and software availability.} The data that support the findings of this study and FORTRAN/Python codes developed by the authors and
used to obtain all of the presented numerical results are available from the corresponding author upon reasonable request.
\end{DA}

\bibliography{reference_YL}
\bibliographystyle{siam}
\end{document}